\documentclass[12pt,twoside,a4paper]{amsart}
\usepackage{amssymb,amscd,amsxtra,calc,mathrsfs}
\usepackage{amsmath}
\usepackage{bm}
\usepackage[all]{xy}
\usepackage{todonotes}


\allowdisplaybreaks[1]

\title[Coupled stability thresholds]{On the coupled Ding stability and the 
Yau--Tian--Donaldson correspondence for Fano manifolds}
\author{Kento Fujita}
\author{Yoshinori Hashimoto} 
\date{\today}
\address{Department of Mathematics, Graduate School of Science, Osaka University, 
Toyonaka, Osaka 560-0043, Japan}
\email{fujita@math.sci.osaka-u.ac.jp}

\address{Department of Mathematics, 
Osaka Metropolitan University,
3-3-138, Sugimoto, Sumiyoshi-ku, Osaka, 558-8585, Japan.}
\email{yhashimoto@omu.ac.jp}

\newcommand{\pr}{\mathbb{P}}
\newcommand{\PP}{\mathbf{P}}

\newcommand{\Z}{\mathbb{Z}}
\newcommand{\Q}{\mathbb{Q}}
\newcommand{\R}{\mathbb{R}}
\newcommand{\C}{\mathbb{C}}

\newcommand{\T}{\mathbb{T}}

\newcommand{\JJ}{\mathbf{J}}
\newcommand{\DD}{\mathbf{D}}
\newcommand{\bL}{\mathbf{L}}
\newcommand{\A}{\mathbb{A}}
\newcommand{\G}{\mathbb{G}}

\newcommand{\K}{\mathbb{K}}

\newcommand{\Pic}{\operatorname{Pic}}

\newcommand{\Hom}{\operatorname{Hom}}

\newcommand{\Aut}{\operatorname{Aut}}
\newcommand{\Proj}{\operatorname{Proj}}

\newcommand{\id}{\operatorname{id}}

\newcommand{\lct}{\operatorname{lct}}

\newcommand{\Ding}{\operatorname{Ding}}
\newcommand{\DDHH}{\operatorname{DH}}
\newcommand{\ord}{\operatorname{ord}}

\newcommand{\vol}{\operatorname{vol}}
\newcommand{\Image}{\operatorname{Image}}

\newcommand{\Val}{\operatorname{Val}}
\newcommand{\QM}{\operatorname{QM}}

\newcommand{\triv}{\operatorname{triv}}
\newcommand{\mult}{\operatorname{mult}}

\newcommand{\interior}{\operatorname{int}}

\newcommand{\Red}{\operatorname{red}}

\newcommand{\Conv}{\operatorname{Conv}}
\newcommand{\Gr}{\operatorname{Gr}}
\newcommand{\Ric}{\operatorname{Ric}}
\newcommand{\bc}{\operatorname{bc}}
\newcommand{\cp}{\operatorname{cp}}
\newcommand{\dist}{\operatorname{dist}}
\newcommand{\PROD}{\operatorname{prod}}
\newcommand{\sI}{\mathcal{I}}
\newcommand{\sJ}{\mathcal{J}}

\newcommand{\sO}{\mathcal{O}}
\newcommand{\sN}{\mathcal{N}}

\newcommand{\sX}{\mathcal{X}}
\newcommand{\sY}{\mathcal{Y}}
\newcommand{\sL}{\mathcal{L}}

\newcommand{\sM}{\mathcal{M}}
\newcommand{\sF}{\mathcal{F}}
\newcommand{\sG}{\mathcal{G}}

\newcommand{\sW}{\mathcal{W}}
\newcommand{\sZ}{\mathcal{Z}}

\newcommand{\sD}{\mathcal{D}}
\newcommand{\sR}{\mathcal{R}}

\newtheorem{thm}{Theorem}[section]
\newtheorem{lemma}[thm]{Lemma}
\newtheorem{proposition}[thm]{Proposition}
\newtheorem{corollary}[thm]{Corollary}
\newtheorem{claim}[thm]{Claim}

\theoremstyle{definition}
\newtheorem{definition}[thm]{Definition}
\newtheorem{remark}[thm]{Remark}

\newtheorem{example}[thm]{Example}
\newtheorem*{ack}{Acknowledgments}

\begin{document}

\maketitle 

\begin{abstract}
We interpret the coupled Ding semistability and the reduced coupled uniform Ding 
stability of log Fano pairs in the notion of coupled stability thresholds
and reduced coupled stability thresholds. As a corollary, we solve a modified 
version of the conjecture by Hultgren and Witt Nystr\"om for coupled K\"ahler--Einstein metrics on Fano manifolds. 
\end{abstract}

\setcounter{tocdepth}{1}
\tableofcontents

\section{Introduction}\label{section:intro}

Let us consider an $n$-dimensional Fano manifold $X$ over the complex numbers 
$\C$. Take ample $\Q$-divisors $L_1,\dots,L_k$ with $-K_X\sim_\Q L_1+\cdots+L_k$, and K\"ahler metrics $\omega_i\in c_1(L_i)$ for all $1\leq i\leq k$. 
Following Hultgren--Witt Nystr\"om \cite{HWN}, we say that 
$(\omega_i)_{i=1}^k$ are \emph{coupled K\"ahler--Einstein metrics} 
(cKE metrics) if 
\[
\Ric \omega_1=\cdots=\Ric \omega_k=\sum_{i=1}^k \omega_i
\]
holds. Obviously, when $k=1$, this is nothing but the K\"ahler--Einstein 
metric on a Fano manifold $X$. 
Hultgren and Witt Nystr\"om conjectured \cite[Conjecture 1.16]{HWN} that 
the existence of cKE metrics is equivalent to an algebraic stability condition. 
In fact, they conjectured that the condition should be the K-polystability 
or the Ding polystability of $\left(X; \left\{L_i\right\}_{i=1}^k\right)$. 
However, in their definition of K-stability/Ding stability, they assumed that 
the total spaces of test configurations of $(X, L_i)$ are isomorphic 
to each other \cite[Definition 1.3]{HWN}. Later, the second author \cite{hashimoto} 
observed that we should consider test configurations $\left(\sX_i, \sL_i\right)/\A^1$ 
of $(X, L_i)$ such that the total spaces $\sX_i$ are different to each other 
in general, and defined the test configuration $(\sY,\sL_\sY)/\A^1$ 
\emph{generated by the $\C^*$-actions of} 
$\left\{(\sX_i,\sL_i)\right\}_{i=1}^k$ \cite[Definition 18]{hashimoto}, 
and then introduced the coupled Ding invariant 
of $\left\{(\sX_i,\sL_i)\right\}_{i=1}^k$ from $(\sY,\sL_\sY)/\A^1$ 
\cite[Definition 19]{hashimoto}. He conjectured 
\cite[Remark 11]{hashimoto} that, for a Fano manifold $X$ (without any assumption 
of the automorphism group of $X$), the existence of cKE metrics should be 
equivalent to an equivariant version of the \emph{uniform} coupled Ding stability, which 
he did not define precisely. 

On the other hand, assuming that the automorphism group of $X$ is finite, Kewei Zhang 
showed \cite[Remark 5.3]{kewei} that the existence of cKE metrics follows from 
the condition 
\[
\delta\left(X;\left\{L_i\right\}_{i=1}^k\right)>1, 
\]
where $\delta\left(X;\left\{L_i\right\}_{i=1}^k\right)$ is the 
\emph{coupled stability threshold} (see \cite{coupled-delta} for the basic theory, 
see also \cite{RTZ}). 
Therefore, it is important to see 
the relationship between the coupled stability threshold 
and the coupled Ding stability to consider the modified version (in the sense of \cite{hashimoto}) of the conjecture of Hultgren and Witt Nystr\"om. 

The purpose of the article is to see the relationship between 
(a reduced version of) the uniform coupled Ding stability and the coupled stability 
thresholds for log Fano pairs, and also the cKE metrics for Fano manifolds. Let $(X,\Delta)$ be a \emph{log Fano pair}, i.e., 
$(X,\Delta)$ is a projective klt pair with $\Delta$ an effective $\Q$-divisor such that 
$-(K_X+\Delta)$ is ample. Fix an algebraic torus $\T\subset\Aut(X,\Delta)$ and 
take $\T$-linearized ample $\Q$-line bundles $L_1,\dots,L_k$ on $X$ 
such that $L:=\sum_{i=1}^k L_i$ coincides with $-(K_X+\Delta)$ with the standard 
$\T$-linearization. As in the standard textbook \cite[\S 2.2]{Xu}, we firstly introduce the coupled weighted barycenter $\alpha_{\bc}^{\cp}:=\sum_{i=1}^k
\alpha_{\bc}^{L_i}$, 
and define the notion of \emph{$\left(X,\Delta;\left\{L_i\right\}_{i=1}^k\right)$ 
with vanishing coupled $\T$-Futaki characters} as $\alpha_{\bc}^{\cp}=0$. 
We also define \emph{the sum configuration} $(\sX,\sL)/\A^1$
of the test configurations
$(\sX_i,\sL_i)/\A^1$ of $(X, L_i)$ (see Definition \ref{definition:sum-tc}). 
It is worth mentioning that, even when all constituents 
$\sX_i$ are normal, the sum $\sX$ may not be normal in general 
(see \S \ref{section:appendix}). 
Moreover, the sum configuration 
coincides with the one generated by the $\C^*$-actions of the given test 
configurations (see Proposition \ref{proposition:hashimoto}). 
We also define \emph{the sum filtration} of given filtrations in \S \ref{section:sum-filt}, which is 
a natural generalization of sum configurations. Discussions from the point of view of the non-Archimedean metrics are given in \S \ref{section:sum-namc}.
In \S \ref{section:c-ding}, we introduce the coupled Ding invariant, which turns out to be the same as the one in 
\cite[Definition 19]{hashimoto}, and the coupled J-norm for test configurations 
$\{(\sX_i,\sL_i)\}_{i=1}^k$ of $(X, L_i)$. Moreover, 
we introduce the notion of \emph{$\T$-equivariant coupled Ding semistability} and 
\emph{$\T$-reduced uniform coupled Ding stability} (see Definition 
\ref{definition:c-ding}). In \S \ref{section:c-delta}, as a generalization of the reduced stability thresholds \cite{Li, XZ} and the coupled stability 
thresholds \cite{RTZ, coupled-delta}, we introduce \emph{the $\T$-reduced coupled stability threshold} 
\[
\delta_{\T}^{\Red}\left(X,\Delta;\left\{L_i\right\}_{i=1}^k\right) 
\]
of $\left(X,\Delta;\{L_i\}_{i=1}^k\right)$ (see Definition \ref{definition:c-delta}). The relationship to the cKE metrics is discussed in \S \ref{section:analytic}.

Here is the main result of this article. 

\begin{thm}\label{corollary:apmnthm}
Let $X$ be a Fano manifold over the complex number field $\C$, with $\Delta = 0$, and let $\T\subset\Aut_0(X)$ be a maximal torus with the maximal compact subgroup $\T_r \subset \T$. Then the following are equivalent.
\begin{enumerate}
\renewcommand{\theenumii}{\roman{enumii}}
\renewcommand{\labelenumii}{(\theenumii)}
\item\label{corollary:apmnthm1}
$\left(X ;\{L_i\}_{i=1}^k\right)$ is $\T$-reduced uniformly coupled Ding stable.
\item\label{corollary:apmnthm2}
$\left(X ;\{L_i\}_{i=1}^k\right)$ has vanishing coupled $\T$-Futaki characters and $\delta_{\T}^{\Red}\left(X;\{L_i\}_{i=1}^k\right)>1$.
\item\label{corollary:apmnthm3}
$\left(X ;\{L_i\}_{i=1}^k\right)$ admits a $\T_r$-invariant coupled K\"ahler--Einstein metric.
\end{enumerate}
\end{thm}

In the above and throughout this article, $\Aut_0 (X)$ stands for the identity component of $\Aut (X)$. Note that the proof of Theorem \ref{corollary:apmnthm} for the case when $\mathrm{Aut}(X)$ is finite can be simplified significantly, since essentially it reduces to proving \eqref{corollary:apmnthm1}$\implies$\eqref{corollary:apmnthm2}; \eqref{corollary:apmnthm2}$\implies$\eqref{corollary:apmnthm3} was already proved in this case by Kewei Zhang \cite[Remark 5.3]{kewei}, and the proof of \eqref{corollary:apmnthm3}$\implies$\eqref{corollary:apmnthm1}, given in Theorem \ref{theorem:apmnthm}, is much easier than the converse.

Theorem \ref{corollary:apmnthm} is an immediate consequence of Theorems \ref{theorem:main} and \ref{theorem:mainmc} stated below. The first result establishes \eqref{corollary:apmnthm1}$\iff$\eqref{corollary:apmnthm2} and holds for any log Fano pairs over any algebraically closed field of characteristic zero.

\begin{thm}[{see Corollary \ref{corollary:conclusion} in detail}]\label{theorem:main}
\begin{enumerate}
\renewcommand{\theenumi}{\arabic{enumi}}
\renewcommand{\labelenumi}{(\theenumi)}
\item\label{theorem:main1}
The following are equivalent: 
\begin{enumerate}
\renewcommand{\theenumii}{\roman{enumii}}
\renewcommand{\labelenumii}{(\theenumii)}
\item\label{theorem:main11}
$\left(X,\Delta;\{L_i\}_{i=1}^k\right)$ is coupled Ding semistable.
\item\label{theorem:main12}
$\left(X,\Delta;\{L_i\}_{i=1}^k\right)$ is $\T$-equivariantly coupled Ding semistable.
\item\label{theorem:main13}
$\delta\left(X,\Delta;\{L_i\}_{i=1}^k\right)\geq 1$ holds. 
\end{enumerate}
\item\label{theorem:main2}
If $\left(X,\Delta;\{L_i\}_{i=1}^k\right)$ is $\T$-reduced uniformly coupled Ding 
stable, then $\T\subset\Aut(X,\Delta)$ must be a maximal torus. 
\item\label{theorem:main3}
The following are equivalent for a maximal torus $\T\subset\Aut(X,\Delta)$: 
\begin{enumerate}
\renewcommand{\theenumii}{\roman{enumii}}
\renewcommand{\labelenumii}{(\theenumii)}
\item\label{theorem:main31}
$\left(X,\Delta;\{L_i\}_{i=1}^k\right)$ is $\T$-reduced uniformly coupled Ding stable.
\item\label{theorem:main32}
$\left(X,\Delta;\{L_i\}_{i=1}^k\right)$ has vanishing coupled $\T$-Futaki characters and
\[
\delta^{\Red}_{\T}\left(X,\Delta;\left\{L_i\right\}_{i=1}^k\right)> 1
\] 
holds. 
\end{enumerate}
\end{enumerate}
\end{thm}

For the proof of Theorem \ref{theorem:main}, 
we \emph{heavily} depend on the recent progress in 
K-stability of log Fano pairs. In fact, the strategy of the proof of 
Theorem \ref{theorem:main} is very close to the construction of the arguments 
in \cite{Xu}. 
We refer the reader to the new book \cite{Xu} for the terminologies used in this article. 

Furthermore, relying on the recent progress in the variational methods for K\"ahler--Einstein and constant scalar curvature K\"ahler metrics such as \cite{BBJ,BJ22,darlectures,DZ,gzbook,LiCsck,Li,kewei2}, we solve a 
modified version of the conjecture \cite[Conjecture 1.16]{HWN} by Hultgren and Witt Nystr\"om in the affirmative, establishing \eqref{corollary:apmnthm1}$\iff$\eqref{corollary:apmnthm3} in Theorem \ref{corollary:apmnthm}.

\begin{thm}\label{theorem:mainmc}
Let $X$ be a Fano manifold over the complex number field $\C$, with $\Delta = 0$, and let $\T\subset\Aut_0(X)$ be a maximal torus with the maximal compact subgroup $\T_r \subset \T$. Then $\left(X ;\{L_i\}_{i=1}^k\right)$ admits a $\T_r$-invariant coupled K\"ahler--Einstein metric if and only if $\left(X ;\{L_i\}_{i=1}^k\right)$ is $\T$-reduced uniformly coupled Ding stable.
\end{thm}



In this paper, we only consider a torus $\T\subset\Aut(X,\Delta)$ which is often assumed to be maximal. We believe that the results in this paper can be generalized to the situation where we consider any connected reductive subgroup $\G$ of $\Aut(X , \Delta )$, as Chi Li \cite{Li} proves results for the $\G$-reduced uniform Ding stability and the $\G$-reduced stability threshold for the K\"ahler--Einstein case ($k=1$). We do not pursue this direction any further in this paper as the problem seems nontrivial, particularly since \cite{Li} uses the $\G$-equivariant minimal model program which does not seem readily extendable to our situation.

Throughout the rest of this article (except for \S 
\ref{section:analytic}), 
we work over any algebraically closed field $\Bbbk$ 
of characteristic zero.

\begin{ack}
The authors thank Chenyang Xu who wrote the nice book \cite{Xu}. 
The authors thank Chi Li for answering our questions on Remark \ref{remark:theta}. 
K.F.\ was supported by JSPS KAKENHI Grant Number 22K03269, 
Royal Society International Collaboration Award 
ICA\textbackslash 1\textbackslash 23109 and Asian Young Scientist Fellowship. 
Y.H.\ was supported by JSPS KAKENHI Grant Number 23K03120 and 24K00524, and part of this work was carried out when he was visiting Centre de Recherche Math\'ematique as a CRM-Simons Scholar supported by the CRM and the Simons Foundation.
\end{ack}

\section{Torus actions on polarized varieties}\label{section:torus}

We recall the results in \cite[\S 2--3]{Li} and \cite[\S 6]{Xu}. 
In this section, we fix an $n$-dimensional projective 
klt pair $(X,\Delta)$ (with $\Delta$ an effective $\Q$-divisor), 
an algebraic torus $\T\simeq\G_m^p$ with $p\in\Z_{\geq 0}$, 
an injection $\T\to\Aut(X,\Delta)$, and a $\T$-linearized ample $\Q$-line bundle 
$L$ on $X$, i.e., there exists $r\in\Z_{>0}$ such that $r L$ is a $\T$-linearized 
ample line bundle on $X$. 
As in \cite[\S 2.2.1]{Xu}, we set 
\begin{eqnarray*}
M(\T):=\Hom(\T,\G_m), \quad &&M_\K(\T):=M(\T)\otimes_\Z\K, \\
N(\T):=\Hom(\G_m,\T), \quad &&N_\K(\T):=N(\T)\otimes_\Z\K
\end{eqnarray*}
for $\K\in\{\Q,\R\}$. Set 
\[
R:=\bigoplus_{m\in r \Z_{\geq 0}}R_m:=\bigoplus_{m\in r \Z_{\geq 0}}H^0(X, m L),
\]
and let 
\[
R_m=\bigoplus_{\alpha\in M(\T)}R_{m,\alpha}
\]
be the weight decomposition of $R_m$, i.e., 
\[
R_{m,\alpha}:=\left\{s\in R_m\,\,|\,\,\xi(t)\cdot s=t^{\langle\alpha,\xi\rangle}
\cdot s\,\,\left(\forall\xi\in N(\T),\,\,\forall t\in\G_m\right)\right\},
\]
as in \cite[(2.22)]{Xu}. 

\begin{definition}[{\cite[\S 2.4]{Li}, \cite[\S A]{XZ}, 
\cite[\S 6.1]{Xu}}]\label{definition:valuation-twist}
Let $\Val_X^\T$ be the set of all $\T$-invariant valuations on $X$, and let us set 
\begin{eqnarray*}
\Val_X^{<\infty,\T}&:=&\left\{v\in\Val_X^\T\,\,|\,\,A_{X,\Delta}(v)<\infty\right\}, \\
\QM_X^\T&:=&\{v\in\Val_X^{<\infty,\T}\,\,|\,\,v:\text{ quasi-monomial}\},
\end{eqnarray*}
where $A_{X,\Delta}(v)\in\R_{\geq 0}\cup\{\infty\}$ is the log discrepancy 
of $(X,\Delta)$ along $v$. For any $\xi\in N_\R(\T)$, as in 
\cite[Definition-Lemma 6.15]{Xu}, we can define 
$\operatorname{wt}_\xi\in\QM_X^\T$. Moreover, for any $v\in\Val_X^\T$, we can 
define the \emph{$\xi$-twist} $v_\xi\in\Val_X^\T$ of $v$. 
We know that, if $v\in\Val_X^{<\infty,\T}$ (resp., if $v\in\QM_X^\T$), then we have 
$v_\xi\in\Val_X^{<\infty,\T}$ (resp., $v_\xi\in\QM_X^\T$). We set 
\begin{eqnarray*}
\Val_X^{*,\T}&:=&\Val_X^{<\infty,\T}\setminus\{\operatorname{wt}_\xi\,\,
|\,\,\xi\in N_\R(\T)\}, \\
\QM_X^{*,\T}&:=&\QM_X^{\T}\setminus\{\operatorname{wt}_\xi\,\,
|\,\,\xi\in N_\R(\T)\}.
\end{eqnarray*}
\end{definition}

\begin{definition}[{\cite[\S 2.2.1]{Xu}}]\label{definition:polytope}
\begin{enumerate}
\renewcommand{\theenumi}{\arabic{enumi}}
\renewcommand{\labelenumi}{(\theenumi)}
\item\label{definition:polytope1}
For any $m\in r\Z_{\geq 0}$, set 
\begin{eqnarray*}
\Lambda_m^L&:=&\left\{\alpha\in M(\T)\,\,|\,\,R_{m,\alpha}\neq 0\right\}, \\
\PP_m^L&:=&\Conv\left(\Lambda_m^L\right)\subset M_\R(\T), \\
\PP^L&:=&\Conv\left(\bigcup_{m\in r\Z_{>0}}\frac{1}{m}\PP_m^L\right).
\end{eqnarray*}
As in \cite[Lemma 2.33]{Xu}, the set $\PP^L$ is a rational polytope of maximal 
dimension. Moreover, for any sufficiently divisible $m\in r\Z_{>0}$, we have 
$\PP^L=\frac{1}{m}\PP_m^L$. 
\item\label{definition:polytope2}
The \emph{$\T$-equivariant Duistermaat--Heckman measure} $d\nu_{\DDHH,\T}$ 
on $\PP^L\subset M_\R(\T)$ is defined to be the weak limit of 
\[
d\rho_{m,\T}:=\frac{1}{m^n}\sum_{\alpha\in\Lambda_m^L}\dim R_{m,\alpha}
\cdot\delta_{\frac{\alpha}{m}}.
\]
The \emph{weighted barycenter $\alpha_{\bc}^L\in M_\R(\T)$ of $\PP^L$} is 
defined to be 
\[
\alpha_{\bc}^L:=\lim_{m\to\infty}\frac{1}{m\dim R_m}\sum_{\alpha\in\Lambda_m^L}
\dim R_{m,\alpha}\cdot\alpha=\frac{n!}{\vol(L)}\int_{\PP^L}\alpha d\nu_{\DDHH,\T}.
\]
As in \cite[Lemmas 2.33 and 2.35]{Xu}, we know that $\alpha_{\bc}^L\in M_\Q(\T)$ 
with $\alpha_{\bc}^L\in\interior\left(\PP^L\right)$.
\end{enumerate}
\end{definition}

\begin{definition}[{\cite[\S 2.2.1]{Xu}}]\label{definition:theta}
\begin{enumerate}
\renewcommand{\theenumi}{\arabic{enumi}}
\renewcommand{\labelenumi}{(\theenumi)}
\item\label{definition:theta1}
For any $m\in r\Z_{\geq 0}$ and $\alpha\in\Lambda_m^L$, we set 
\[
I_{m,\alpha}:=\Image\left(R_{m,\alpha}\otimes_\Bbbk\sO_X(-m L)\to \sO_X\right). 
\]
Note that the graded $\sO_X$-algebra
\[
\bigoplus_{m\in r\Z_{\geq 0}}\bigoplus_{\alpha\in\Lambda_m^L}I_{m,\alpha}
\]
is finitely generated. 
For any $v\in\Val_X$, we set 
\[
v\left(I_{\bullet(m,\alpha)}\right):=\inf_{k\in\Z_{>0}}\frac{v(I_{k m,k \alpha})}{k}
=\lim_{k\to\infty}\frac{v(I_{k m,k \alpha})}{k}.
\]
From the above finite generation, we have 
\[
v\left(I_{\bullet(m,\alpha)}\right)=\frac{v(I_{k m,k \alpha})}{k}
\]
holds for any sufficiently divisible $k\in\Z_{>0}$. 
\item\label{definition:theta2}
For any $v\in\Val_X$, let us define 
\begin{eqnarray*}
\gamma_v\colon\PP^L\cap M_\Q(\T)&\to&\R_{\geq 0}\\
\alpha&\mapsto&\frac{v\left(I_{\bullet(k,k\alpha)}\right)}{k}
\end{eqnarray*}
for a sufficiently divisible $k\in\Z_{>0}$ for each $\alpha\in\PP^L\cap M_\Q(\T)$. 
By \cite[Proposition 4.7]{ELMNP}, the above function uniquely extends to a 
convex, continuous and rationally piecewise affine function 
\[
\gamma_v\colon\PP^L\to\R_{\geq 0}.
\]
More precisely, $\PP^L$ can be covered by finitely many rational polytopes 
$C_\lambda$ such that $\gamma_v|_{C_\lambda}$ is affine for each $C_\lambda$. 
\item\label{definition:theta3} (cf.\ \cite[\S 2.5.3]{Li})
Take any $v\in\Val_X$ and $\xi\in N_\R(\T)$. 
\begin{itemize}
\item
For any $m\in r\Z_{>0}$, we set 
\[
\theta_{\xi,m}^L(v):=\frac{1}{m}\max_{\alpha\in\Lambda_m^L}
\left\{-\langle\alpha,\xi\rangle-v(I_{m,\alpha})\right\}.
\]
\item
Set 
\[
\theta_\xi^L(v):=\max_{\alpha\in\PP^L}
\left\{-\langle\alpha,\xi\rangle-\gamma_v(\alpha)\right\}.
\]
The maximum of the right hand side can be attained by an element in 
$\PP^L\cap M_\Q(\T)$ since $\gamma_v$ is rationally piecewise affine. 
By the following Lemma \ref{lemma:theta}, we have 
\[
\theta_\xi^L(v)=\lim_{m\to\infty}\theta_{\xi,m}^L(v).
\]
\end{itemize}
\end{enumerate}
\end{definition}

\begin{lemma}\label{lemma:theta}
We have
\[
\lim_{m\to\infty}\theta_{\xi,m}^L(v)=\max_{\alpha\in\PP^L}
\left\{-\langle\alpha,\xi\rangle-\gamma_v(\alpha)\right\}.
\]
\end{lemma}

\begin{proof}
Observe that 
\begin{eqnarray*}
\sup_{m\in r\Z_{>0}}\theta_{\xi,m}^L(v)&=&\sup_m\max_{\alpha\in\Lambda_m^L}
\left\{-\left\langle\frac{\alpha}{m},\xi\right\rangle-\frac{1}{m}v(I_{m,\alpha})\right\}\\
&=&\sup_{\alpha\in\PP^L\cap M_\Q(\T)}
\left\{-\langle\alpha,\xi\rangle-\gamma_v(\alpha)\right\}
=\max_{\alpha\in\PP^L}\left\{-\langle\alpha,\xi\rangle-\gamma_v(\alpha)\right\}.
\end{eqnarray*}
Take $\alpha_0\in\PP^L\cap M_\Q(\T)$ and a sufficiently divisible $m\in r\Z_{>0}$ 
with 
\begin{eqnarray*}
\max_{\alpha\in\PP^L}\left\{-\langle\alpha,\xi\rangle-\gamma_v(\alpha)\right\}
&=&-\langle\alpha_0,\xi\rangle-\gamma_v(\alpha_0), \\
 m\alpha_0&\in&\Lambda_m^L, \\
 \gamma_v(\alpha_0)&=&\frac{1}{m}v(I_{m,m\alpha_0}).
\end{eqnarray*}
Then we have 
\[
\theta_{\xi,m}^L(v)\geq \frac{1}{m}\left(-\langle m\alpha_0,\xi\rangle
-v(I_{m,m\alpha_0})\right)=-\langle\alpha_0,\xi\rangle-\gamma_v(\alpha_0).
\]
Thus the assertion follows.
\end{proof}

\begin{definition}[{\cite[\S 6.1.1]{Xu}}]\label{definition:filtration}
A \emph{$\T$-equivariant filtration $\sF$ on $R$} in this article is defined to be 
a $\T$-equivariant linearly bounded and graded multiplicative filtration on $R$ 
in the sense of \cite[Definition 3.14 and \S 6.1.1]{Xu}. 
For any $m\in r\Z_{\geq 0}$ and $\alpha\in M(\T)$, we denote the restriction of 
$\sF$ to $R_{m,\alpha}$ by $\sF$ again. Then, for any $x\in\R$, we have 
\[
\sF^x R_m=\bigoplus_{\alpha\in M(\T)}\sF^x R_{m,\alpha}.
\]
For any $m\in r\Z_{\geq 0}$, $\alpha\in M(\T)$ and $x\in\R$, we set 
\[
I_{(m,\alpha;x)}(\sF):=\Image\left(\sF^x R_{m,\alpha}\otimes_{\Bbbk}
\sO_X(-m L)\to\sO_X\right)
\]
and 
\[
I_{(m;x)}(\sF):=\sum_{\alpha\in M(\T)}I_{(m,\alpha;x)}(\sF)
=\Image\left(\sF^x R_m\otimes_{\Bbbk}
\sO_X(-m L)\to\sO_X\right).
\]
Moreover, let $\sI_m(\sF)$ be the ($\G_m$-invariant) fractional ideal sheaf 
on $X_{\A^1}:=X\times\A^1_t$ defined by 
\[
\sI_m(\sF):=\bigoplus_{\lambda\in\Z}t^{-\lambda}I_{(m;\lambda)}(\sF). 
\]
We also define the following: 
\begin{enumerate}
\renewcommand{\theenumi}{\arabic{enumi}}
\renewcommand{\labelenumi}{(\theenumi)}
\item\label{definition:filtration1}
For any $C\in\R$, let $\sF_{[C]}$ be the \emph{$C$-shift} of $\sF$, i.e., 
$\sF_{[C]}^x R_m:=\sF^{x-C m}R_m$. 
\item\label{definition:filtration2}
For any $\xi\in N_\R(\T)$, let $\sF_\xi$ be the \emph{$\xi$-twist} of $\sF$, 
i.e., 
\[
\sF_\xi^x R_m:=\bigoplus_{\alpha\in M(\T)}
\sF^{x-\langle\alpha,\xi\rangle}R_{m,\alpha}.
\]
\end{enumerate}
It is obvious from the definition that, both $\sF_{[C]}$ and $\sF_\xi$ are 
$\T$-equivariant filtrations of $R$. Moreover, we have 
$\sF_{[C],\xi}=\sF_{\xi,[C]}$. 
\end{definition}

\begin{example}\label{example:twisted-tc}
Let $(\sX,\sL)/\A^1$ be any $\T$-equivariant test configuration of $(X,L)$ 
in the sense of \cite[Definition 2.3 and Remark 2.20]{Xu}. Then, as in 
\cite[Example 3.34]{Xu}, after replacing $r\in\Z_{>0}$ if necessary, 
we get the $\T$-equivariant filtration $\sF_{\sX,\sL}$ on $R$ associated to 
$(\sX,\sL)/\A^1$. For any $\xi\in N(\T)$, the \emph{$\xi$-twisted test configuration}
$(\sX_\xi,\sL_\xi)/\A^1$ of $(\sX,\sL)/\A^1$ as in \cite[Example 6.9]{Xu} 
satisfies that $\sF_{\sX_\xi,\sL_\xi}=\left(\sF_{\sX,\sL}\right)_\xi$
by \cite[Lemma 6.10]{Xu}. 
\end{example}

\begin{lemma}\label{lemma:uniform-convergence}
Let $\sF$ be a $\T$-equivariant filtration on $R$ and take any 
$v\in\Val_X$ and $\lambda\in\left(-\infty, \lambda_{\max}(\sF)\right)$, 
where $\lambda_{\max}$ is as in \cite[Definition 3.20]{Xu}. 
Then, the sequence 
\[
\left\{x\mapsto\frac{v(I_{(m; m x)}(\sF))}{m}\right\}_{m\in r\Z_{>0}}
\]
of functions over $x\in\left(-\infty, \lambda\right]$ uniformly converges to 
the function $x\mapsto v\left(I_{\bullet}^{(x)}(\sF)\right)$, where 
$I_{\bullet}^{(x)}(\sF)$ is the sequence of graded ideal sheaves on $X$ 
defined by $I_{m}^{(x)}(\sF):=I_{(m; m x)}(\sF)$ (see \cite[Definition 3.42]{Xu}).  
\end{lemma}

\begin{proof}
Set $f_m(x):=v(I_{(m; m x)}(\sF))/m$ and $f(x):=v\left(I_{\bullet}^{(x)}(\sF)\right)$. 
We know that, over $x\in\left(-\infty, \lambda_{\max}(\sF)\right)$, 
the functions $f_m(x)$ and $f(x)$ are \emph{non-decreasing} functions. 
Moreover, the function $f(x)$ is \emph{continuous} since it is convex. 
The sequence of functions $\{f_m(x)\}_{m\in r\Z_{>0}}$ pointwise converges to 
$f(x)$. Moreover, there exists $a\in\R$ such that $f_m(x)=f(x)=0$ for any $x<a$. 
Thus it is enough to show the uniform convergence over the area $[a,\lambda]$, 
but then the assertion is well-known: if a sequence of monotone functions over 
$[a,\lambda]$ pointwise converges to a continuous function, then the 
convergence is uniform. 
\end{proof}

\begin{proposition}[{cf.\ \cite[Proposition 3.3]{Li}}]\label{proposition:theta}
Take any $v\in\Val_X^{<\infty,\T}$. Then the value $\theta_\xi^L(v)$ in Definition 
\ref{definition:theta} coincides with the one in \cite[(111), (121)]{Li} 
(see Remark \ref{remark:theta}). In particular, if $L=-(K_X+\Delta)$ with the 
standard $\T$-linearization, then we have the equality 
\[
\theta_\xi^{-(K_X+\Delta)}(v)=A_{X,\Delta}(v_\xi)-A_{X,\Delta}(v).
\]
\end{proposition}

\begin{proof}
We follow the notation in \cite{Li}. Let $\sF_{\triv}$ be the trivial filtration on $R$ 
(in the sense of \cite[Example 3.21]{Xu}). Let $\phi^{\sF_{\triv,-\xi}}$ be the 
non-Archimedean potential associated with $\sF_{\triv,-\xi}$ in the sense of 
\cite{BJ22}. As in \cite[Proposition 3.9]{Li}, we have 
\[
\phi^{\sF_{\triv,-\xi}}(v)=\lim_{m\to\infty}\phi_m^{\sF_{\triv,-\xi}}(v), 
\]
where we have 
\[
\phi_m^{\sF_{\triv,-\xi}}(v)=\frac{1}{m}\max_{\substack{
x\in\R \\ s\in\sF_{\triv,-\xi}^x R_m}}\left\{x-v(s)\right\} 
\]
by \cite[(98), (100)]{Li}. Since 
\[
\sF_{\triv,-\xi}^x R_m=\bigoplus_{\alpha\in M(\T); x\leq-\langle\alpha,\xi\rangle}
R_{m,\alpha},
\]
we have 
\begin{eqnarray*}
\phi_m^{\sF_{\triv,-\xi}}(v)&=&\frac{1}{m}\max_{x\in\R}\left\{x-v\left(
\sum_{\alpha\in M(\T), x\leq-\langle\alpha,\xi\rangle}I_{m,\alpha}\right)\right\}\\
&=&\frac{1}{m}\max_{x\in\R}\max_{\alpha\in M(\T); x\leq-\langle\alpha,\xi\rangle}
\left\{x-v(I_{m,\alpha})\right\}\\
&=&\frac{1}{m}\max_{\alpha\in\Lambda_m^L}\left\{-\langle\alpha,\xi\rangle
-v(I_{m,\alpha})\right\}=\theta_{\xi,m}^L(v).
\end{eqnarray*}
Thus the assertion follows by Lemma \ref{lemma:theta}. 
\end{proof}

\begin{remark}\label{remark:theta}
In our terminologies of group actions, 
on the left hand side of 
the equations (111), (121), (130), (131) in \cite{Li}, we must replace $\xi$ with 
$-\xi$. For example, the equation (111) should be replaced by 
\[
\phi_{(\sZ_{-\xi},\sL_{-\xi})}(w)=\phi_{(\sZ,\sL)}(w_\xi)+\theta_\xi^L(w).
\]
Let us consider a simple example. Let us assume that $L=-(K_X+\Delta)$
with the standard $\T$-linearization for simplicity. 
For any $\xi\in N_\R(\T)$, we know that 
\[
\sF_{\operatorname{wt}_\xi}=\sF_{\triv,\xi,
\left[\theta_\xi^{-(K_X+\Delta)}(v_{\triv})\right]}
\]
by \cite[Lemma 6.22]{Xu}, where $v_{\triv}$ is the trivial valuation and 
$\sF_{\triv}$ is the trivial filtration on $R$. Thus, by 
\cite[Example 6.13]{Xu}, we must have 
\[
\theta_\xi^{-(K_X+\Delta)}(v_{\triv})=-\lambda_{\PP}(\xi), 
\]
where 
\[
\lambda_{\PP}(\xi):=\min_{\alpha\in\PP^{-(K_X+\Delta)}}\left\{
\langle\alpha,\xi\rangle\right\}.
\]
On the other hand, by (the corrected version of) \cite[(121)]{Li} shows that 
\[
\phi^{\sF_{\triv,-\xi}}(v_{\triv})=\theta_\xi^{-(K_X+\Delta)}(v_{\triv}).
\]
As in the proof of Proposition \ref{proposition:theta}, we have 
\[
\phi^{\sF_{\triv, -\xi}}(v_{\triv})=\lim_{m\to\infty}\frac{1}{m}\max_{\alpha\in
\Lambda_m^{-(K_X+\Delta)}}\left\{-\langle\alpha,\xi\rangle
-v_{\triv}(I_{m,\alpha})\right\}=-\lambda_{\PP}(\xi).
\]
We note that 
\[
\phi^{\sF_{\triv, \xi}}(v_{\triv})=\max_{\alpha\in\PP^{-(K_X+\Delta)}}
\left\{\langle\alpha,\xi\rangle\right\}, 
\]
which is different from $-\lambda_{\PP}(\xi)$ in general. 
\end{remark}

\begin{corollary}[{cf.\ \cite[Proposition 3.3]{Li}, 
\cite[Lemma 6.21]{Xu}}]\label{corollary:theta}
For any $\xi\in N(\T)$, let $\phi_{\xi}\colon\G_m\to\Aut(X,\Delta)$ be the 
one-parameter subgroup of $\Aut(X,\Delta)$ defined by $-\xi$. Set 
\begin{eqnarray*}
\sigma_{\xi}\colon X\times\G_m&\to&X\times\G_m\\
(x,t)&\mapsto&\left(\phi_{\xi}(t)\cdot x,t\right).
\end{eqnarray*}
Let us consider a birational model resolving $\sigma_{\xi}$: 
\[\xymatrix{
&\sW \ar[ld]_{\mu_1} \ar[rd]^{\mu_2} &\\
X_{\A^1} \ar@{-->}[rr]_{\sigma_{\xi}} & & X_{\A^1}.
}\]
For any $v\in\Val_X^{<\infty,\T}$, let $G(v)$ be the Gauss extension of $v$ 
in the sense of \cite[(93)]{Li}, which is a valuation on the \emph{left hand side} 
of $X_{\A^1}$. Then we have
\[
\theta_\xi^L(v)=G(v)\left(\mu_2^*L_{\A^1}-\mu_1^*L_{\A^1}\right).
\]
\end{corollary}

\begin{proof}
Follows immediately from \cite[(94), (121)]{Li}. 
\end{proof}


\begin{corollary}[{cf.\ \cite[Lemma 6.22]{Xu}}]\label{corollary:twist-valuation}
For any $\alpha\in M(\T)$, $m\in r\Z_{>0}$, $s\in R_{m,\alpha}\setminus\{0\}$, 
$v\in\Val_X^{<\infty,\T}$ and $\xi\in N_\R(\T)$, we have 
\[
v_\xi(s)=v(s)+\langle\alpha,\xi\rangle+m\theta_\xi^L(v).
\]
In particular, we have $\sF_{v_\xi}=\left(\sF_v\right)_{\xi,[\theta_\xi^L(v)]}$.
\end{corollary}

\begin{proof}
The proof is same as the proof of \cite[Lemma 6.22]{Xu}. 
When $\xi\in N(\T)$, we have the assertion by using Corollary \ref{corollary:theta}
and the argument in \cite[Lemma 6.22]{Xu}. For any $e\in\R_{>0}$, we know that 
$\theta_{e\theta}^L(e v)=e\theta_\xi^L(v)$ and $e(v_\xi)=(e v)_{e\xi}$, we get the 
assertion when $\xi\in N_\Q(\T)$. Thus we get the assertion when $\xi\in N_\R(\T)$
by the continuities. 
\end{proof}

\begin{corollary}[{cf.\ \cite[Lemma 6.23]{Xu}}]\label{corollary:S-twist}
For any $v\in\Val_X^{<\infty,\T}$ and $\xi\in N_\R(\T)$, we have 
\[
S_L(v_\xi)=S_L(v)+\langle\alpha_{\bc}^L,\xi\rangle+\theta_\xi^L(v),
\]
where $S_L$ is the $S$-invariant \cite[(4.42)]{Xu}. 
\end{corollary}

\begin{proof}
The proof is same as the proof of \cite[Lemma 6.23]{Xu}. We have 
\begin{eqnarray*}
S_L(v_\xi)&=&S_L(\sF_{v_\xi})=S_L\left(\left(\sF_v\right)_{\xi,[\theta_\xi^L(v)]}\right)\\
&=&\theta^L_\xi(v)+S_L\left(\left(\sF_v\right)_\xi\right)
=\theta^L_\xi(v)+S_L(\sF_v)+\langle\alpha_{\bc}^L,\xi\rangle, 
\end{eqnarray*}
where the second equality follows from Corollary \ref{corollary:twist-valuation}, 
the third equality follows from the obvious equality $S_L(\sF_{[C]})=S_L(\sF)+C$, 
and the last equality follows from \cite[Lemma 6.4]{Xu}. 
\end{proof}

Recall the following notations in \cite[\S 3.4]{Xu}: 

\begin{definition}[{\cite[\S 3.4]{Xu}}]\label{definition:approx}
Let $\sF$ be a $\T$-equivariant filtration on $R$. 
\begin{enumerate}
\renewcommand{\theenumi}{\arabic{enumi}}
\renewcommand{\labelenumi}{(\theenumi)}
\item\label{definition:approx1} \cite[Definitions 3.2 and 3.16]{Xu}
We say that $\sF$ is a \emph{$\Z$-valued filtration} if 
$\sF^x R_m=\sF^{\lceil x\rceil}R_m$ holds for any $x\in\R$ and $m\in r\Z_{\geq 0}$. 
For any $\T$-equivariant filtration $\sF$ on $R$, let us define the 
$\Z$-valued filtration $\sF_\Z$ as 
$\sF_\Z^x R_m:=\sF^{\lceil x\rceil}R_m$. 
\item\label{definition:approx2} \cite[Definition 3.55]{Xu}
A sequence of $\T$-equivariant filtrations $\left\{\sF_{(m)}\right\}_{m\in r\Z_{>0}}$ 
on $R$ is said to be an \emph{approximating sequence of $\sF$} if, for any 
$m\in r\Z_{>0}$, we have: 
\begin{enumerate}
\renewcommand{\theenumii}{\roman{enumii}}
\renewcommand{\labelenumii}{(\theenumii)}
\item\label{definition:approx21}
$\sF_{(m)}\subset\sF$, 
\item\label{definition:approx22}
$\sF^x_{(m)}R_m=\sF^x R_m$ for any $x\in\R$, and 
\item\label{definition:approx23}
for any $s\in\Z_{>0}$ and $x\in\R$, we have 
\[
\sF_{(m)}^x R_{ms}=\sum_{x_1+\cdots+x_s\geq x}\sF^{x_1}R_m\cdots\sF^{x_s}R_m.
\]
\end{enumerate}
By \cite[Definition-Lemma 3.56]{Xu}, for any $\T$-equivariant filtration $\sF$ 
on $R$, there exists a ($\T$-equivariant) approximating sequence of $\sF$. 
\end{enumerate}
\end{definition}

\begin{lemma}\label{lemma:Z-twist}
Let $\sF$ be a $\T$-equivariant filtration on $R$. 
Set $\sG:=\sF_\Z$. Take any $\xi\in N_\R(\T)$. Then we have 
\[
S_L(\sF_\xi)=S_L(\sG_\xi)\quad\text{and}\quad
\lambda_{\max}(\sF_\xi)=\lambda_{\max}(\sG_\xi). 
\]
\end{lemma}

\begin{proof}
Take any $x\in\R$ and $m\in r\Z_{>0}$. From the definition, we have 
$\sG_\xi^{m x}R_m\subset\sF_\xi^{m x}R_m$. On the other hand, for any 
$\varepsilon\in\R_{>0}$, if we take $m\in r\Z_{>0}$ with $m\varepsilon>1$, then 
\begin{eqnarray*}
\sG_\xi^{m(x-\varepsilon)}R_m&=&
\bigoplus_{\alpha\in\Lambda_m^L}\sF^{\lceil m(x-\varepsilon)
-\langle\alpha,\xi\rangle\rceil}R_{m,\alpha}\\
&\supset&\bigoplus_{\alpha\in\Lambda_m^L}\sF^{m x
-\langle\alpha,\xi\rangle}R_{m,\alpha}=\sF_\xi^{m x}R_m.
\end{eqnarray*}
Thus we get the assertion. 
\end{proof}

\begin{lemma}\label{lemma:approx}
Let $\sF$ be a $\T$-equivariant filtration on $R$, and 
let $\left\{\sF_{(m)}\right\}_{m\in r\Z_{>0}}$ be an approximating sequence of $\sF$. 
\begin{enumerate}
\renewcommand{\theenumi}{\arabic{enumi}}
\renewcommand{\labelenumi}{(\theenumi)}
\item\label{lemma:approx1}
For any $C\in\R$, let $\sF_{(m),[C]}$ be the $C$-shift of $\sF_{(m)}$. Then 
$\left\{\sF_{(m),[C]}\right\}_{m\in r\Z_{>0}}$ is an approximating sequence of 
$\sF_{[C]}$. 
\item\label{lemma:approx2}
For any $\xi\in N_\R(\T)$, let $\sF_{(m),\xi}$ be the $\xi$-twist of $\sF_{(m)}$. 
Then $\left\{\sF_{(m),\xi}\right\}_{m\in r\Z_{>0}}$ is an approximating sequence of 
$\sF_{\xi}$. 
\end{enumerate}
\end{lemma}

\begin{proof}
\eqref{lemma:approx1} is trivial. We only see \eqref{lemma:approx2}. Observe the 
following: 
\begin{enumerate}
\renewcommand{\theenumi}{\roman{enumi}}
\renewcommand{\labelenumi}{(\theenumi)}
\item\label{proof:approx1}
$\sF_{(m),\xi}^x R_{m',\alpha}=\sF_{(m)}^{x-\langle\alpha,\xi\rangle}R_{m',\alpha}
\subset\sF^{x-\langle\alpha,\xi\rangle}R_{m',\alpha}=\sF_\xi^x R_{m',\alpha}$. 
\item\label{proof:approx2}
$\sF_{(m),\xi}^x R_{m,\alpha}=\sF_{(m)}^{x-\langle\alpha,\xi\rangle}R_{m,\alpha}
=\sF^{x-\langle\alpha,\xi\rangle}R_{m,\alpha}=\sF_\xi^x R_{m,\alpha}$. 
\item\label{proof:approx3}
For any $s\in\Z_{>0}$, we have 
\begin{eqnarray*}
&&\sF_{(m),\xi}^x R_{m s,\alpha}
=\sF_{(m)}^{x-\langle \alpha,\xi\rangle}R_{m s,\alpha}\\
&=&\sum_{\substack{\alpha_1,\dots,\alpha_s\in M(\T);\\ 
\alpha_1+\cdots+\alpha_s=\alpha}}
\sum_{y_1+\cdots+y_s\geq x-\langle\alpha,\xi\rangle}
\sF^{y_1}R_{m,\alpha_1}\cdots\sF^{y_s}R_{m,\alpha_s}\\
&=&\sum_{\substack{\alpha_1,\dots,\alpha_s\in M(\T);\\ 
\alpha_1+\cdots+\alpha_s=\alpha}}
\sum_{x_1+\cdots+x_s\geq x}
\sF^{x_1-\langle\alpha_1,\xi\rangle}R_{m,\alpha_1}
\cdots\sF^{x_s-\langle\alpha_s,\xi\rangle}R_{m,\alpha_s}\\
&=&\sum_{\substack{\alpha_1,\dots,\alpha_s\in M(\T);\\ 
\alpha_1+\cdots+\alpha_s=\alpha}}
\sum_{x_1+\cdots+x_s\geq x}
\sF_\xi^{x_1}R_{m,\alpha_1}\cdots\sF_\xi^{x_s}R_{m,\alpha_s}.
\end{eqnarray*}
\end{enumerate}
Thus the assertion \eqref{lemma:approx2} follows. 
\end{proof}

We recall the following result: 

\begin{lemma}[{\cite[Theorems 3.58 and 3.60]{Xu}}]\label{lemma:S-approx}
Let $\sF$ be a $\T$-equivariant filtration on $R$, and 
let $\left\{\sF_{(m)}\right\}_{m\in r\Z_{>0}}$ be an approximating sequence of $\sF$.
Then we have
\[
\lim_{m\to\infty}S_L(\sF_{(m)})=S_L(\sF)\quad\text{and}\quad
\lim_{m\to\infty}\lambda_{\max}(\sF_{(m)})=\lambda_{\max}(\sF).
\]
\end{lemma}

\begin{definition}\label{definition:base-change}
Let $(\sX,\sL)/\A^1$ be a $\T$-equivariant test configuration of $(X, L)$. 
For any $e\in\Z_{>0}$, let us set 
\[
\left(\sX^{(e)},\sL^{(e)}\right):=\left(\sX\times_{\A^1,\pi_e}\A^1,\pi_e^*\sL\right), 
\]
where 
\begin{eqnarray*}
\pi_e\colon\A^1&\to&\A^1\\
t&\mapsto&t^e.
\end{eqnarray*}
The $\left(\sX^{(e)},\sL^{(e)}\right)/\A^1$ is also a $\T$-equivariant 
test configuration of $(X, L)$. 
Moreover, let $\nu\colon \sX^{\overline{(e)}}\to\sX^{(e)}$ be the normalization and 
set $\sL^{\overline{(e)}}:=\nu^*\sL^{(e)}$. 
The $\left(\sX^{\overline{(e)}},\sL^{\overline{(e)}}\right)/\A^1$ is obviously a 
$\T$-equivariant normal test configuration of $(X, L)$. 
\end{definition}

\begin{definition}[{\cite[Definitions 2.8, 3.40 and 6.28]{Xu}}]\label{definition:J}
Let $\sF$ be a $\T$-equivariant filtration on $R$. 
Let 
\[
\JJ(\sF):=\lambda_{\max}\left(\sF\right)-S_L\left(\sF\right)
\]
be the \emph{$\JJ$-norm} of $(\sX,\sL)/\A^1$. 
If $(\sX,\sL)/\A^1$ is a $\T$-equivariant test configuration of $(X, L)$, then 
we set $\JJ(\sX,\sL):=\JJ\left(\sF_{\sX,\sL}\right)$. 
As in 
\cite[Definitions 2.8 and 3.40 and Proposition 3.41]{Xu}, 
$\JJ(\sX,\sL)$ can be expressed in terms of intersection numbers \cite[(2.8)]{Xu}. 
If $\xi\in N_\Q(\T)$, then, by \cite[Definition 6.28]{Xu}, 
we have the equality 
\[
\JJ\left(\left(\sF_{\sX,\sL}\right)_\xi\right)
=\frac{1}{e}\JJ\left(\left(\sX^{(e)}\right)_{e\xi},\left(\sL^{(e)}\right)_{e \xi}\right), 
\]
where $e\in\Z_{>0}$ is any positive integer with $e\xi\in N(\T)$. 
We set $\JJ(\sX_\xi,\sL_\xi):=\JJ\left(\left(\sF_{\sX,\sL}\right)_\xi\right)$. 
If $(\sX,\sL)/\A^1$ is the trivial test configuration of $(X, L)$ (in the sense of 
\cite[Example 2.5]{Xu}), then we write 
$\JJ(X_\xi,L_\xi):=\JJ(\sX_\xi,\sL_\xi)$. 
\end{definition}

\begin{lemma}\label{lemma:J-xi}
For any $\xi\in N_\Q(\T)$, we have
\[
\JJ(X_\xi,L_\xi)=\max_{\alpha\in\PP^L}\left\{\langle\alpha,\xi\rangle\right\}
-\langle\alpha_{\bc}^L,\xi\rangle. 
\]
\end{lemma}

\begin{proof}
We may assume that $\xi\in N(\T)$. 
Note that $S_L(\sF_\xi)=S_L(\sF_{\triv})+\langle\alpha_{\bc}^L,\xi\rangle
=\langle\alpha_{\bc}^L,\xi\rangle$ by \cite[Lemma 6.4]{Xu}. 
Moreover, we know that $\lambda_{\max}=T$, the $T$-invariant 
(see \cite[Lemma 3.22]{Xu}). 
Since 
\[
\sF_\xi^x R_m=\bigoplus_{\alpha\in M(\T)}\sF_{\triv}^{x-\langle\alpha,\xi\rangle}
R_{m,\alpha}, 
\]
we get 
\[
T_m\left(\sF_\xi\right)=\frac{1}{m}\max_{\alpha\in\Lambda_m^L}
\left\{\langle\alpha,\xi\rangle\right\}. 
\]
Thus we get $\lambda_{\max}(\sF_\xi)=
\max_{\alpha\in\PP^L}\left\{\langle\alpha,\xi\rangle\right\}$.
\end{proof}

\begin{definition}[{\cite[Definition 3.62]{Xu}}]\label{definition:blowup-tc}
Let $\sF$ be a $\T$-equivariant filtration on $R$. Take any $m\in r\Z_{>0}$. 
Let 
\[
q_{(m)}\colon\sY_{(m)}\to X_{\A^1}
\]
be the normalized blowup along the fractional ideal sheaf $\sI_m(\sF)$ 
(see Definition \ref{definition:filtration}), and let $m E$ be the Cartier divisor 
on $\sY_{(m)}$ defined by the equation 
$q_{(m)}^{-1}\left(\sI_m(\sF)\right)=\sO_{\sY_{(m)}}(-m E)$. 
(Since $\sI_m(\sF)$ is a fractional ideal, the Cartier divisor $m E$ may not be 
effective.)
Set
\[
\sM_{(m)}:=q_{(m)}^* L_{\A^1}-E, 
\]
where $L_{\A^1}$ is the pullback of $L$ under the projection $X_{\A^1}\to X$. By 
\cite[Lemma 3.61]{Xu}, the $\sM_{(m)}$ is semiample over $\A^1$. 
The ample model $(\sX_{(m)},\sL_{(m)})/\A^1$ of $(\sY_{(m)},\sM_{(m)})$ 
over $\A^1$ is said to be the \emph{normalized blowup test configuration along 
$\sI_m(\sF)$}. 
\end{definition}

\begin{proposition}\label{proposition:tc-twist}
Let $\sF$ be a $\Z$-valued and $\T$-equivariant filtration on $R$. 
For $m\in r\Z_{>0}$, let $\left(\sX_{(m)},\sL_{(m)}\right)/\A^1$ be the 
normalized blowup test configuration along $\sI_m(\sF)$. 
\begin{enumerate}
\renewcommand{\theenumi}{\arabic{enumi}}
\renewcommand{\labelenumi}{(\theenumi)}
\item\label{proposition:tc-twist1}
For any $\xi\in N(\T)$, set 
\[
\sI_m(\sF)_\xi:=\sI_m(\sF_\xi)=\bigoplus_{\lambda\in\Z}
t^{-\lambda}\sum_{\alpha\in M(\T)}I_{(m,\alpha;\lambda-\langle\alpha,\xi\rangle)}
(\sF). 
\]
Then the $\xi$-twisted test configuration 
$\left(\sX_{(m),\xi},\sL_{(m),\xi}\right)/\A^1$ of 
$\left(\sX_{(m)},\sL_{(m)}\right)/\A^1$ is equal to the normalized blowup test 
configuration along $\sI_m(\sF)_\xi$. 
\item\label{proposition:tc-twist2}
For any $e\in\Z_{>0}$ and for any $\T$-equivariant filtration $\sG$ on $R$, 
let us define the $\T$-equivariant filtration 
$\sG^{(e)}$ on $R$ as 
\[
\sG^{(e),x}R_m:=\sG^{\lceil x/e\rceil}R_m
\]
for any $x\in\R$ and $m\in r\Z_{\geq 0}$. Then the test configuration 
$\left(\left(\sX_{(m)}\right)^{\overline{(e)}}, 
\left(\sL_{(m)}\right)^{\overline{(e)}}\right)/\A^1$ (see Definition 
\ref{definition:base-change} for the notation) 
is equal to the normalized blowup test configuration 
along $\sI_m\left(\sF^{(e)}\right)$. 
\end{enumerate}
\end{proposition}

\begin{proof}
\eqref{proposition:tc-twist1} is trivial from the definition of $\xi$-twisted 
test configurations. Let us show \eqref{proposition:tc-twist2}. For the morphism 
\begin{eqnarray*}
\pi_e\colon X_{\A^1}&\to& X_{\A^1}\\
(x,t)&\mapsto&(x,t^e), 
\end{eqnarray*}
we can directly check that $\pi_e^{-1}\sI_m(\sF)=\sI_m\left(\sF^{(e)}\right)$. 
Therefore, the assertion immediately follows by the universality of blowups. 
\end{proof}

\begin{proposition}\label{proposition:approx-base-change}
Let $\sF$ be a $\Z$-valued and $\T$-equivariant filtration on $R$, and let 
$\{\sF_{(m)}\}_{m\in r\Z_{>0}}$ be a \emph{$\Z$-valued} approximating sequence of 
$\sF$. For any $e\in\Z_{>0}$, let $\sF^{(e)}$ and 
$\sF_{(m)}^{(e)}:=\left(\sF_{(m)}\right)^{(e)}$ be as in 
Proposition \ref{proposition:tc-twist} \eqref{proposition:tc-twist2}. 
Then $\left\{\sF_{(m)}^{(e)}\right\}_{m\in r\Z_{>0}}$ is an approximating sequence 
of $\sF^{(e)}$. 
\end{proposition}

\begin{proof}
Observe the following: 
\begin{enumerate}
\renewcommand{\theenumi}{\roman{enumi}}
\renewcommand{\labelenumi}{(\theenumi)}
\item\label{proof:approx-base-change1}
For any $\lambda\in\Z$, we have 
\[
\sF_{(m)}^{(e),\lambda}R_{m'}=\sF_{(m)}^{\lceil\lambda/e\rceil}R_{m'}\subset
\sF^{\lceil\lambda/e\rceil}R_{m'}=\sF^{(e),\lambda}R_{m'}. 
\]
\item\label{proof:approx-base-change2}
For any $\lambda\in\Z$, we have 
\[
\sF_{(m)}^{(e),\lambda}R_m=\sF_{(m)}^{\lceil\lambda/e\rceil}R_m=
\sF^{\lceil\lambda/e\rceil}R_m=\sF^{(e),\lambda}R_m. 
\]
\item\label{proof:approx-base-change3}
Note that $\sF_{(m)}^{(e)}$ is $e\Z$-valued. For any $s\in\Z_{>0}$ and 
$\lambda\in e\Z$, we have 
\begin{eqnarray*}
&&\sF_{(m)}^{(e),\lambda}R_{m s}=\sF_{(m)}^{\lambda/e}R_{m s}
=\sum_{\substack{\mu_1+\cdots+\mu_s=\lambda/e; \\ \mu_1,\dots,\mu_s\in\Z}}
\sF^{\mu_1}R_m\cdots\sF^{\mu_s}R_m\\
&=&\sum_{\substack{\lambda_1+\cdots+\lambda_s=\lambda; \\ 
\lambda_1,\dots,\lambda_s\in e\Z}}\sF^{(e),\lambda_1}R_m\cdots
\sF^{(e),\lambda_s}R_m
=\sum_{\substack{\lambda_1+\cdots+\lambda_s\geq \lambda; \\ 
\lambda_1,\dots,\lambda_s\in \R}}\sF^{(e),\lambda_1}R_m\cdots
\sF^{(e),\lambda_s}R_m,
\end{eqnarray*}
where the second equality follows from the assumption $\sF$ is $\Z$-valued, 
and the fourth equality follows from the fact $\sF^{(e)}$ is $e\Z$-valued. 
\end{enumerate}
Thus the assertion follows. 
\end{proof}

We prepare the following lemmas: 

\begin{lemma}[{see \cite[Proof of Theorem 3.52]{Xu}}]\label{lemma:lct-compute}
Let $\sF$ be a $\T$-equivariant filtration on $R$, and let us take any 
$\delta\in\R_{>0}$. Set 
\[
\mu:=\mu(\sF;\delta):=\sup\left\{t\in\R\,\,|\,\,\lct\left(
X,\Delta;I_{\bullet}^{(t)}(\sF)\right)\geq \delta\right\}
\]
as in \cite[Definition 3.45]{Xu}, where $I_\bullet^{(t)}(\sF)$ is as in Lemma 
\ref{lemma:uniform-convergence}. 
If $\mu<\lambda_{\max}(\sF)$, then we have 
$\delta=\lct\left(X,\Delta;I_{\bullet}^{(\mu)}(\sF)\right)$. Moreover, 
there exists $v\in\Val_X^{<\infty,\T}$ such that 
\begin{itemize}
\item
we have 
\[
\delta=\frac{A_{X,\Delta}(v)}{v\left(I_\bullet^{(\mu)}(\sF)\right)}, 
\]
i.e., $v$ computes the log canonical threshold, and 
\item
satisfying that 
\[\sF\subset\sF_{v,\left[\mu-\frac{A_{X,\Delta}(v)}{\delta}\right]}.
\] 
\end{itemize}
\end{lemma}

\begin{proof}
The equality $\delta=\lct\left(X,\Delta;I_{\bullet}^{(\mu)}(\sF)\right)$ follows 
by the completely same argument in the proof of \cite[Lemma 3.46]{Xu}. 
Note that $I_{m;m\mu}(\sF)$ are $\T$-invariant ideals. By the equivariant version 
of \cite[Theorem 7.3]{JM} (see also \cite[Remark 3.9]{BLXZ}), there exists 
$v\in\Val_X^{<\infty,\T}$ such that computing the log canonical threshold. 
The rest of the proof is completely 
same as the proof of \cite[Theorem 3.52]{Xu}. 
\end{proof}

\begin{lemma}\label{lemma:prepare-e}
Let $\sF$ be a $\T$-equivariant filtration on $R$, and let us 
take any $e\in\Z_{>0}$.  
\begin{enumerate}
\renewcommand{\theenumi}{\arabic{enumi}}
\renewcommand{\labelenumi}{(\theenumi)}
\item\label{lemma:prepare-e1}
Consider the filtration $\sF^{(e)}$ as in Proposition \ref{proposition:tc-twist} 
\eqref{proposition:tc-twist2}. Then we have 
\[
S_L\left(\sF^{(e)}\right)=e\cdot S_L(\sF), \quad
\lambda_{\max}\left(\sF^{(e)}\right)=e\cdot \lambda_{\max}(\sF), \quad
\mu\left(\sF^{(e)}\right)=e\cdot \mu(\sF), 
\]
where $\mu(\bullet):=\mu(\bullet;1)$ (see Lemma \ref{lemma:lct-compute}). 
\item\label{lemma:prepare-e2}
Let us take $\xi\in N_\R(\T)$. 
Then we have 
\begin{eqnarray*}
S_L\left(\left(\sF^{(e)}\right)_{e\xi}\right)&=&e\cdot S_L\left(\sF_\xi\right),\\
\lambda_{\max}\left(\left(\sF^{(e)}\right)_{e\xi}\right)
&=&e\cdot \lambda_{\max}\left(\sF_\xi\right),\\
\mu\left(\left(\sF^{(e)}\right)_{e\xi}\right)
&=&e\cdot \mu\left(\sF_\xi\right).
\end{eqnarray*}
\end{enumerate}
\end{lemma}

\begin{proof}
\eqref{lemma:prepare-e1} is trivial. We see \eqref{lemma:prepare-e2}. 
For any $x\in\R$, $m\in r\Z_{\geq 0}$ and $\alpha\in M(\T)$, we have 
\begin{eqnarray*}
\left(\sF^{(e)}\right)_{e\xi}^x R_{m,\alpha}
&=&\sF^{\lceil(x-\langle\alpha,e\xi\rangle)/e\rceil}R_{m,\alpha}
=\sF^{\lceil x/e-\langle\alpha,\xi\rangle\rceil}R_{m,\alpha}, \\
\left(\left(\sF_\xi\right)^{(e)}\right)^x R_{m,\alpha}
&=&\sF_\xi^{\lceil x/e\rceil}R_{m,\alpha}=\sF^{\lceil x/e\rceil
-\langle\alpha,\xi\rangle}R_{m,\alpha}.
\end{eqnarray*}
This immediately implies that 
\begin{eqnarray*}
S_L\left(\left(\sF^{(e)}\right)_{e\xi}\right)
&=&S_L\left(\left(\sF_\xi\right)^{(e)}\right),\\
\lambda_{\max}\left(\left(\sF^{(e)}\right)_{e\xi}\right)
&=&\lambda_{\max}\left(\left(\sF_\xi\right)^{(e)}\right),\\
\mu\left(\left(\sF^{(e)}\right)_{e\xi}\right)
&=&\mu\left(\left(\sF_\xi\right)^{(e)}\right).
\end{eqnarray*}
Together with \eqref{lemma:prepare-e1}, we get the assertion 
\eqref{lemma:prepare-e2}.
\end{proof}

\section{The sum of test configurations}\label{section:sum-tc}

In this section, we introduce the notion of the sum of test configurations and 
see its basic properties. In this section, we fix an $n$-dimensional projective klt pair 
$(X,\Delta)$, an algebraic torus $\T\simeq\G_m^p$ with $p\in\Z_{\geq 0}$, 
an injection $\T\to\Aut(X,\Delta)$, and $\T$-linearized ample $\Q$-line bundles 
$L_1,\dots,L_k$ on $X$. Fix $r\in\Z_{>0}$ such that each $r L_i$ is a 
$\T$-linearized ample line bundle on $X$. We set 
\[
R^i:=\bigoplus_{m\in r\Z_{\geq 0}}R_m^i:=\bigoplus_{m\in r\Z_{\geq 0}}H^0(X, m L_i). 
\]
Moreover, we set $L:=\sum_{i=1}^k L_i$ with the natural $\T$-linearization, and set
\[
R:=\bigoplus_{m\in r\Z_{\geq 0}}R_m:=\bigoplus_{m\in r\Z_{\geq 0}}H^0(X, m L).
\]
After replacing $r\in\Z_{>0}$ if necessary, we may assume that the multiplication 
homomorphism 
\[
\mult_m\colon \tilde{R}_m\to R_m
\]
is surjective for any $m\in r\Z_{\geq 0}$,
where we set 
\[
\tilde{R}_m:=R_m^1\otimes_{\Bbbk}\cdots\otimes_{\Bbbk}R_m^k,\quad
\tilde{R}:=\bigoplus_{m\in r\Z_{\geq 0}}\tilde{R}_m.
\]

\begin{lemma}\label{lemma:sum-basics}
\begin{enumerate}
\renewcommand{\theenumi}{\arabic{enumi}}
\renewcommand{\labelenumi}{(\theenumi)}
\item\label{lemma:sum-basics1}
Let $\PP^{L_i}\subset M_\R(\T)$ (resp., $\PP^L\subset M_\R(\T)$) be 
the moment polytope associated with $(X, L_i)$ (resp., $(X, L)$). Then 
$\PP^L$ is the Minkowski sum of $\PP^{L_1},\dots,\PP^{L_k}$. 
\item\label{lemma:sum-basics2}
For any $1\leq i\leq k$, let $L'_i$ be another $\T$-linearized ample $\Q$-line 
bundle on $X$ with $L_i\sim_\Q L'_i$ such that 
$L=\sum_{i=1}^k L'_i$ holds as $\T$-linearized $\Q$-line bundles. Then 
we have
\[
\sum_{i=1}^k\alpha_{\bc}^{L_i}=\sum_{i=1}^k\alpha_{\bc}^{L'_i}.
\]
\item\label{lemma:sum-basics3}
Take any $\xi\in N_\R(\T)$ and $v\in\Val_X^{<\infty,\T}$. Then we have 
the equality 
\[
\theta_\xi^{L}(v)=\sum_{i=1}^k\theta_\xi^{L_i}(v). 
\]
\item\label{lemma:sum-basics4}
Take any proper subtorus $\T'\subsetneq\T$. Then, for any 
$\xi\in N_\Q(\T)\setminus N_\Q(\T')$, we have 
\[
\inf_{\xi'\in N_\Q(\T')}\left\{
\sum_{i=1}^k \JJ\left(X_{\xi+\xi'},(L_i)_{\xi+\xi'}\right)\right\}>0.
\]
\end{enumerate}
\end{lemma}

\begin{proof}
\eqref{lemma:sum-basics1}
For any $m\in r\Z_{>0}$, we have 
\[
\Lambda_m^L=\sum_{i=1}^k\Lambda_m^{L_i}, 
\]
where the right hand side is the Minkowski sum. Thus we get the assertion 
\eqref{lemma:sum-basics1}. 

\eqref{lemma:sum-basics2}
Since $L_i\sim_\Q L'_i$, there exists $\beta_i\in M_\Q(\T)$ such that 
$\PP^{L'_i}=\PP^{L_i}+\beta_i$ and $\alpha_{\bc}^{L'_i}=\alpha_{\bc}^{L_i}+\beta_i$. 
Since 
\[
\sum_{i=1}^k\PP^{L_i}=\PP^L=\sum_{i=1}^k\PP^{L'_i}
=\left(\sum_{i=1}^k\PP^{L_i}\right)+\sum_{i=1}^k\beta_i, 
\]
we get $\sum_{i=1}^k\beta_i=0$. Thus we get 
\[
\sum_{i=1}^k\alpha_{\bc}^{L'_i}=\sum_{i=1}^k\left(\alpha_{\bc}^{L_i}+
\beta_i\right)=\sum_{i=1}^k\alpha_{\bc}^{L_i}. 
\]

\eqref{lemma:sum-basics3}
The value $\theta_\xi^L(v)$ is continuous over $\xi\in N_\R(\T)$. Thus we may 
assume that $\xi\in N_\Q(\T)$. Since $\theta_{e\xi}^L(e v)=e\theta_\xi^L(v)$ 
for any $e\in\R_{>0}$, we may further assume that $\xi\in N(\T)$. 
Then the assertion immediately follows by Corollary \ref{corollary:theta}. 

\eqref{lemma:sum-basics4}
Set 
\[
C_1:=\dist\left(\sum_{i=1}^k\alpha_{\bc}^{L_i},\partial\left(\PP^L\right)\right), \quad
C_2:=\dist\left(0,\xi+N_\R(\T')\right).
\]
By \eqref{lemma:sum-basics1}, we have $\sum_{i=1}^k\alpha_{\bc}^{L_i}\in
\interior(\PP^L)$. Thus $C_1\in\R_{>0}$. Moreover, since $\xi\not\in N_\R(\T')$, 
we also have $C_2\in\R_{>0}$. 
By Lemma \ref{lemma:J-xi}, we have 
\[
\JJ\left(X_{\xi+\xi'},(L_i)_{\xi+\xi'}\right)=\max_{\alpha_i\in\PP^{L_i}}
\left\{\langle\alpha_i,\xi+\xi'\rangle\right\}-\langle\alpha_{\bc}^{L_i},\xi+\xi'\rangle.
\]
Thus we have
\[
\sum_{i=1}^k \JJ\left(X_{\xi+\xi'},(L_i)_{\xi+\xi'}\right)
\geq\max_{\alpha\in\PP^L}\left\{\langle\alpha,\xi+\xi'\rangle\right\}-\left\langle
\sum_{i=1}^k\alpha_{\bc}^{L_i},\xi+\xi'\right\rangle
\geq C_1 C_2. 
\]
Thus we get the assertion \eqref{lemma:sum-basics4}.
\end{proof}

We define the sum of test configurations. 

\begin{definition}\label{definition:sum-tc}
For any $1\leq i\leq k$, take any (possibly non-normal) $\T$-equivariant 
test configuration $\pi_i\colon(\sX_i,\sL_i)\to\A_t^1$ of $(X,L_i)$. 
Take a sufficiently divisible $r_0\in r{\Z_{>0}}$. Set 
\[
\sR^i:=\bigoplus_{m\in r_0\Z_{\geq 0}}\sR_m^i:=\bigoplus_{m\in r_0\Z_{\geq 0}}
H^0\left(\sX_i,m\sL_i\right), 
\]
and let
\[
\sR_m^i=\bigoplus_{\lambda\in\Z}t^{-\lambda}\sF_i^\lambda R_m^i
\]
be the weight decomposition with respects to the test configuration, i.e., 
$\sF_i=\sF_{\sX_i,\sL_i}$. As in \cite[Proposition 2.15]{BHJ}, we have 
\[
\left(\sX_i,\sL_i\right)=\left(\Proj_{\Bbbk[t]}\left(\sR^i\right),\sO(1)\right). 
\]
Set 
\[
I^{\langle i\rangle}_{(m,\alpha;\lambda)}:=
I_{(m,\alpha;\lambda)}(\sF_i), \quad
I^{\langle i\rangle}_{(m;\lambda)}:=
I_{(m;\lambda)}(\sF_i)=\sum_{\alpha\in M(\T)}I^{\langle i\rangle}_{(m,\alpha;\lambda)}.
\]

Let us take $(\G_m\times\T)$-equivariant common partial resolutions
\[
\sigma_i\colon\sZ\to\sX_i
\]
with $\sZ$ normal and $\sigma_i$ an isomorphism over $\A^1\setminus\{0\}$. 
Then, 
\[
\sigma_i^*\sR^i=\bigoplus_{m\in r_0\Z_{\geq 0}}\sigma_i^*\sR^i_m,
\]
where $\sigma_i^*\sR^i_m$ is the image of 
\[
\sigma_i^*\colon H^0\left(\sX_i,m\sL_i\right)\to 
H^0\left(\sZ,\sigma_i^* m\sL_i\right),
\]
is canonically isomorphic to $\sR^i$. We define 
\begin{eqnarray*}
\sR_m&:=&\Image\Big(\sigma_1^*\sR_m^1\otimes_{\Bbbk[t]}\cdots
\otimes_{\Bbbk[t]}\sigma_k^*\sR_m^k\\
&\hookrightarrow&H^0\left(\sZ,m\sigma^*_1\sL_1\right)\otimes_{\Bbbk[t]}\cdots
\otimes_{\Bbbk[t]}H^0\left(\sZ,m\sigma^*_k\sL_k\right)\\
&\to&H^0\left(\sZ,m\left(\sigma^*_1\sL_1+\cdots+\sigma^*_k\sL_k\right)\right)
\Big)
\end{eqnarray*}
and 
\[
\sR:=\bigoplus_{m\in r_0\Z_{\geq 0}}\sR_m\subset
\bigoplus_{m\in r_0\Z_{\geq 0}}H^0\left(
\sZ,m\left(\sigma^*_1\sL_1+\cdots+\sigma^*_k\sL_k\right)\right).
\]
By the following Lemma \ref{lemma:sum-tc}, the above $\sR$ is 
a finitely generated $\Bbbk[t]$-algebra with the natural $(\G_m\times\T)$-action. 
As in \cite[Proposition 2.15]{BHJ}, this $\sR$ induces a $\T$-equivariant 
test configuration $\pi\colon(\sX,\sL)\to\A^1_t$ of $(X, L)$. 
The $(\sX,\sL)/\A^1$ is said to be the \emph{sum configuration} of 
$\left\{(\sX_i,\sL_i)\right\}_{i=1}^k$. Moreover, let $\nu\colon \sX^\nu\to\sX$ be 
the normalization. Then the $\T$-equivariant normal test configuration 
$(\sX^\nu, \nu^*\sL)/\A^1$ is said to be the 
\emph{normalized sum configuration} of $\left\{(\sX_i,\sL_i)\right\}_{i=1}^k$. 
\end{definition}

\begin{lemma}\label{lemma:sum-tc}
\begin{enumerate}
\renewcommand{\theenumi}{\arabic{enumi}}
\renewcommand{\labelenumi}{(\theenumi)}
\item\label{lemma:sum-tc1}
The above $\sR$ is a finitely generated $\Bbbk[t]$-algebra such that 
each $\sR_m$ admits the weight decomposition 
\[
\sR_m=\bigoplus_{\lambda\in\Z}t^{-\lambda}\bigoplus_{\alpha\in M(\T)}
\sF^{\lambda}R_{m,\alpha}
\]
with 
\[
\sF^{\lambda}R_{m,\alpha}=\sum_{\substack{\lambda_1,\dots,\lambda_k\in\Z;\\ 
\lambda_1+\cdots+\lambda_k=\lambda}}
\sum_{\substack{\alpha_1,\dots,\alpha_k\in M(\T); \\
\alpha_1+\cdots+\alpha_k=\alpha}}
\sF_1^{\lambda_1}R_{m,\alpha_1}^1\cdots\sF_k^{\lambda_k}R_{m,\alpha_k}^k
\]
for any $\lambda\in\Z$ and $\alpha\in M(\T)$. 
\item\label{lemma:sum-tc2}
The definitions of the sum configurations and the normalized sum configurations 
do not depend on the choice of $\sZ$. 
\item\label{lemma:sum-tc3}
The normalized sum configuration of $\left\{(\sX_i,\sL_i)\right\}_{i=1}^k$ is 
nothing but the ample model of 
$\left(\sZ,\sum_{i=1}^k\sigma_i^*\sL_i\right)$ over $\A^1$. 
\item\label{lemma:sum-tc4}
Let us set 
\begin{eqnarray*}
\sX_{\PROD}&:=&\sX_1\times_{\A^1}\cdots\times_{\A^1}\sX_k\to\A^1, \\
\sL_{\PROD}&:=&\sL_1\boxtimes\cdots\boxtimes\sL_k.
\end{eqnarray*}
Then there is a canonical $(\G_m\times\T)$-equivariant closed embedding 
\[
\iota^{\cp}\colon\sX\hookrightarrow\sX_{\PROD}
\]
over $\A^1$ such that $\sL_{\PROD}|_{\sX}\simeq\sL$ holds. In particular, 
for any $1\leq i\leq k$, there is a canonical $(\G_m\times\T)$-equivariant 
birational morphism $\tau_i\colon\sX\to\sX_i$ over $\A^1$. 
\end{enumerate}
\end{lemma}

\begin{proof}
By the K\"uneth formula, we have
\begin{eqnarray*}
&&H^0\left(\sX_{\PROD},m\sL_{\PROD}\right)
\simeq H^0\left(\sX_1,m\sL_1\right)\otimes_{\Bbbk[t]}\cdots\otimes_{\Bbbk[t]}
H^0\left(\sX_k,m\sL_k\right)\\
&\simeq&\left(\bigoplus_{\lambda_1\in\Z}\bigoplus_{\alpha_1\in M(\T)}
t^{-\lambda_1}\sF_1^{\lambda_1}R_{m,\alpha_1}^1\right)
\otimes_{\Bbbk[t]}\cdots\otimes_{\Bbbk[t]}
\left(\bigoplus_{\lambda_k\in\Z}\bigoplus_{\alpha_k\in M(\T)}
t^{-\lambda_k}\sF_k^{\lambda_k}R_{m,\alpha_k}^k\right).
\end{eqnarray*}
Note that $\sR$ is obtained by the natural surjection 
\[
\bigoplus_{m\in r_0\Z_{\geq 0}}H^0\left(\sX_{\PROD},m\sL_{\PROD}\right)
\twoheadrightarrow\sR
\]
and we have 
\[
\sR_m=\bigoplus_{\lambda\in\Z}\bigoplus_{\alpha\in M(\T)}t^{-\lambda}
\sum_{\substack{\lambda_1,\dots,\lambda_k\in\Z; \\
\lambda_1+\cdots+\lambda_k=\lambda}}\sum_{\substack{
\alpha_1,\dots,\alpha_k\in M(\T); \\ \alpha_1+\cdots+\alpha_k=\alpha}}
\sF_1^{\lambda_1}R_{m,\alpha_1}^1\cdots\sF_k^{\lambda_k}R_{m,\alpha_k}^k.
\]
Thus the assertion \eqref{lemma:sum-tc1} follows. Moreover, from the above 
$(\G_m\times\T)$-weight decomposition, the $(\sX,\sL)/\A^1$ gives a 
$\T$-equivariant test configuration of $(X, L)$. The assertion 
\eqref{lemma:sum-tc4} is trivial from the above construction. 

The assertion \eqref{lemma:sum-tc2} is trivial since $\sigma^*_i$ is injective. For 
\eqref{lemma:sum-tc3}, we may assume that there exists a morphism 
$\sigma\colon\sZ\to\sX$ after replacing $\sZ$ if necessary. Then we have 
\[
\sigma^*\sL\sim_\Q\sigma_1^*\sL_1+\cdots+\sigma_k^*\sL_k.
\]
Thus we get the assertion \eqref{lemma:sum-tc3}. 
\end{proof}

\begin{remark}\label{remark:appendix}
\begin{enumerate}
\renewcommand{\theenumi}{\arabic{enumi}}
\renewcommand{\labelenumi}{(\theenumi)}
\item\label{remark:appendix1}
Even when $\T$ is trivial (i.e., $p=0$), $X$ is a Fano manifold and $L=-K_X$, 
the sum configuration of normal test configurations may not be normal in general. 
See \S \ref{section:appendix} for examples. 
\item\label{remark:appendix2}
The sum configuration of $\left\{(\sX_i,\sL_i)\right\}_{i=1}^k$ coincides with the 
test configuration generated by the $\G_m$-action of 
$\left\{(\sX_i,\sL_i)\right\}_{i=1}^k$ in the sense of \cite[Definition 18]{hashimoto}. 
See Proposition \ref{proposition:hashimoto}.
\end{enumerate}
\end{remark}

\begin{lemma}\label{lemma:sum-tc-ideal}
For any $1\leq i\leq k$, let $(\sX_i,\sL_i)/\A^1$ be a $\T$-equivariant 
test configuration of $(X,L_i)$. Let $(\sX,\sL)/\A^1$ be the sum configuration of 
$\left\{(\sX_i,\sL_i)\right\}_{i=1}^k$. We set 
$\sF_i:=\sF_{\sX_i,\sL_i}$ and $\sF:=\sF_{\sX,\sL}$. Moreover, set 
\[
I_{(m,\alpha;\lambda)}:=I_{(m,\alpha;\lambda)}(\sF), \quad
I_{(m;\lambda)}:=I_{(m;\lambda)}(\sF).
\]
Then we have 
\begin{eqnarray*}
\sF^{\lambda}R_{m,\alpha}&=&
\sum_{\substack{\lambda_1,\dots,\lambda_k\in\Z; \\
\lambda_1+\cdots+\lambda_k=\lambda}}\sum_{\substack{
\alpha_1,\dots,\alpha_k\in M(\T); \\ \alpha_1+\cdots+\alpha_k=\alpha}}
\sF_1^{\lambda_1}R_{m,\alpha_1}^1\cdots\sF_k^{\lambda_k}R_{m,\alpha_k}^k, \\
\sF^{\lambda}R_m&=&\sum_{\substack{\lambda_1,\dots,\lambda_k\in\Z; \\
\lambda_1+\cdots+\lambda_k=\lambda}}\sF_1^{\lambda_1}R_m^1\cdots
\sF_k^{\lambda_k}R_m^k, \\
I_{(m,\alpha;\lambda)}&=&
\sum_{\substack{\lambda_1,\dots,\lambda_k\in\Z; \\
\lambda_1+\cdots+\lambda_k=\lambda}}\sum_{\substack{
\alpha_1,\dots,\alpha_k\in M(\T); \\ \alpha_1+\cdots+\alpha_k=\alpha}}
I_{(m,\alpha_1;\lambda_1)}^{\langle 1\rangle}\cdots
I_{(m,\alpha_k;\lambda_k)}^{\langle k\rangle}, \\
I_{(m;\lambda)}&=&\sum_{\substack{\lambda_1,\dots,\lambda_k\in\Z; \\
\lambda_1+\cdots+\lambda_k=\lambda}}
I_{(m;\lambda_1)}^{\langle 1\rangle}\cdots
I_{(m;\lambda_k)}^{\langle k\rangle}
\end{eqnarray*}
for any $m\in r_0\Z_{\geq 0}$, $\lambda\in\Z$ and $\alpha\in M(\T)$. 
\end{lemma}

\begin{proof}
Trivial from the proof of Lemma \ref{lemma:sum-tc}. 
\end{proof}

The following lemma is a special case of Proposition \ref{proposition:sum-filt-twist}. 
For the proof, see the proof of Proposition \ref{proposition:sum-filt-twist}. 

\begin{lemma}\label{lemma:sum-tc-twist}
Let $(\sX_i,\sL_i)/\A^1$ and $(\sX,\sL)/\A^1$ be as in 
Lemma \ref{lemma:sum-tc-ideal}. Take any $\xi\in N(\T)$ and consider 
the $\xi$-twisted test configuration $(\sX_{i,\xi},\sL_{i,\xi})/\A^1$
of $(\sX_i,\sL_i)/\A^1$ (see Example \ref{example:twisted-tc}). Then the sum 
configuration of $\left\{(\sX_{i,\xi},\sL_{i,\xi})\right\}_{i=1}^k$ is equal to 
the $\xi$-twisted test configuration $(\sX_\xi,\sL_\xi)/\A^1$ of $(\sX,\sL)/\A^1$.
\end{lemma}

\begin{proposition}\label{lproposition:sum-Di}
Under the assumptions in Lemma \ref{lemma:sum-tc-ideal}, assume moreover that 
each $\sX_i$ is normal. Let $\nu\colon\sY\to\sX$ be the normalization and let 
\[\xymatrix{
&\sZ \ar[ld]_\theta \ar[rd]^\rho&\\
X_{\A^1}\ar@{-->}[rr]&&\sY
}\]
be a $(\G_m\times\T)$-equivariant partial resolution with $\sZ$ normal and 
$\theta$ an isomorphism over $\A^1\setminus\{0\}$. Let us set 
\[
\chi_i:=\tau_i\circ\nu\circ\rho\colon\sZ\to\sX_i, 
\]
where $\tau_i$ be as in Lemma \ref{lemma:sum-tc}. Set 
\[
\sD_i:=\theta^*\left((L_i)_{\A^1}\right)-\chi_i^*\sL_i
\]
supported on the fiber over $0\in\A^1$. Then, for any sufficiently divisible 
$m\in r_0\Z_{>0}$, the inverse image of the fractional ideal 
\[
\bigoplus_{\lambda\in\Z}t^{-\lambda}I_{(m;\lambda)}^{\langle i\rangle}
\]
by $\theta^*$ is equal to $\sO_{\sZ}\left(-m\sD_i\right)$. In particular, 
the inverse image of the fractional ideal 
\[
\bigoplus_{\lambda\in\Z}t^{-\lambda}I_{(m;\lambda)}
\]
by $\theta^*$ is equal to $\sO_{\sZ}\left(-m\sum_{i=1}^k \sD_i\right)$. 
\end{proposition}

\begin{proof}
Since $\chi_i^*\sL_i$ is semiample over $\A^1$, we have a surjection 
\[
H^0\left(\sZ, \chi_i^*(m\sL_i)\right)\otimes_{\Bbbk[t]}\sO_{\sZ}
\twoheadrightarrow\chi_i^*\sO_{\sX_i}(m\sL_i). 
\]
Since $\sX_i$ is normal, we have 
\[
H^0\left(\sZ, \chi_i^*(m\sL_i)\right)=\chi_i^*H^0\left(\sX_i, m\sL_i\right)
=\chi_i^*\left(\bigoplus_{\lambda\in\Z}t^{-\lambda}\sF_i^\lambda R_m^i\right). 
\]
Therefore, we get the surjection 
\[
\theta^*\left(\bigoplus_{\lambda\in\Z}t^{-\lambda}
\left(\sF_i^\lambda R_m^i\otimes_{\Bbbk}\sO_X(-m L_i)\right)\right)
\twoheadrightarrow\sO_{\sZ}(-m\sD_i). 
\]
Since 
\[
\sF_i^\lambda R_m^i\otimes_{\Bbbk}\sO_X(-m L_i)\twoheadrightarrow 
I_{(m;\lambda)}^{\langle i\rangle}, 
\]
the inverse image of $\bigoplus_{\lambda\in\Z}t^{-\lambda}
I_{(m;\lambda)}^{\langle i\rangle}$ is equal to $\sO_{\sZ}(-m\sD_i)$. 
Since 
\[
\prod_{i=1}^k\left(\bigoplus_{\lambda\in\Z}t^{-\lambda}
I_{(m;\lambda)}^{\langle i\rangle}\right)
=\bigoplus_{\lambda\in\Z}t^{-\lambda}\sum_{\substack{
\lambda_1,\dots,\lambda_k\in\Z;\\ \lambda_1+\cdots+\lambda_k=\lambda}}
I_{(m;\lambda_1)}^{\langle 1\rangle}\cdots I_{(m;\lambda_k)}^{\langle k\rangle}
=\bigoplus_{\lambda\in\Z}t^{-\lambda}I_{(m;\lambda)}, 
\]
we get the assertion. 
\end{proof}

\begin{corollary}\label{corollary:trivial-tc}
For any $1\leq i\leq k$, let $(\sX_i,\sL_i)/\A^1$ be a $\T$-equivariant 
normal test configuration of $(X,L_i)$. The following are equivalent: 
\begin{enumerate}
\renewcommand{\theenumi}{\arabic{enumi}}
\renewcommand{\labelenumi}{(\theenumi)}
\item\label{corollary:trivial-tc1}
For all $1\leq i\leq k$, the test configurations $(\sX_i,\sL_i)/\A^1$ are shifts of the 
trivial test configurations of $(X, L_i)$. 
\item\label{corollary:trivial-tc2}
The sum configuration of $\left\{(\sX_i,\sL_i)\right\}_{i=1}^k$ is a shift of the 
trivial test configuration of $(X, L)$. 
\item\label{corollary:trivial-tc3}
The normalized 
sum configuration of $\left\{(\sX_i,\sL_i)\right\}_{i=1}^k$ is a shift of the 
trivial test configuration of $(X, L)$. 
\end{enumerate}
\end{corollary}

\begin{proof}
The assertions \eqref{corollary:trivial-tc1} $\implies$ \eqref{corollary:trivial-tc2} 
and \eqref{corollary:trivial-tc2} $\implies$ \eqref{corollary:trivial-tc3} are 
trivial. We show the assertion 
\eqref{corollary:trivial-tc3} $\implies$ \eqref{corollary:trivial-tc1}. 
We use the notation in Proposition \ref{lproposition:sum-Di}. 
By the assumption, we have $\sY=X_{\A^1}$ and 
$\nu^*\left(\tau_1^*\sL_1+\cdots+\tau_k^*\sL_k\right)=L_{\A^1}
-C\cdot(X\times\{0\})$ for some $C\in\Q$. 
Thus we have $\sum_{i=1}^k\sD_i\sim_{\Q,\A^1}0$. This implies that 
$\sum_{i=1}^k\chi_i^*\sL_i\sim_{\Q, X_{\A^1}}0$. Since $\chi_i^*\sL_i$ is nef 
over $X_{\A^1}$, we have $\chi_i^*\sL_i\equiv_{X_{\A^1}}0$ for any 
$1\leq i\leq k$ by the negativity lemma. 
Thus the ample model of $(\sZ,\chi_i^*\sL_i)$ over $X_{\A^1}$ is $X_{\A^1}$ 
itself. This implies that $\sD_i\sim_{\Q,\A^1}0$. This immediately implies that the 
condition \eqref{corollary:trivial-tc1}. 
\end{proof}

From now on, for any test configuration $(\sX_i,\sL_i)/\A^1$ of $(X, L_i)$, 
we see that 
the sum configuration of $\left\{(\sX_i,\sL_i)\right\}_{i=1}^k$ coincides with the 
test configuration generated by the $\G_m$-actions of 
$\left\{(\sX_i,\sL_i)\right\}_{i=1}^k$ in the sense of \cite[Definition 18]{hashimoto}. 
We will use the following results only in \S \ref{section:analytic}. 
We set $r_0\in r\Z_{>0}$ to be sufficiently divisible so that each $r_0\sL_i$ is 
$\pi_i$-very ample over $\A^1$ and the $\G_m$-action linearizes to $\sL_i$. Let $\sR^i$, $\sR$, $\sF_i$ and $\sF$ 
be as in Definition \ref{definition:sum-tc} and Lemma \ref{lemma:sum-tc}. 
As in \cite[\S 2.3]{BHJ}, for any $m\in r_0\Z_{>0}$ and $1\leq i\leq k$, 
the test configuration $\left(\sX_i,m\sL_i\right)/\A^1$ of $(X, m L_i)$ gives a 
one-parameter subgroup $\rho_m^i\colon \G_m\to\mathrm{GL}\left(R_m^i\right)$.
We recall \cite[Definition 18]{hashimoto}: 

\begin{definition}[{\cite[Definition 18]{hashimoto}}]\label{definition:hashimoto}
Let us consider the composition of the natural embeddings 
\[
\iota_m^{\cp}\colon X\hookrightarrow\pr^*(R_m)\hookrightarrow
\pr^*\left(\tilde{R}_m\right).
\]
(Note that the projectivizations $\pr^*$ are in the sense of 
Grothendieck.) We say that a test configuration $(\sY,\sL_{\sY})/\A^1$ of 
$(X, L)$ is \emph{generated by the $\G_m$-action of 
$\left\{(\sX_i,\sL_i)\right\}_{i=1}^k$}, if $\sY$ is defined as the reduced 
Zariski closure of the $\G_m$-orbit of $\iota_m^{\cp}(X)$ in 
$\pr^*(R_m)\times\A^1$, 
under the natural tensor product $\G_m$-action $\rho_m^1\otimes\cdots\otimes
\rho_m^k$ on 
$\tilde{R}_m=R_m^1\otimes\cdots\otimes R_m^k$, and 
$\sL_{\sY}:=\frac{1}{m}\sO_{\pr^*\left(\tilde{R}_m\right)\times\A^1}(1)|_{\sY}$. 
\end{definition}

There is a minor difference to the original definition \cite[Definition 18]{hashimoto}, 
where it was defined for a $k$-tuple of very ample test configurations 
with exponent $m$, but in this paper we re-scale the polarization by $1/m$. 

\begin{lemma} \label{lemma:lmrmivssp}
Suppose that, for any $1 \le i \le k$ and any $m \in r_0\mathbb{Z}_{>0}$, 
$\G_m$ acts on $R^i_m$ according to the weight decomposition
\[
R^i_m=\bigoplus_{\lambda_i\in\Z} \mathcal{F}^{\lambda_i}_i R^i_m,
\]
which naturally induces the tensor product action
\[
\G_m \curvearrowright R^1_m\otimes\cdots\otimes R^k_m=\tilde{R}_m.
\]
Suppose furthermore that $\G_m$ acts on $R_m$ according to the weight 
decomposition $R_m=\bigoplus_{\lambda\in\Z}\sF^\lambda R_m$ with
\[
\sF^\lambda R_m=\sum_{\substack{\lambda_1,\dots,\lambda_k\in\Z;\\ 
\lambda_1+\cdots+\lambda_k=\lambda}}\sF^{\lambda_1}_1 R_m^1\cdots
\sF^{\lambda_k}_k R_m^k.
\]
Then the multiplication homomorphism $\mult_m$ is $\G_m$-equivariant.
\end{lemma}

\begin{proof}
The weight decomposition of the $\G_m$-action on $\tilde{R}_m$ induced from the 
tensor product action is given by $\tilde{R}_{m} = \bigoplus_{\lambda\in\Z} 
\tilde{\sF}^\lambda\tilde{R}_m$, where
\[
\tilde{\sF}^\lambda\tilde{R}_m:= \bigoplus_{\substack{\lambda_1,\dots,\lambda_k\in\Z;\\ 
\lambda_1+\cdots+\lambda_k=\lambda}} \sF^{\lambda_1}_1 R^1_m\otimes\cdots 
\otimes\sF^{\lambda_k}_k R^k_m.
\]
Thus the homomorphism $\mult_m$ is obviously $\G_m$-equivariant. 
\end{proof}

\begin{proposition} \label{proposition:hashimoto}
Suppose that $(\sY, \sL_{\sY})/\A^1$ is the test configuration 
generated by the $\G_m$-actions of $\left\{(\sX_i,\sL_i)\right\}_{i=1}^k$ in the sense 
of Definition \ref{definition:hashimoto} obtained by $m\in r_0\Z_{>0}$. 
Then the test configuration $(\sY, \sL_{\sY})/\A^1$ agrees with the sum 
configuration of 
$\left\{(\sX_i,\sL_i)\right\}_{i=1}^k$ in the sense of Definition \ref{definition:sum-tc}.
\end{proposition}

\begin{proof}
By Lemma \ref{lemma:lmrmivssp}, the $\G_m$-orbit of $X$ in 
$\pr^*(\tilde{R}_m)\times\A^1$ is contained in the linear subspace 
$\pr^*(R_m)\times\A^1$, and hence the Zariski closure of its $\G_m$-orbit in 
$\pr^*(\tilde{R}_m)\times\A^1$ agrees with the one in $\pr^*(R_m)\times\A^1$. 
Thus $\sY$ agrees with the Zariski closure of the $\G_m$-orbit of 
$X\times\G_m \hookrightarrow\pr^*(R_m)\times\A^1$ with respect to 
the naturally induced $\G_m$-action on $R_m$, which corresponds to 
the weight decomposition in Lemma \ref{lemma:lmrmivssp}. 
It is well-known \cite[\S 1.2, 2.3, and 2.5]{BHJ} that such a test configuration can 
also be written as $\left(\Proj_{\Bbbk[t]}(\sR),\frac{1}{m}\sO(1)\right)$. Moreover, since the embedding $\pr^*(R_m)\hookrightarrow\pr^*\left(\tilde{R}_m\right)$ 
is linear, we find that the natural polarization $\sO(1)$ on 
$\Proj_{\Bbbk [t]}(\sR)$ agrees with $m\sL_{\sY}$ by the definition of 
$\sL_{\sY}$. Recalling that the sum configuration of 
$\left\{(\sX_i,\sL_i)\right\}_{i=1}^k$ can be written as 
$\left(\Proj_{\Bbbk[t]}(\sR), \frac{1}{m}\sO(1)\right)$, 
by Definition \ref{definition:sum-tc} and Lemma \ref{lemma:sum-tc}, we get the 
claimed result.
\end{proof}

\begin{remark}\label{remark:hsmt}
As proved in Corollary \ref{corollary:trivial-tc}, the sum test configuration is trivial 
if and only if all the constituents $\left\{(\sX_i,\sL_i)\right\}_{i=1}^k$ are trivial 
(and normal). For the test configuration $(\sY, \sL_{\sY})$ generated by their 
$\G_m$-actions, this point was claimed in \cite[\S 3.4]{hashimoto} 
without detailed arguments. We provide the details omitted therein: 
it is easy to prove that the $\G_m$-action on $\tilde{R}_m$ is trivial if and only if 
all the constituents are all trivial test configurations, 
and Lemma \ref{lemma:lmrmivssp} shows that it happens if and only if the 
$\G_m$-action on $R_m$ is trivial since there does not exist any weight 
$\lambda$ such that $\mult_m(\tilde{\sF}^\lambda\tilde{R}_{m})=0 $ if 
$\tilde{\sF}^\lambda\tilde{R}_m\neq 0$, by the property of the multiplication 
homomorphism.
\end{remark}

%
%

\section{The sum of filtrations}\label{section:sum-filt}

In this section, we introduce the notion of the sum of filtrations, which is a 
generalization of the notion of sum configurations, and see its basic properties. 
In this section, we follow the notation in \S \ref{section:sum-tc}. 

\begin{definition}\label{definition:sum-filt}
For any $1\leq i\leq k$, let $\sF_i$ be a $\T$-equivariant filtration on $R^i$. 
The \emph{sum filtration} of $\left\{\sF_i\right\}_{i=1}^k$ is the $\T$-equivariant 
filtration $\sF$ on $R$ defined by 
\[
\sF^x R_m:=\sum_{\substack{x_1,\dots,x_k\in\R; \\ 
x_1+\cdots+x_k\geq x}}\sF_1^{x_1}R_m^1\cdots\sF_k^{x_k}R_m^k
=\sum_{\substack{x_1,\dots,x_k\in\R; \\ 
x_1+\cdots+x_k= x}}\sF_1^{x_1}R_m^1\cdots\sF_k^{x_k}R_m^k
\]
for any $m\in r\Z_{\geq 0}$ and $x\in\R$ (see Lemma \ref{lemma:sum-filt}). 
\end{definition}

\begin{lemma}\label{lemma:sum-filt}
\begin{enumerate}
\renewcommand{\theenumi}{\arabic{enumi}}
\renewcommand{\labelenumi}{(\theenumi)}
\item\label{lemma:sum-filt1}
The above $\sF$ is indeed a $\T$-equivariant filtration on $R$. 
\item\label{lemma:sum-filt2}
For any $1\leq i\leq k$, let $(\sX_i,\sL_i)/\A^1$ be a $\T$-equivariant 
test configuration of $(X, L_i)$ and let $(\sX,\sL)/\A^1$ be the sum configuration 
of $\left\{(\sX_i,\sL_i)\right\}_{i=1}^k$. Then $\sF_{\sX,\sL}$ is the 
sum filtration of $\left\{\sF_{\sX_i,\sL_i}\right\}_{i=1}^k$. 
\end{enumerate}
\end{lemma}

\begin{proof}
\eqref{lemma:sum-filt2} is trivial by Lemma \ref{lemma:sum-tc-ideal}. We show 
\eqref{lemma:sum-filt1}. 
\begin{itemize}
\item
(Linearly boundedness)
There exist $e_-, e_+\in\R$ such that, for any $1\leq i\leq k$ and $m\in r\Z_{>0}$, 
$\sF_i^{x m}R_m^i=R_m^i$ for any $x\leq e_-$ and $\sF_i^{x m}R_m^i=0$ for any 
$x\geq e_+$. This implies that $\sF^{x m}R_m=R_m$ for any $x\leq k e_-$ and 
$\sF^{x m}R_m=0$ for any $x\geq k e_+$. 
\item
(Multiplicativity)
Take any $x,x'\in\R$ and $m,m'\in r\Z_{\geq 0}$. 
Take any $x_1,\dots,x_k\in\R$ with $\sum_{i=1}^k x_i=x$ and 
$x'_1,\dots,x'_k\in\R$ with $\sum_{i=1}^k x'_i=x'$. Then we have 
\[
\left(\sF_1^{x_1}R_m^1\cdots\sF_k^{x_k}R_m^k\right)\cdot
\left(\sF_1^{x'_1}R_{m'}^1\cdots\sF_k^{x'_k}R_{m'}^k\right)\subset
\sF_1^{x_1+x'_1}R_{m+m'}^1\cdots\sF_k^{x_k+x'_k}R_{m+m'}^k.
\]
This implies that $\sF^x R_m\cdot\sF^{x'}R_{m'}\subset\sF^{x+x'}R_{m+m'}$. 
\item
It is obvious that $\sF^x R_m\subset\sF^{x'}R_m$ whenever $x\geq x'$. 
\item
(Left continuity)
Take any $m\in r\Z_{\geq 0}$ and $x\in\R$. 
For any $1\leq i\leq k$, let  $\lambda_{i,1}<\dots<\lambda_{i,M_i}$ be defined to be 
\[
\left\{\lambda_{i,1},\dots,\lambda_{i,M_i}\right\}
:=\left\{\lambda\in\R\,\,|\,\,\dim\Gr_{\sF_i}^x R_m^i\neq 0\right\}, 
\]
where $\Gr_{\sF_i}^x R_m^i$ be as in \cite[(3.1)]{Xu}. 
Note that we have the following: 
\[
\sF^x R_m=\sum_{\substack{1\leq j_i\leq M_i\,\,(1\leq i\leq k);\\ 
\lambda_{1,j_1}+\cdots+\lambda_{k,j_k}\geq x}}
\sF_1^{\lambda_{1,j_1}}R_m^1\cdots\sF_k^{\lambda_{k,j_k}}R_m^k.
\]
For $x\in \R$, there exists a very small $\varepsilon\in\R_{>0}$ such that
\begin{eqnarray*}
&&\left\{(j_1,\dots,j_k)\in\Z_{>0}^k\,\,|\,\,j_i\leq M_i\,\,(1\leq i\leq k), 
\text{ and }
\lambda_{1,j_1}+\cdots+\lambda_{k,j_k}\geq x-\varepsilon\right\}\\
&=&\left\{(j_1,\dots,j_k)\in\Z_{>0}^k\,\,|\,\,j_i\leq M_i\,\,(1\leq i\leq k), 
\text{ and }
\lambda_{1,j_1}+\cdots+\lambda_{k,j_k}\geq x\right\}
\end{eqnarray*}
holds. Thus we get $\sF^x R_m=\sF^{x-\varepsilon}R_m$. 
\item
($\T$-equivariant property)
We can directly check that 
\[
\sF^x R_m=\bigoplus_{\alpha\in M(\T)}\sF^x R_{m,\alpha}
\]
with 
\[
\sF^x R_{m,\alpha}=\sum_{\substack{\alpha_1,\dots,\alpha_k\in M(\T); \\
\alpha_1+\cdots+\alpha_k=\alpha}}
\sum_{\substack{x_1,\dots,x_k\in\R; \\ x_1+\cdots+x_k=x}}
\sF_1^{x_1}R_{m,\alpha_1}^1\cdots\sF_k^{x_k}R_{m,\alpha_k}^k.
\]
Thus $\sF$ is $\T$-equivariant. 
\end{itemize}
Thus we get the assertion \eqref{lemma:sum-filt1}. 
\end{proof}

\begin{proposition}\label{proposition:sum-filt-lambda}
For any $1\leq i\leq k$, let $\sF_i$ be a $\T$-equivariant filtration on $R^i$, and 
let $\sF$ be the sum filtration of $\{\sF_i\}_{i=1}^k$. Then we have the equality
\[
\lambda_{\max}(\sF)=\sum_{i=1}^k\lambda_{\max}(\sF_i).
\]
\end{proposition}

\begin{proof}
Assume that $\lambda_{\max}(\sF)<\sum_{i=1}^k\lambda_{\max}(\sF_i)$. 
Then there exists $x_1,\dots,x_k\in\R$ such that $x_i<\lambda_{\max}(\sF_i)$ 
for any $1\leq i\leq k$ and $x:=\sum_{i=1}^k x_i>\lambda_{\max}(\sF)$. 
As in \cite[Lemma 3.17]{Xu}, for any sufficiently divisible $m\in r\Z_{>0}$, 
we know that $\sF_i^{m x_i}R_m^i\neq 0$ for any $1\leq i\leq k$. 
This implies that 
\[
\sF^{m x}R_m\supset\prod_{i=1}^k\sF_i^{m x_i}R_m^i\neq 0, 
\]
a contradiction. Thus $\lambda_{\max}(\sF)
\geq \sum_{i=1}^k\lambda_{\max}(\sF_i)$ holds. 

Assume that $\lambda_{\max}(\sF)>\sum_{i=1}^k\lambda_{\max}(\sF_i)$. 
Take any $x\in\left(\sum_{i=1}^k\lambda_{\max}(\sF_i), \lambda_{\max}(\sF)\right)$. 
As in \cite[Lemma 3.17]{Xu}, 
for any sufficiently divisible $m\in r\Z_{>0}$, 
we know that $\sF^{m x}R_m\neq 0$. This implies that, for any $1\leq i\leq k$, 
there exists $x_i\in\R$ such that $\sum_{i=1}^k x_i=x$ and 
$\sF_i^{m x_i}R_m^i\neq 0$. This leads to the inequality 
$x_i\leq \lambda_{\max}(\sF_i)$. Thus we get 
$x\leq\sum_{i=1}^k\lambda_{\max}(\sF_i)$, 
a contradiction. Thus $\lambda_{\max}(\sF)
\leq \sum_{i=1}^k\lambda_{\max}(\sF_i)$ holds. 
\end{proof}

\begin{proposition}\label{proposition:sum-filt-twist}
For any $1\leq i\leq k$, let $\sF_i$ be a $\T$-equivariant filtration on $R^i$, and 
let $\sF$ be the sum filtration of $\{\sF_i\}_{i=1}^k$. 
\begin{enumerate}
\renewcommand{\theenumi}{\arabic{enumi}}
\renewcommand{\labelenumi}{(\theenumi)}
\item\label{proposition:sum-filt-twist1}
For any $1\leq i\leq k$, let $\sF_{i,[C_i]}$ be the $C_i$-shift of $\sF_i$ for 
$C_i\in\R$. Then the sum filtration of $\left\{\sF_{i,[C_i]}\right\}_{i=1}^k$ is 
equal to the $\left(\sum_{i=1}^k C_i\right)$-shift 
$\sF_{\left[\sum_{i=1}^k C_i\right]}$ of $\sF$. 
\item\label{proposition:sum-filt-twist2} (cf.\ Lemma \ref{lemma:sum-tc-twist})
Take any $\xi\in N_\R(\T)$. Let us consider the $\xi$-twist $\sF_{i,\xi}$ of $\sF_i$ 
for any $1\leq i\leq k$. Then the sum filtration of $\left\{\sF_{i,\xi}\right\}_{i=1}^k$ 
is equal to the $\xi$-twist 
$\sF_{\xi}$ of $\sF$. 
\end{enumerate}
\end{proposition}

\begin{proof}
\eqref{proposition:sum-filt-twist1}
Let $\sG$ be the sum filtration of $\left\{\sF_{i,[C_i]}\right\}_{i=1}^k$. 
Then we know that
\begin{eqnarray*}
\sG^x R_m&=&\sum_{x_1+\cdots+x_k=x}\sF_1^{x_1-C_1 m}R_m^1\cdots
\sF_k^{x_k-C_k m}R_m^k\\
&=&\sum_{y_1+\cdots+y_k=x-\left(\sum_{i=1}^k C_i\right)m}
\sF_1^{y_1}R_m^1\cdots\sF_k^{y_k}R_m^k\\
&=&\sF^{x-\left(\sum_{i=1}^k C_i\right)m}R_m
=\sF_{\left[\sum_{i=1}^k C_i\right]}^x R_m.
\end{eqnarray*}
Thus $\sG$ is equal to $\sF_{\left[\sum_{i=1}^k C_i\right]}$. 

\eqref{proposition:sum-filt-twist2}
Let $\sG$ be the sum filtration of $\left\{\sF_{i,\xi}\right\}_{i=1}^k$. 
Then we know that 
\begin{eqnarray*}
\sG^x R_{m,\alpha}&=&\sum_{\sum_{i=1}^k\alpha_i=\alpha}\sum_{\sum_{i=1}^k x_i=x}
\sF_{1,\xi}^{x_1}R_{m,\alpha_1}^1\cdots\sF_{k,\xi}^{x_k}R_{m,\alpha_k}^k \\
&=&\sum_{\sum_{i=1}^k\alpha_i=\alpha}\sum_{\sum_{i=1}^k x_i=x}
\sF_1^{x_1-\langle\alpha_1,\xi\rangle}R_{m,\alpha_1}^1
\cdots\sF_k^{x_k-\langle\alpha_k,\xi\rangle}R_{m,\alpha_k}^k \\
&=&\sum_{\sum_{i=1}^k\alpha_i=\alpha}\sum_{\sum_{i=1}^k y_i=x
-\langle\alpha,\xi\rangle}
\sF_1^{y_1}R_{m,\alpha_1}^1\cdots\sF_k^{y_k}R_{m,\alpha_k}^k \\
&=&\sF^{x-\langle\alpha,\xi\rangle}R_{m,\alpha}=\sF_\xi^x R_{m,\alpha}.
\end{eqnarray*}
Thus $\sG$ is equal to $\sF_{\xi}$. 
\end{proof}

\begin{proposition}\label{proposition:approx-sum-filt}
For any $1\leq i\leq k$, let $\sF_i$ be a $\T$-equivariant filtration on $R^i$, and 
let $\sF$ be the sum filtration of $\{\sF_i\}_{i=1}^k$. 
For any $1\leq i\leq k$, let us take an approximating sequence 
$\left\{\sF_{i,(m)}\right\}_{m\in r\Z_{>0}}$ of $\sF_i$. For any $m\in r\Z_{>0}$, 
let $\sF_{(m)}$ be the sum filtration of $\left\{\sF_{i,(m)}\right\}_{i=1}^k$. 
Then $\left\{\sF_{(m)}\right\}_{m\in r\Z_{>0}}$ is an approximating sequence 
of $\sF$. 
\end{proposition}

\begin{proof}
Observe the following: 
\begin{enumerate}
\renewcommand{\theenumi}{\roman{enumi}}
\renewcommand{\labelenumi}{(\theenumi)}
\item\label{proof:approx-sum-filt1}
We have 
\begin{eqnarray*}
\sF_{(m)}^x R_{m'}&=&\sum_{\sum_{i=1}^k x_i\geq x}
\sF_{1,(m)}^{x_1}R_{m'}^1\cdots\sF_{k,(m)}^{x_k}R_{m'}^k\\
&\subset&\sum_{\sum_{i=1}^k x_i\geq x}
\sF_{1}^{x_1}R_{m'}^1\cdots\sF_{k}^{x_k}R_{m'}^k=\sF^x R_{m'}. 
\end{eqnarray*}
\item\label{proof:approx-sum-filt2}
We have 
\begin{eqnarray*}
\sF_{(m)}^x R_{m}&=&\sum_{\sum_{i=1}^k x_i\geq x}
\sF_{1,(m)}^{x_1}R_{m}^1\cdots\sF_{k,(m)}^{x_k}R_{m}^k\\
&=&\sum_{\sum_{i=1}^k x_i\geq x}
\sF_{1}^{x_1}R_{m}^1\cdots\sF_{k}^{x_k}R_{m}^k=\sF^x R_{m}. 
\end{eqnarray*}
\item\label{proof:approx-sum-filt3}
For any $s\in\Z_{>0}$, we have 
\begin{eqnarray*}
\sF_{(m)}^x R_{m s}&=&\sum_{\sum_{i=1}^k x_i\geq x}\sF_{1,(m)}^{x_1}R_{m s}^1
\cdots\sF_{k,(m)}^{x_k}R_{m s}^k\\
&=&\sum_{\sum_{i=1}^k x_i\geq x}\sum_{\substack{\sum_{j=1}^s y_{1j}\geq x_1\\
\vdots\\ \sum_{j=1}^a y_{k j}\geq x_k}}
\left(\sF_1^{y_{11}}R_m^1\cdots\sF_1^{y_{1s}}R_m^1\right)\cdots
\left(\sF_k^{y_{k1}}R_m^k\cdots\sF_k^{y_{ks}}R_m^k\right)\\
&=&\sum_{\sum_{j=1}^s y_j\geq x}\sum_{\substack{\sum_{i=1}^k y_{i 1}\geq y_1 \\
\vdots\\ \sum_{i=1}^k y_{i s}\geq y_s}}
\left(\sF_1^{y_{1 1}}R_m^1\cdots\sF_k^{y_{k 1}}R_m^k\right)\cdots
\left(\sF_1^{y_{1 s}}R_m^1\cdots\sF_k^{y_{k s}}R_m^k\right)\\
&=&\sum_{\sum_{j=1}^s y_j\geq x}\sF^{y_1}R_m\cdots\sF^{y_s}R_m.
\end{eqnarray*}
\end{enumerate}
Thus $\left\{\sF_{(m)}\right\}_{m\in r\Z_{>0}}$ is an approximating sequence 
of $\sF$. 
\end{proof}

\begin{proposition}\label{proposition:approx-filt-e}
Under the assumption in Proposition \ref{proposition:approx-sum-filt}, assume 
moreover that $\sF_i$ and $\sF_{i,(m)}$ are $\Z$-valued. Take any $e\in\Z_{>0}$. 
Let $\sF_i^{(e)}:=\left(\sF_i\right)^{(e)}$, $\sF^{(e)}$, 
$\sF_{i,(m)}^{(e)}:=\left(\sF_{i,(m)}\right)^{(e)}$, 
$\sF_{(m)}^{(e)}:=\left(\sF_{(m)}\right)^{(e)}$ be as in Proposition 
\ref{proposition:tc-twist} \eqref{proposition:tc-twist2}. Then we have 
the following: 
\begin{enumerate}
\renewcommand{\theenumi}{\arabic{enumi}}
\renewcommand{\labelenumi}{(\theenumi)}
\item\label{proposition:approx-filt-e1}
The filtration $\sF^{(e)}$ is the sum filtration of $\left\{\sF_i^{(e)}\right\}_{i=1}^k$.
\item\label{proposition:approx-filt-e2}
For any $1\leq i\leq k$, $\left\{\sF_{i,(m)}^{(e)}\right\}_{i=1}^k$ is 
an approximating sequence of $\sF_i^{(e)}$. 
\item\label{proposition:approx-filt-e3}
The filtration $\sF_{(m)}^{(e)}$ is the sum filtration of 
$\left\{\sF_{i,(m)}^{(e)}\right\}_{i=1}^k$.
\item\label{proposition:approx-filt-e4}
$\left\{\sF_{(m)}^{(e)}\right\}_{i=1}^k$ is 
an approximating sequence of $\sF^{(e)}$. 
\end{enumerate}
\end{proposition}

\begin{proof}
By Propositions \ref{proposition:approx-base-change} and 
\ref{proposition:approx-sum-filt}, it is enough to show the assertion 
\eqref{proposition:approx-filt-e1}. 
We note that $\sF^{(e)}$ and $\sF_i^{(e)}$ are $e\Z$-valued. 
Let $\sG$ be the sum filtration 
of $\left\{\sF_i^{(e)}\right\}_{i=1}^k$. For any $\lambda\in e\Z$, we have
\begin{eqnarray*}
\sG^\lambda R_m&=&\sum_{\substack{\lambda_1,\dots,\lambda_k\in e\Z; \\
\lambda_1+\cdots+\lambda_k=\lambda}}
\sF_1^{(e),\lambda_1}R_m^1\cdots\sF_k^{(e),\lambda_k}R_m^k\\
&=&\sum_{\substack{\lambda_1,\dots,\lambda_k\in e\Z; \\
\lambda_1+\cdots+\lambda_k=\lambda}}
\sF_1^{\lambda_1/e}R_m^1\cdots\sF_k^{\lambda_k/e}R_m^k\\
&=&\sum_{\substack{\mu_1,\dots,\mu_k\in \Z; \\
\mu_1+\cdots+\mu_k=\lambda/e}}
\sF_1^{\mu_1}R_m^1\cdots\sF_k^{\mu_k}R_m^k\\
&=&\sF^{\lambda/e}R_m=\sF^{(e),\lambda}R_m.
\end{eqnarray*}
Thus the assertion follows. 
\end{proof}

\section{The sum of non-Archimedean metrics}\label{section:sum-namc}

It is important that we extend the sum of test configurations in terms of the non-Archimedean metrics, following the approach of Boucksom--Jonsson \cite{BJ22} and Boucksom--Hisamoto--Jonsson \cite{BHJ}. A detailed explanation of the non-Archimedean metrics or the Berkovich analytification is completely out of reach of this paper, and the reader is referred to the aforementioned papers for the details.

We include here the bare minimum of notation that is used afterwards. The reference that we follow is Boucksom--Jonsson \cite{BJ22}. For a polarized projective variety $(X,L)$, we define a connected and compact topological space called the Berkovich analytification $X^{\mathrm{an}}$. It contains a subset $X^{\mathrm{div}}$ consisting of divisorial valuations which is dense in $X^{\mathrm{an}}$ \cite[Theorem 2.14]{BJ22}. Equivalence classes \cite[Definition 6.1]{BHJ} of semiample test configurations for $(X,L)$ corresponds one-to-one with certain plurisubharmonic (psh) functions on $X^{\mathrm{an}}$, as explained in \cite{BHJ,BJ22} and very briefly recalled below. Let $(\mathcal{X} , \mathcal{L}) / \A^1$ be a semiample test configuration for $(X,L)$, and pick another test configuration $\tilde{\mathcal{X}}$ with morphisms $\mu : \tilde{\mathcal{X}} \to \mathcal{X}$ and $\rho : \tilde{\mathcal{X}} \to X_{\A^1}$, where $(X_{\A^1} , L_{\A^1}) / \A^1$ is the trivial test configuration. Then
\begin{equation} \label{eqdfvcarl}
	D:=  \mu^* \mathcal{L} - \rho^* L_{\A^1}
\end{equation}
is a $\G_m$-invariant $\mathbb{Q}$-Cartier $\mathbb{Q}$-divisor with support in $\tilde{\mathcal{X}}_0$, and the map $\mathcal{L} \mapsto D$ is known to define a one-to-one correspondence between the set of equivalence classes of test configurations $(\mathcal{X} , \mathcal{L}) / \A^1$ (with $\mathcal{L}$ possibly non-semiample over $\A^1$) for $(X,L)$ and the space of piecewise linear functions on $X^{\mathrm{an}}$ \cite[Definition 2.1, Theorem 2.7, and \S 2.7]{BJ22}. Furthermore, given a test configuration $(\tilde{\mathcal{X}} , \mu^* \mathcal{L}) / \A^1$ for $(X,L)$, the Gauss extension yields an embedding $\sigma_{\tilde{\mathcal{X}}} : X^{\mathrm{an}} \to \tilde{\mathcal{X}}^{\mathrm{an}}$ as in \cite[\S 1.3 and \S 2.1]{BJ22}. When $D$ is a $\G_m$-invariant $\mathbb{Q}$-Cartier $\mathbb{Q}$-divisor with support in $\tilde{\mathcal{X}}_0$, we define a continuous function $\varphi_D \in C^0 (X^{\mathrm{an}} , \mathbb{R})$ by
\begin{equation} \label{eqdfvnapl}
	\varphi_D (v) := \sigma_{\tilde{\mathcal{X}}} (v) (D)
\end{equation}
for $v \in X^{\mathrm{an}}$ \cite[\S 2.2]{BJ22}. Thus, a test configuration $(\mathcal{X}, \mathcal{L})$ defines a function $\varphi_{\mathcal{L}} := \varphi_D \in C^0 (X^{\mathrm{an}} , \mathbb{R})$ by (\ref{eqdfvnapl}), where $D$ is defined as in (\ref{eqdfvcarl}). We note that the map $D \mapsto \varphi_D$ is invariant under the pullback, i.e.~$\varphi_{\mu^* D} = \varphi_D$ \cite[\S 2.2]{BJ22}, and hence the construction above does not depend on the choice of $\tilde{\mathcal{X}}$. The map $\mathcal{L} \mapsto \varphi_{\mathcal{L}}$ is known to set up a one-to-one correspondence between the set of equivalence classes of semiample test configurations for $(X,L)$ and the set $\mathcal{H}^{\mathrm{NA}} (L)$ of rational Fubini--Study functions on $X^{\mathrm{an}}$ \cite[\S 2.4 and Theorem 2.31]{BJ22}. The set $\mathcal{H}^{\mathrm{NA}} (L)$ admits a completion denoted by $\mathcal{E}^{1, \mathrm{NA}} (L)$ \cite[Theorem 4.15 (ii), \S 7.2, \S 12.1]{BJ22}; note that our notation $\mathcal{H}^{\mathrm{NA}} = \mathcal{H}^{\mathrm{NA}} (L)$ (resp.~$\mathcal{E}^{1, \mathrm{NA}} = \mathcal{E}^{1, \mathrm{NA}} (L)$) corresponds to $\mathcal{H}_{\mathbb{Q}}$ (resp.~to $\mathcal{E}^1$) in \cite{BJ22}. We also write $\mathrm{E}^{\mathrm{NA}}$ for the non-Archimedean Monge--Amp\`ere energy defined for $\phi^{\mathrm{NA}} \in \mathcal{E}^{1, \mathrm{NA}}$ by
\begin{equation*}
	\mathrm{E}^{\mathrm{NA}} (\phi^{\mathrm{NA}}) := \frac{(L,\phi^{\mathrm{NA}})^{n+1}}{(n+1) (L^{\cdot n})} > - \infty ,
\end{equation*}
where $(L,\phi^{\mathrm{NA}})^{n+1}$ is the energy pairing defined in \cite[Theorem 7.1]{BJ22}. In what follows, we write $(\phi^{\mathrm{NA}})^{n+1}$ for $(L,\phi^{\mathrm{NA}})^{n+1}$ to simplify the notation. Moreover, in this paper we also call $\mathcal{H}^{\mathrm{NA}}$ the set of all non-Archimedean \textit{metrics}, following the terminology of \cite[Definition 6.4]{BHJ} (see also \cite[Example 3.3]{BJ22}).

In this section, let $(X, \Delta)$, $L_1 , \dots , L_k$, and $L= \sum_{i=1}^k L_i$ be as in \S \ref{section:sum-tc}. We write $\mathcal{H}^{\mathrm{NA}}(L)$ for the set of rational Fubini--Study functions for $L$, and $\mathcal{H}^{\mathrm{NA}}_1, \dots , \mathcal{H}^{\mathrm{NA}}_k$ for the ones for $L_1 , \dots , L_k$ respectively.

We first observe the following straightforward consequence of 
Lemma \ref{lemma:sum-tc} in terms of the non-Archimedean metrics 
in the sense of \cite[\S 6]{BHJ}.

\begin{lemma} \label{lemma:aplmstcna}
The non-Archimedean metric on $L$ represented by the sum configuration 
$(\sX, \sL)$ of $\left\{(\sX_i, \sL_i)\right\}_{i=1}^k$ depends only on 
the non-Archimedean metrics on $L_1,\dots,L_k$ represented respectively by 
$(\sX_1, \sL_1),\dots,(\sX_k, \sL_k)$.
\end{lemma}

\begin{proof}
Let $\sZ$ be a common partial resolution of 
$(\sX_1, \sL_1),\dots,(\sX_k, \sL_k)$. If we take another representative, 
say $(\sX'_i, \sL'_i)$, of the non-Archimedean metric defined by 
$(\sX_i , \sL_i)$, we simply replace $\sZ$ by a common partial resolution of 
$\sZ, (\sX'_1, \sL'_1),\dots, (\sX'_k, \sL'_k)$, which defines the same 
non-Archimedean metric as $(\sX, \sL)$ by Lemma \ref{lemma:sum-tc}.
\end{proof}

Thus, the following map
\begin{equation*}
	\mathsf{S} : \mathcal{H}^{\mathrm{NA}}_1 \times \cdots \times \mathcal{H}^{\mathrm{NA}}_k \to \mathcal{H}^{\mathrm{NA}}(L),
\end{equation*}
taking any ample representatives from $\mathcal{H}^{\mathrm{NA}}_1, \dots , \mathcal{H}^{\mathrm{NA}}_k$ and giving (the equivalence class of) the sum test configuration of them, is well-defined. In fact, this map agrees with the sum of non-Archimedean metrics, as in the following lemma.

\begin{lemma} \label{lmsmnatc}
	The map $\mathsf{S}$ above satisfies
	\begin{equation*}
		\mathsf{S} (\phi^{\mathrm{NA}}_1 , \dots , \phi^{\mathrm{NA}}_k ) = \phi^{\mathrm{NA}}_1 + \cdots + \phi^{\mathrm{NA}}_k
	\end{equation*}
	where the right hand side is the the sum of non-Archimedean metrics as defined in \cite[\S 6.2]{BHJ} or \cite[Proposition 3.6]{BJ22}. Moreover, if $\phi^{\mathrm{NA}}_i , \psi^{\mathrm{NA}}_i \in \mathcal{H}^{\mathrm{NA}}_i$ satisfy $\phi^{\mathrm{NA}}_i \le  \psi^{\mathrm{NA}}_i$ for all $i=1 , \dots , k$, we have
\begin{equation} \label{eqspsord}
	\mathsf{S} (\phi^{\mathrm{NA}}_1 , \dots , \phi^{\mathrm{NA}}_k ) \le \mathsf{S} (\psi^{\mathrm{NA}}_1 , \dots , \psi^{\mathrm{NA}}_k ).
\end{equation}
\end{lemma}

\begin{proof}
	We take $\phi^{\mathrm{NA}}_i \in \mathcal{H}^{\mathrm{NA}}_i$ and take an ample representative $(\mathcal{X}_i, \mathcal{L}_i)$ for $i=1 , \dots , k$. We replace this model by the common partial resolution $\sigma_i : \mathcal{Z} \to \mathcal{X}_i$ and note that $(\mathcal{Z} , \sigma^*_i \mathcal{L}_i )$ represents $\phi^{\mathrm{NA}}_i$. Then the sum $\phi^{\mathrm{NA}}_1 + \cdots + \phi^{\mathrm{NA}}_k$ is represented by $(\mathcal{Z} , \sigma_1^* \mathcal{L}_1 + \cdots + \sigma_k^* \mathcal{L}_k)$, establishing the first part of the claim by recalling Lemma \ref{lemma:sum-tc}. We immediately observe that, if $\phi^{\mathrm{NA}}_i , \psi^{\mathrm{NA}}_i \in \mathcal{H}^{\mathrm{NA}}_i$ satisfies $\phi^{\mathrm{NA}}_i \le  \psi^{\mathrm{NA}}_i$ for all $i=1 , \dots , k$, we have
\begin{equation*} 
	\phi^{\mathrm{NA}}_1 + \cdots + \phi^{\mathrm{NA}}_k \le \psi^{\mathrm{NA}}_1 + \cdots + \psi^{\mathrm{NA}}_k
\end{equation*}
as claimed.
\end{proof}

\begin{proposition} \label{ppsmctnna}
	The map $\mathsf{S}$ can be extended to the map $\mathsf{S} : \mathcal{E}^{1, \mathrm{NA}}_1 \times \cdots \times \mathcal{E}^{1, \mathrm{NA}}_k \to \mathcal{E}^{1, \mathrm{NA}} (L)$ defined by
\begin{equation*}
	\mathsf{S}(\phi^{\mathrm{NA}}_1 , \dots , \phi^{\mathrm{NA}}_k) := \phi^{\mathrm{NA}}_1 + \cdots + \phi^{\mathrm{NA}}_k
\end{equation*}
which is continuous with respect to the strong topology.
\end{proposition}

This result may be well-known to the experts, particularly because it follows directly from various results in \cite{BJ22} as we can see below. In the rest of the paper, we do not need the continuity of $\mathsf{S}$, but the statement is included here for completeness.

\begin{proof}
	We first prove that the image of $\mathsf{S}$ is contained in $\mathcal{E}^{1, \mathrm{NA}} (L)$. We pick a decreasing net
	\begin{equation*}
		\{ (\phi^{\mathrm{NA}}_{1j} , \dots , \phi^{\mathrm{NA}}_{kj}) \}_j \subset \mathcal{H}^{\mathrm{NA}}_1 \times \cdots \times \mathcal{H}^{\mathrm{NA}}_k
	\end{equation*}
	converging (in the product topology) to $(\phi^{\mathrm{NA}}_1, \dots , \phi^{\mathrm{NA}}_k) \in \mathcal{E}^{1, \mathrm{NA}}_1 \times \cdots \times \mathcal{E}^{1, \mathrm{NA}}_k$. By (\ref{eqspsord}) and the definition of $\mathsf{S}$, $\left\{ \sum_{i=1}^k \phi_{ij}^{\mathrm{NA}} \right\}_j \subset \mathcal{H}^{\mathrm{NA}} (L)$ is a decreasing net converging pointwise to $\sum_{i=1}^k \phi_{i}^{\mathrm{NA}}$ over $X^{\mathrm{div}}$. Thus it suffices to prove
	\begin{equation*}
		\inf_{j} \mathrm{E}^{\mathrm{NA}} \left( \sum_{i=1}^k \phi_{ij}^{\mathrm{NA}} \right) = \inf_j \frac{(\phi^{\mathrm{NA}}_{1j} + \cdots + \phi^{\mathrm{NA}}_{kj})^{n+1}}{(n+1) (L^{\cdot n})} > - \infty.
	\end{equation*}
	We now recall from \cite[Lemma 7.10]{BJ22} that there exists a constant $C>0$ depending only on $n$, $k$, and $L_1 , \dots,  L_k$ such that
	\begin{equation*}
		(\phi^{\mathrm{NA}}_{1j} + \cdots + \phi^{\mathrm{NA}}_{kj})^{n+1} > C \sum_{i=1}^k (\phi^{\mathrm{NA}}_{ij})^{n+1} = C \sum_{i=1}^k (L_i^{\cdot n}) \mathrm{E}_i^{\mathrm{NA}} (\phi^{\mathrm{NA}}_{ij})
	\end{equation*}
	for each $j$, where $\mathrm{E}_i^{\mathrm{NA}}$ is the non-Archimedean Monge--Amp\`ere energy for $L_i$. The infimum of the right hand side is finite by the monotonicity of the Monge--Amp\`ere energy functional $\mathrm{E}_i^{\mathrm{NA}}$, with $\{ \phi^{\mathrm{NA}}_{ij} \}_j$ decreasing to $\phi^{\mathrm{NA}}_i  \in \mathcal{E}^{1, \mathrm{NA}}_i$ for each $i = 1 , \dots , k$. Thus we get $\inf_{j} \mathrm{E}^{\mathrm{NA}} \left(\sum_{i=1}^k \phi_{ij}^{\mathrm{NA}} \right) > - \infty$ as required.
	
	It remains to prove that $\mathsf{S} : \mathcal{E}^{1, \mathrm{NA}}_1 \times \cdots \times \mathcal{E}^{1, \mathrm{NA}}_k \to \mathcal{E}^{1, \mathrm{NA}} (L)$ is continuous with respect to the strong topology (see \cite[Definition 12.1]{BJ22}). Let $\{ (\psi^{\mathrm{NA}}_{1j} , \dots , \psi^{\mathrm{NA}}_{kj}) \}_j$ be a net in $\mathcal{E}^{1, \mathrm{NA}}_1 \times \cdots \times \mathcal{E}^{1, \mathrm{NA}}_k$ strongly converging (in the product topology) to $(\phi^{\mathrm{NA}}_1, \dots , \phi^{\mathrm{NA}}_k) \in \mathcal{E}^{1, \mathrm{NA}}_1 \times \cdots \times \mathcal{E}^{1, \mathrm{NA}}_k$. Again by the definition of $\mathsf{S}$, we find that $\left\{ \sum_{i=1}^k \psi^{\mathrm{NA}}_{ij} \right\}_j$ converges pointwise to $\sum_{i=1}^k \phi_i^{\mathrm{NA}}$ over $X^{\mathrm{div}}$. It thus suffices to show that $\mathrm{E}^{\mathrm{NA}} \left( \sum_{i=1}^k \psi^{\mathrm{NA}}_{ij} \right)$ converges to $\mathrm{E}^{\mathrm{NA}} \left( \sum_{i=1}^k \phi_i^{\mathrm{NA}} \right)$, by \cite[\S 12.1]{BJ22}.
	
	To establish this convergence, we first normalize
	\begin{equation*}
		\tilde{\phi}^{\mathrm{NA}}_{i} = \phi^{\mathrm{NA}}_{i} - \sup_{X^{\mathrm{an}}} \phi^{\mathrm{NA}}_{i}, \quad \tilde{\psi}^{\mathrm{NA}}_{ij} = \psi^{\mathrm{NA}}_{ij} - \sup_{X^{\mathrm{an}}} \psi^{\mathrm{NA}}_{ij},
	\end{equation*}
	for each $i=1 , \dots , k$ and each $j$, where we note that $\sup_{X^{\mathrm{an}}} \psi^{\mathrm{NA}}_{ij} \to \sup_{X^{\mathrm{an}}} \phi^{\mathrm{NA}}_{i}$ as $j \to \infty$ since the supremum is attained at a fixed point $v_{\mathrm{triv}} \in X^{\mathrm{div}}$ by \cite[Proposition 4.12 (ii)]{BJ22} (as $X$ is assumed to be irreducible) and $\{ \psi^{\mathrm{NA}}_{ij} \}_j$ converges pointwise to $\phi^{\mathrm{NA}}_{i}$ over $X^{\mathrm{div}}$ by the strong convergence. This in turn shows that $\{ \tilde{\psi}^{\mathrm{NA}}_{ij} \}_j$ converges strongly to $\tilde{\phi}^{\mathrm{NA}}_{i}$ for each $i=1 , \dots , k$, since we have
	\begin{equation*}
		\mathrm{E}_i^{\mathrm{NA}}(\tilde{\psi}^{\mathrm{NA}}_{ij}) = \mathrm{E}_i^{\mathrm{NA}}(\psi^{\mathrm{NA}}_{ij}) - \sup_{X^{\mathrm{an}}} \psi^{\mathrm{NA}}_{ij}
	\end{equation*}
	and similarly for $\mathrm{E}_i^{\mathrm{NA}}(\tilde{\phi}^{\mathrm{NA}}_{i})$ by \cite[Proposition 7.7]{BJ22}. We then find
	\begin{equation*}
		(\phi^{\mathrm{NA}}_{1} + \cdots + \phi^{\mathrm{NA}}_{k})^{n+1} = (\tilde{\phi}^{\mathrm{NA}}_{1} + \cdots + \tilde{\phi}^{\mathrm{NA}}_{k})^{n+1} + (n+1)(L^{\cdot n}) \sum_{i=1}^k \sup_{X^{\mathrm{an}}} \phi^{\mathrm{NA}}_{i}
	\end{equation*}
	and 
	\begin{equation*}
		(\psi^{\mathrm{NA}}_{1j} + \cdots + \psi^{\mathrm{NA}}_{kj})^{n+1} = (\tilde{\psi}^{\mathrm{NA}}_{1j} + \cdots + \tilde{\psi}^{\mathrm{NA}}_{kj})^{n+1} + (n+1)(L^{\cdot n}) \sum_{i=1}^k \sup_{X^{\mathrm{an}}} \psi^{\mathrm{NA}}_{ij}
	\end{equation*}
	by \cite[Proposition 3.14]{BJ22} and recalling $(L_1 + \cdots +L_k)^{\cdot n} = (L^{\cdot n})$. Then, expanding out the energy pairing, \cite[Theorem 7.34]{BJ22} implies that there exists a constant $C >0$ depending only on $L_1 , \dots , L_k$, $n$, and $\tilde{\phi}^{\mathrm{NA}}_{1} , \dots , \tilde{\phi}^{\mathrm{NA}}_{k}$ such that
	\begin{equation*}
		\left| (\tilde{\phi}^{\mathrm{NA}}_{1} + \cdots + \tilde{\phi}^{\mathrm{NA}}_{k})^{n+1} - (\tilde{\psi}^{\mathrm{NA}}_{1j} + \cdots + \tilde{\psi}^{\mathrm{NA}}_{kj})^{n+1} \right| \le C \max_{i=1, \dots , k} \bar{I} (\tilde{\phi}^{\mathrm{NA}}_{i} , \tilde{\psi}^{\mathrm{NA}}_{ij})^{\alpha_n}
	\end{equation*}
	holds eventually for the nets $\{ \tilde{\psi}^{\mathrm{NA}}_{ij} \}_j$ converging to $\tilde{\phi}^{\mathrm{NA}}_{i}$, where $\alpha_n \in (0,1]$ is a constant depending only on $n$ and $\bar{I}$ is the quasi-metric defined in \cite[Definition 12.3]{BJ22} (see also \cite[(7.29) in Proposition 7.27]{BJ22}). Since $\{ \tilde{\psi}^{\mathrm{NA}}_{ij} \}_j$ converges strongly to $\tilde{\phi}^{\mathrm{NA}}_{i}$, we have $\bar{I} (\tilde{\phi}^{\mathrm{NA}}_{i} , \tilde{\psi}^{\mathrm{NA}}_{ij}) \to 0$ for each $i=1 , \dots , k$ by \cite[Theorem 12.4]{BJ22}, which shows
	\begin{equation*}
		(\tilde{\psi}^{\mathrm{NA}}_{1j} + \cdots + \tilde{\psi}^{\mathrm{NA}}_{kj})^{n+1} \to (\tilde{\phi}^{\mathrm{NA}}_{1} + \cdots + \tilde{\phi}^{\mathrm{NA}}_{k})^{n+1}.
	\end{equation*}
	Thus, combining the above argument we have
	\begin{equation*}
		(\psi^{\mathrm{NA}}_{1j} + \cdots + \psi^{\mathrm{NA}}_{kj})^{n+1} \to (\phi^{\mathrm{NA}}_{1} + \cdots + \phi^{\mathrm{NA}}_{k})^{n+1},
	\end{equation*}
	which shows that $\mathrm{E}^{\mathrm{NA}} \left( \sum_{i=1}^k \psi^{\mathrm{NA}}_{ij} \right)$ converges to $\mathrm{E}^{\mathrm{NA}} \left(\sum_{i=1}^k \phi_i^{\mathrm{NA}} \right)$, as claimed.
\end{proof}

\section{The coupled Ding invariant}\label{section:c-ding}

In this section, we introduce the notion of coupled Ding semistability and 
(reduced) coupled Ding stability for log Fano pairs. 
In this section, we fix an $n$-dimensional (klt) log Fano pair $(X,\Delta)$, 
an algebraic torus $\T\simeq\G_m^p$ with $p\in\Z_{\geq 0}$, 
an injection $\T\to\Aut(X,\Delta)$, and $\T$-linearized ample $\Q$-line 
bundles $L_1,\dots,L_k$ on $X$ such that $L:=\sum_{i=1}^k L_i$ coincides with 
$-(K_X+\Delta)$ with the standard $\T$-linearization. 

\begin{definition}\label{definition:cFutaki}
We set 
\[
\alpha_{\bc}^{\cp}:=\sum_{i=1}^k \alpha_{\bc}^{L_i}\in M_\R(\T) 
\]
(see Definition \ref{definition:polytope}). 
We say that $\left(X,\Delta;\{L_i\}_{i=1}^k\right)$ \emph{has vanishing coupled 
$\T$-Futaki characters} if $\alpha_{\bc}^{\cp}=0$. 
By Lemma \ref{lemma:sum-basics} \eqref{lemma:sum-basics2}, those definitions 
do not depend on the choice of $\T$-linearizations of $L_i$. 
If $\T$ is a maximal torus of $\Aut(X,\Delta)$, we simply say that 
$\left(X,\Delta;\{L_i\}_{i=1}^k\right)$ \emph{has vanishing coupled Futaki characters}.
Note that all maximal tori of $\Aut(X,\Delta)$ are mutually conjugate to 
each other. 
\end{definition}

\begin{definition}\label{definition:ding}
\begin{enumerate}
\renewcommand{\theenumi}{\arabic{enumi}}
\renewcommand{\labelenumi}{(\theenumi)}
\item\label{definition:ding1}
For any $1\leq i\leq k$, let $\sF_i$ be a $\T$-equivariant filtration on $R^i$, and 
let $\sF$ be the sum filtration of $\{\sF_i\}_{i=1}^k$. 
\begin{enumerate}
\renewcommand{\theenumii}{\roman{enumii}}
\renewcommand{\labelenumii}{(\theenumii)}
\item\label{definition:ding11}
We set 
\[
\JJ^{\cp}\left(\left\{\sF_i\right\}_{i=1}^k\right):=\sum_{i=1}^k\JJ(\sF_i), 
\]
where $\JJ(\sF_i)$ is defined in Definition \ref{definition:J}. 
\item\label{definition:ding12}
We set 
\[
\JJ_{\T}^{\cp}\left(\left\{\sF_i\right\}_{i=1}^k\right):=\inf_{\xi\in N_\R(\T)}
\JJ^{\cp}\left(\left\{\sF_{i,\xi}\right\}_{i=1}^k\right). 
\]
By \cite[Lemma 6.4 and Proposition 6.6 (i)]{Xu}, the functions 
\[
\xi\mapsto\sum_{i=1}^k\lambda_{\max}\left(\sF_{i,\xi}\right), \quad
\xi\mapsto S_{L_i}\left(\sF_{i,\xi}\right)
\]
are continuous. Thus we have 
\[
\JJ_{\T}^{\cp}\left(\left\{\sF_i\right\}_{i=1}^k\right)=\inf_{\xi\in N_\Q(\T)}
\JJ^{\cp}\left(\left\{\sF_{i,\xi}\right\}_{i=1}^k\right). 
\]
\item\label{definition:ding13}
For any $\delta\in\R_{>0}$, the \emph{coupled Ding invariant of $\{\sF_i\}_{i=1}^k$ 
with the slope $\delta$} is defined to be 
\[
\DD\left(\{\sF_i\}_{i=1}^k; \delta\right):=\mu(\sF;\delta)
-\sum_{i=1}^k S_{L_i}(\sF_i), 
\]
where $\mu(\sF;\delta)$ is the $\delta$-lc slope of $\sF$ in the sense of 
\cite[Definition 3.45]{Xu} (see also Lemma \ref{lemma:lct-compute}). 
We often denote $\mu(\sF;\delta)$ by 
$\mu\left(\left\{\sF_i\right\}_{i=1}^k;\delta\right)$. 
We can express 
\[
\DD\left(\{\sF_i\}_{i=1}^k; \delta\right)=\DD(\sF; \delta)
+S_L(\sF)-\sum_{i=1}^k S_{L_i}(\sF_i),
\]
where $\DD(\sF; \delta)$ is the Ding invariant of the filtration $\sF$ with the 
slope $\delta$ in the sense of \cite[Definition 3.45]{Xu}. 

We define the \emph{coupled Ding invariant of $\{\sF_i\}_{i=1}^k$} as 
\[
\DD\left(\{\sF_i\}_{i=1}^k\right):=\mu(\sF)
-\sum_{i=1}^k S_{L_i}(\sF_i)=\DD(\sF)+S_L(\sF)-\sum_{i=1}^k S_{L_i}(\sF_i),
\]
where 
\[
\mu(\sF):=\mu\left(\left\{\sF_i\right\}_{i=1}^k\right):=\mu(\sF; 1),\quad
\DD(\sF):=\DD(\sF;1). 
\]
Equivalently, $\DD\left(\{\sF_i\}_{i=1}^k\right):=\DD\left(\{\sF_i\}_{i=1}^k;1\right)$.
\end{enumerate}
\item\label{definition:ding2}
For any $1\leq i\leq k$, let $(\sX_i,\sL_i)/\A^1$ be a $\T$-equivariant 
test configuration of $(X,L_i)$, and let $(\sX,\sL)/\A^1$ be the sum configuration 
of $\left\{(\sX_i,\sL_i)\right\}_{i=1}^k$. 
\begin{enumerate}
\renewcommand{\theenumii}{\roman{enumii}}
\renewcommand{\labelenumii}{(\theenumii)}
\item\label{definition:ding21}
We set 
\[
\JJ^{\cp}\left(\left\{\sX_i,\sL_i\right\}_{i=1}^k\right):=
\JJ^{\cp}\left(\left\{\sF_{\sX_i,\sL_i}\right\}_{i=1}^k\right)
=\sum_{i=1}^k\JJ\left(\sX_i,\sL_i\right). 
\]
\item\label{definition:ding22}
We set 
\[
\JJ^{\cp}_{\T}\left(\left\{\sX_i,\sL_i\right\}_{i=1}^k\right):=
\JJ^{\cp}_{\T}\left(\left\{\sF_{\sX_i,\sL_i}\right\}_{i=1}^k\right)
=\inf_{\xi\in N_\Q(\T)}\sum_{i=1}^k \JJ\left(\sX_{i,\xi},\sL_{i,\xi}\right), 
\]
where the last equality follows from the observations in \eqref{definition:ding1} 
and Definition \ref{definition:J}. 
\item\label{definition:ding23}
The \emph{coupled Ding invariant of $\{(\sX_i,\sL_i)\}_{i=1}^k$} is defined as 
\[
\Ding\left(\{\sX_i,\sL_i\}_{i=1}^k\right)
:=\Ding(\sX,\sL)+\frac{\left(\bar{\sL}^{\cdot n+1}\right)}{(n+1)(L^{\cdot n})}
-\sum_{i=1}^k \frac{\left(\bar{\sL_i}^{\cdot n+1}\right)}{(n+1)(L_i^{\cdot n})}, 
\]
where 
$\Ding(\sX,\sL)$ is the Ding invariant of (the normalization of) $(\sX,\sL)/\A^1$ 
in the sense of \cite[Definition 2.24]{Xu}, and $(\bar{\sX},\bar{\sL})/\pr^1$, 
$(\bar{\sX_i},\bar{\sL_i})/\pr^1$ are the $\infty$-trivial compactifications of 
$(\sX,\sL)/\A^1$, $(\sX_i,\sL_i)/\A^1$ in the sense of \cite[Definition 2.7]{Xu},  
respectively. 
\end{enumerate}
\end{enumerate}
\end{definition}

\begin{definition}\label{definition:c-ding}
\begin{enumerate}
\renewcommand{\theenumi}{\arabic{enumi}}
\renewcommand{\labelenumi}{(\theenumi)}
\item\label{definition:c-ding1}
We say that $\left(X,\Delta; \{L_i\}_{i=1}^k\right)$ is 
\emph{$\T$-equivariantly coupled Ding semistable} if
\[
\Ding\left(\left\{\sX_i,\sL_i\right\}_{i=1}^k\right)\geq 0
\]
holds for any 
$\T$-equivariant test configuration 
$(\sX_i,\sL_i)/\A^1$ of $(X, L_i)$ $(1\leq i\leq k)$.
When $\T$ is trivial (i.e., $p=0$), we simply say that 
$\left(X,\Delta; \{L_i\}_{i=1}^k\right)$ is 
\emph{coupled Ding semistable}. 

We remark that, for any subtorus $\T'\subset\T$, if 
$\left(X,\Delta; \{L_i\}_{i=1}^k\right)$ is 
$\T'$-equivariantly coupled Ding semistable, then 
$\left(X,\Delta; \{L_i\}_{i=1}^k\right)$ is obviously 
$\T$-equivariantly coupled Ding semistable. 
\item\label{definition:c-ding2}
We say that $\left(X,\Delta; \{L_i\}_{i=1}^k\right)$ is 
\emph{$\T$-reduced uniformly coupled Ding stable} if there exists a positive 
real number $\varepsilon\in\R_{>0}$ such that 
\[
\Ding\left(\left\{\sX_i,\sL_i\right\}_{i=1}^k\right)\geq \varepsilon\cdot
\JJ_{\T}^{\cp}\left(\{\sX_i,\sL_i\}_{i=1}^k\right)
\]
holds for any 
$\T$-equivariant test configuration 
$(\sX_i,\sL_i)/\A^1$ of $(X, L_i)$ $(1\leq i\leq k)$.
If $\T$ is trivial, then we simply say that
$\left(X,\Delta; \{L_i\}_{i=1}^k\right)$ is 
\emph{uniformly coupled Ding stable}. 
If $\T\subset\Aut(X,\Delta)$ is a maximal torus, then we say that 
$\left(X,\Delta; \{L_i\}_{i=1}^k\right)$ is 
\emph{reduced uniformly coupled Ding stable}. The notion does not depend on the 
choice of maximal tori. 
\end{enumerate}
\end{definition}

\begin{lemma}\label{lemma:ding-normal}
In the above definitions of $\T$-equivariant coupled Ding semistability and 
$\T$-reduced uniform coupled Ding stability, we may assume that 
the test configurations $(\sX_i,\sL_i)/\A^1$ $(1\leq i\leq k)$ are normal. 
In fact, for the normalizations $\nu\colon \sX'_i\to\sX_i$, we have 
\[
\Ding\left(\left\{\sX'_i,\nu^*\sL_i\right\}_{i=1}^k\right)
=\Ding\left(\left\{\sX_i,\sL_i\right\}_{i=1}^k\right),\quad
\JJ_{\T}^{\cp}\left(\left\{\sX'_i,\nu^*\sL_i\right\}_{i=1}^k\right)
=\JJ_{\T}^{\cp}\left(\left\{\sX_i,\sL_i\right\}_{i=1}^k\right).
\]
\end{lemma}

\begin{proof}
Observe that the normalized sum configurations of 
$\left\{\left(\sX_i,\sL_i\right)\right\}_{i=1}^k$
and $\left\{\left(\sX'_i,\nu^*\sL_i\right)\right\}_{i=1}^k$ are equal by 
Lemma \ref{lemma:sum-tc} \eqref{lemma:sum-tc3}. 
Moreover, for any $\xi\in N_\Q(\T)$ and $1\leq i\leq k$, we get 
\[
\JJ\left(\sX'_{i,\xi},\left(\nu^*\sL_i\right)_\xi\right)
=\JJ\left(\sX_{i,\xi},\sL_{i,\xi}\right)
\]
by the definition of $\JJ\left(\sX_{i,\xi},\sL_{i,\xi}\right)$ (see 
Definition \ref{definition:J}). 
Thus we get the assertions. 
\end{proof}

We see an analogous statement in \cite[\S 3.4]{Xu}. 

\begin{proposition}[{cf.\ \cite[Theorem 3.63]{Xu}}]\label{proposition:D-Ding}
For any $1\leq i\leq k$, let $(\sX_i,\sL_i)/\A^1$ be a $\T$-equivariant 
test configuration of $(X,L_i)$. 
\begin{enumerate}
\renewcommand{\theenumi}{\arabic{enumi}}
\renewcommand{\labelenumi}{(\theenumi)}
\item\label{proposition:D-Ding1}
We have the inequality 
\[
\Ding\left(\{\sX_i,\sL_i\}_{i=1}^k\right)\leq
\DD\left(\{\sF_{\sX_i,\sL_i}\}_{i=1}^k\right).
\]
\item\label{proposition:D-Ding2}
If moreover all $\sX_i$ are normal $(1\leq i\leq k)$, then we have 
\[
\Ding\left(\{\sX_i,\sL_i\}_{i=1}^k\right)=
\DD\left(\{\sF_{\sX_i,\sL_i}\}_{i=1}^k\right).
\]
\end{enumerate}
\end{proposition}

\begin{proof}
Let $(\sX',\sL')/\A^1$ be the sum configuration of 
$\left\{(\sX_i,\sL_i)\right\}_{i=1}^k$, and let $\nu\colon\sX\to \sX'$ be the 
normalization and set $\sL:=\nu^*\sL'$. We set $\sF_i:=\sF_{\sX_i,\sL_i}$ and 
$\sF:=\sF_{\sX',\sL'}$. We know that $\sF$ is the sum filtration of $\{\sF_i\}_{i=1}^k$. 
Note that $\sF$ is a $\Z$-valued filtration such that, for any sufficiently divisible 
$m\in r\Z_{>0}$, we have $\sI_m(\sF)^s=\sI_{m s}(\sF)$ for any $s\in\Z_{>0}$. 
Moreover, $(\sX,\sL)/\A^1$ is nothing but the normalized blowup test configuration 
along $\sI_m(\sF)$ (see Definition \ref{definition:blowup-tc}). 

\eqref{proposition:D-Ding1}
We have 
\begin{eqnarray*}
\Ding\left(\{\sX_i,\sL_i\}_{i=1}^k\right)
&=&\Ding\left(\sX,\sL\right)+S_L(\sF)-\sum_{i=1}^k S_{L_i}(\sF_i) \\
&\geq&\DD(\sF)+S_L(\sF)-\sum_{i=1}^k S_{L_i}(\sF_i)
=\DD\left(\{\sF_i\}_{i=1}^k\right),
\end{eqnarray*}
where the first equality follows from \cite[Lemma 3.35]{Xu}, and the second 
inequality follows from \cite[Theorem 3.63]{Xu}. 

\eqref{proposition:D-Ding2}
Assume that all $\sX_i$ are normal. We use the notations in \S \ref{section:sum-tc}, 
especially, Definition \ref{definition:sum-tc}, Lemma \ref{lemma:sum-tc-ideal} 
and Proposition \ref{lproposition:sum-Di}. 
Fix a sufficiently divisible $m\in r\Z_{>0}$ and consider 
\[
\sI_m^{\langle i\rangle}:=
\bigoplus_{\lambda\in\Z}t^{-\lambda}I_{(m;\lambda)}^{\langle i\rangle}
=\sI_m(\sF_i)
\]
with $I_{(m;\lambda)}^{\langle i\rangle}:=I_{(m;\lambda)}(\sF_i)$
and 
\[
\sI_m:=
\bigoplus_{\lambda\in\Z}t^{-\lambda}I_{(m;\lambda)}=\sI_m(\sF)
\]
with $I_{(m;\lambda)}:=I_{(m;\lambda)}(\sF)$. 
Set $\sO_{\sZ}(-m\sD_i):=\theta^{-1}\sI_m^{\langle i\rangle}$. As in Proposition 
\ref{lproposition:sum-Di}, we have 
\[
\sI_m=\prod_{i=1}^k\sI_m^{\langle i\rangle}
\]
and 
\[
\sum_{i=1}^k\sD_i=\theta^*L_{\A^1}-\rho^*\sL. 
\]
Let us recall the $\bL$-invariant $\bL(\sF)$ of $\sF$ in the sense of 
\cite[Definition 3.51]{Xu}. Since $m\in r\Z_{>0}$ is sufficiently divisible, we have 
\begin{eqnarray*}
\bL(\sF)&=&\lct\left(X_{\A^1},\Delta_{\A^1}+\sI_m^{\frac{1}{m}};X\times\{0\}\right)
-1\\
&=&\lct\left(\sZ,\Delta_{\sZ}+E+\sum_{i=1}^k\sD_i; \sZ_0\right)-1, 
\end{eqnarray*}
where $\sZ_0$ is the fiber of $\sZ\to\A^1$ at $0\in\A^1$,
and $E$ is the $\Q$-divisor on $\sZ$ supported on $\sZ_0$ defined by 
\[
K_{\sZ}+\Delta_{\sZ}+E=\theta^*\left(K_{X_{\A^1}}+\Delta_{\A^1}\right).
\]
Let us consider the $\Q$-divisor $\sD_{\sX,\sL}$ on $\sX$ in the sense of 
\cite[Definition 2.24]{Xu}, i.e., the $\Q$-divisor on $\sX$ supported on $\sX_0$ 
with 
\[
\sD_{\sX,\sL}\sim_\Q-\bar{\sL}-K_{\bar{\sX}/\pr^1}-\Delta_{\bar{\sX}}.
\]
Then we have 
\[
\rho^*\left(K_{\sX}+\Delta_{\sX}+\sD_{\sX,\sL}\right)
=K_{\sZ}+\Delta_{\sZ}+E+\sum_{i=1}^k\sD_i.
\]
This implies that 
\[
\bL(\sF)=\lct\left(\sX,\Delta_{\sX}+\sD_{\sX,\sL}; \sX_0\right)-1, 
\]
where $\sX_0$ is the fiber of $\sX\to\A^1$ over $0\in\A^1$. 
On the other hand, by \cite[Theorem 3.52]{Xu}, we have $\bL(\sF)=\mu(\sF)$. 
Thus 
\begin{eqnarray*}
\Ding\left(\{\sX_i,\sL_i\}_{i=1}^k\right)&=&
\lct\left(\sX,\Delta_{\sX}+\sD_{\sX,\sL}; \sX_0\right)-1-\sum_{i=1}^k
\frac{\left(\bar{\sL_i}^{\cdot n+1}\right)}{(n+1)(L_i^{\cdot n})}\\
&=&\mu(\sF)-\sum_{i=1}^k S_{L_i}(\sF_i)=\DD\left(\{\sF_i\}_{i=1}^k\right)
\end{eqnarray*}
holds. where the second equality follows from \cite[Lemma 3.35]{Xu}. 
\end{proof}

\begin{lemma}[{cf.\ \cite[Corollary 6.25]{Xu}}]\label{lemma:D-twist}
For any $1\leq i\leq k$, let $\sF_i$ be a $\T$-equivariant filtration on $R^i$. 
For any $\xi\in N_\R(\T)$, we have 
\[
\DD\left(\{\sF_{i,\xi}\}_{i=1}^k\right)=\DD\left(\{\sF_i\}_{i=1}^k\right)
-\langle\alpha_{\bc}^{\cp},\xi\rangle. 
\]
\end{lemma}

\begin{proof}
Let $\sF$ be the sum filtration of $\{\sF_i\}_{i=1}^k$. By \cite[Lemma 6.24]{Xu}, 
we know that $\mu(\sF)=\mu(\sF_\xi)$. Moreover, by \cite[Lemma 6.4]{Xu} 
and Proposition \ref{proposition:sum-filt-twist} \eqref{proposition:sum-filt-twist2}, 
we get
\begin{eqnarray*}
\DD\left(\{\sF_{i,\xi}\}_{i=1}^k\right)
&=&\mu(\sF_\xi)-\sum_{i=1}^k S_{L_i}(\sF_{i,\xi})\\
&=&\mu(\sF)-\sum_{i=1}^k \left(S_{L_i}(\sF_i)+\left\langle
\alpha_{\bc}^{L_i},\xi\right\rangle\right)\\
&=&\DD\left(\{\sF_i\}_{i=1}^k\right)-\sum_{i=1}^k \left\langle
\alpha_{\bc}^{L_i},\xi\right\rangle.
\end{eqnarray*}
Thus the assertion follows. 
\end{proof}

\begin{corollary}\label{corollary:twist-ding}
For any $1\leq i\leq k$, let $(\sX_i,\sL_i)/\A^1$ be a $\T$-equivariant 
\emph{normal} test configuration of $(X,L_i)$. For any $\xi\in N(\T)$, 
let $(\sX_{i,\xi},\sL_{i,\xi})/\A^1$ be the $\xi$-twisted test configuration of 
$(\sX_i,\sL_i)/\A^1$. Then we have 
\begin{eqnarray*}
\Ding\left(\{\sX_{i,\xi},\sL_{i,\xi}\}_{i=1}^k\right)&=&
\DD\left(\left\{\sF_{\sX_{i,\xi},\sL_{i,\xi}}\right\}_{i=1}^k\right)\\
&=&\DD\left(\left\{\sF_{\sX_i,\sL_i}\right\}_{i=1}^k\right)
-\langle\alpha_{\bc}^{\cp},\xi\rangle\\
&=&\Ding\left(\{\sX_i,\sL_i\}_{i=1}^k\right)-\langle\alpha_{\bc}^{\cp},\xi\rangle.
\end{eqnarray*}
In particular, if $\left(X,\Delta; \{L_i\}_{i=1}^k\right)$ is 
$\T$-equivariantly coupled Ding semistable, then $\alpha_{\bc}^{\cp}=0$ holds, i.e., 
$\left(X,\Delta; \{L_i\}_{i=1}^k\right)$ has vanishing coupled $\T$-Futaki characters. 
\end{corollary}

\begin{proof}
The first and third equalities follow from Proposition \ref{proposition:D-Ding}. 
Since $\sF_{\sX_{i,\xi},\sL_{\sL_{i,\xi}}}=\left(\sF_{\sX_i,\sL_i}\right)_\xi$ 
(see Example \ref{example:twisted-tc}), 
the second equality follows from Lemma \ref{lemma:D-twist}. 
\end{proof}

\begin{lemma}\label{lemma:reduction-Z}
For any $1\leq i\leq k$, let $\sF_i$ be a $\T$-equivariant filtration on $R^i$. 
Set $\sG_i:=\sF_{i,\Z}$ (see Definition \ref{definition:approx} 
\eqref{definition:approx1}). Then, for any $\xi\in N_\R(\T)$ and $\delta\in\R_{>0}$, 
we have the equality
\[
\mu\left(\left\{\sF_{i,\xi}\right\}_{i=1}^k\right)
=\mu\left(\left\{\sG_{i,\xi}\right\}_{i=1}^k\right).
\]
In particular, together with Lemma \ref{lemma:Z-twist}, we have 
\begin{eqnarray*}
\JJ^{\cp}\left(\left\{\sF_{i,\xi}\right\}_{i=1}^k\right)
&=&\JJ^{\cp}\left(\left\{\sG_{i,\xi}\right\}_{i=1}^k\right), \\
\JJ^{\cp}_{\T}\left(\left\{\sF_i\right\}_{i=1}^k\right)
&=&\JJ^{\cp}_{\T}\left(\left\{\sG_i\right\}_{i=1}^k\right), \\
\DD\left(\left\{\sF_{i,\xi}\right\}_{i=1}^k\right)
&=&\DD\left(\left\{\sG_{i,\xi}\right\}_{i=1}^k\right).
\end{eqnarray*}
\end{lemma}

\begin{proof}
Let $\sF$ (resp., $\sG$) be the sum filtration of $\{\sF_i\}_{i=1}^k$
(resp., $\{\sG_i\}_{i=1}^k$). Take any $x\in\R$ and $\varepsilon\in\R_{>0}$. 
As in the proof of Lemma \ref{lemma:Z-twist}, if we take $m\in r\Z_{>0}$ with 
$m\varepsilon>k$, then we have 
\begin{eqnarray*}
\sF_\xi^{m x}R_m&=&\sum_{x_1+\cdots+x_k\geq m x}
\sF_{1,\xi}^{x_1}R_m^1\cdots\sF_{k,\xi}^{x_k}R_m^k\\
&\supset&\sum_{x_1+\cdots+x_k\geq m x}
\sG_{1,\xi}^{x_1}R_m^1\cdots\sG_{k,\xi}^{x_k}R_m^k=\sG_\xi^{m x}R_m, \\
\sG_\xi^{m(x-\varepsilon)}R_m&=&\sum_{x'_1+\cdots+x'_k\geq m(x-\varepsilon)}
\sG_{1,\xi}^{x'_1}R_m^1\cdots\sG_{k,\xi}^{x'_k}R_m^k\\
&\supset&\sum_{x_1+\cdots+x_k\geq m x}
\sF_{1,\xi}^{x_1}R_m^1\cdots\sF_{k,\xi}^{x_k}R_m^k=\sF_\xi^{m x}R_m.
\end{eqnarray*}
Note that the second inclusion follows from the observation in Lemma 
\ref{lemma:Z-twist} that 
\[
\sG_{i,\xi}^{x'_i}R_m^i\supset\sF_{i,\xi}^{x_i}R_m^i
\]
holds for any $x_i,x'_i\in\R$ with $x'_i>x_i-1$. 
Thus we get 
\[
\sG_\xi^{m x}R_m\subset \sF_\xi^{m x}R_m\subset \sG_\xi^{m (x-\varepsilon)}R_m, 
\]
which gives that $\mu\left(\left\{\sF_{i,\xi}\right\}_{i=1}^k\right)
=\mu\left(\left\{\sG_{i,\xi}\right\}_{i=1}^k\right)$. 
The remaining assertions are trivial. 
\end{proof}

We prepare the following: 

\begin{proposition}[{cf.\ 
\cite[Propositions 6.6 and 6.29]{Xu}}]\label{proposition:technical}
Assume that $\left(X,\Delta; \{L_i\}_{i=1}^k\right)$ has vanishing coupled $\T$-Futaki 
characters. 
\begin{enumerate}
\renewcommand{\theenumi}{\arabic{enumi}}
\renewcommand{\labelenumi}{(\theenumi)}
\item\label{proposition:technical1}
Set $C:=\dist\left(0,\partial\left(\PP^L\right)\right)$. 
For any $1\leq i\leq k$, let $\sF_i$ be a $\T$-equivariant filtration on $R^i$, and 
let $\sF$ be the sum filtration of $\{\sF_i\}_{i=1}^k$. Then, for every 
$e_-\in\R$ with $\sF^{e_- m}R_m=R_m$ for any $m\in r\Z_{> 0}$, we have 
\[
\lambda_{\max}(\sF_\xi)\geq C|\xi|+e_- 
\]
for any $\xi\in N_\R(\T)$. 
\item\label{proposition:technical2}
For any $1\leq i\leq k$, let $\sF_i$ be a $\T$-equivariant filtration on $R^i$, 
and take any approximating sequence $\left\{\sF_{i,(m)}\right\}_{m\in r\Z_{>0}}$ 
of $\sF_i$. Then we have 
\[
\lim_{m\to\infty}\JJ_{\T}^{\cp}\left(\left\{\sF_{i,(m)}\right\}_{i=1}^k\right)
=\JJ_{\T}^{\cp}\left(\{\sF_i\}_{i=1}^k\right).
\]
\end{enumerate}
\end{proposition}

\begin{proof}
\eqref{proposition:technical1}
By Lemma \ref{lemma:sum-basics} \eqref{lemma:sum-basics1}, we have 
$0=\alpha_{\bc}^{\cp}\in\interior\left(\PP^L\right)$. In particular, 
we have $C\in\R_{>0}$. 
For any $m\in r\Z_{>0}$ and $\alpha\in\Lambda_m^L$, let us set 
\[
T_{m,\alpha}:=\frac{1}{m}\max\left\{x\in\R\,\,|\,\,\sF^x R_{m,\alpha}\neq 0\right\}.
\]
From the assumption, we have $T_{m,\alpha}\geq e_-$ for any 
$\alpha\in\Lambda_m^L$. Thus we get
\[
T_m(\sF_\xi)=\max_{\alpha\in\Lambda_m^L}\left\{
T_{m,\alpha}+\left\langle\frac{\alpha}{m},\xi\right\rangle\right\}
\geq e_-+\frac{1}{m}\max_{\alpha\in\Lambda_m^L}\left\{\langle
\alpha,\xi\rangle\right\},
\]
where $T_m$ be as in \cite[\S 3.1.2]{Xu}. By \cite[Lemma 3.22]{Xu}, we get 
\[
\lambda_{\max}(\sF_\xi)\geq \max_{\alpha\in\PP^L}\left\{\langle
\alpha,\xi\rangle\right\}+e_-
\geq C|\xi|+e_-.
\]

\eqref{proposition:technical2}
Let $\sF$ (resp., $\sF_{(m)}$) be the sum filtration of $\{\sF_i\}_{i=1}^k$ 
(resp., $\{\sF_{i,(m)}\}_{i=1}^k$). Take any $\xi \in N_\R(\T)$. 
By Lemma \ref{lemma:approx} \eqref{lemma:approx2},
$\left\{\sF_{i,(m),\xi}\right\}_{m\in r\Z_{>0}}$ is an approximating sequence of 
$\sF_{i,\xi}$. 
Moreover, by Proposition \ref{proposition:sum-filt-twist} 
\eqref{proposition:sum-filt-twist2}, 
the sum filtration of $\left\{\sF_{i,(m),\xi}\right\}_{i=1}^k$
(resp., $\left\{\sF_{i,\xi}\right\}_{i=1}^k$) is equal to $\sF_{(m),\xi}$ 
(resp., $\sF_\xi$). Therefore, by Lemma \ref{lemma:S-approx} and 
Proposition \ref{proposition:sum-filt-lambda}, we have
\begin{eqnarray*}
S_{L_i}\left(\sF_{i,\xi}\right)&=&\lim_{m\to\infty}S_{L_i}\left(\sF_{i,(m),\xi}\right), \\
\sum_{i=1}^k\lambda_{\max}\left(\sF_{i,\xi}\right)&=&\lambda_{\max}
\left(\sF_\xi\right)
=\lim_{m\to\infty}\lambda_{\max}\left(\sF_{(m),\xi}\right)
=\lim_{m\to\infty}\sum_{i=1}^k\lambda_{\max}\left(\sF_{i,(m),\xi}\right).
\end{eqnarray*}
This implies that, for any $\xi\in N_\R(\T)$, we have 
\[
\lim_{m\to\infty}\JJ^{\cp}\left(\left\{\sF_{i,(m),\xi}\right\}_{i=1}^k\right)
=\JJ^{\cp}\left(\left\{\sF_{i,\xi}\right\}_{i=1}^k\right).
\]
We observe the following claim: 

\begin{claim}\label{claim:technical}
There exist $C\in\R_{>0}$ and $e\in\R$ such that, for any $m\in r\Z_{> 0}$ 
and $\xi\in N_\R(\T)$, we have 
\begin{eqnarray*}
\JJ^{\cp}\left(\left\{\sF_{i,\xi}\right\}_{i=1}^k\right)&\geq&C|\xi|+e, \\
\JJ^{\cp}\left(\left\{\sF_{i,(m),\xi}\right\}_{i=1}^k\right)&\geq&C|\xi|+e.
\end{eqnarray*}
\end{claim}

\begin{proof}[Proof of Claim \ref{claim:technical}]
Observe that 
\begin{eqnarray*}
\JJ^{\cp}\left(\left\{\sF_{i,\xi}\right\}_{i=1}^k\right)&=&\lambda_{\max}(\sF_\xi)
-\sum_{i=1}^k S_{L_i}\left(\sF_{i,\xi}\right)\\
&=&\lambda_{\max}(\sF_\xi)
-\sum_{i=1}^k S_{L_i}\left(\sF_i\right), \\
\JJ^{\cp}\left(\left\{\sF_{i,(m),\xi}\right\}_{i=1}^k\right)
&=&\lambda_{\max}\left(\sF_{(m),\xi}\right)
-\sum_{i=1}^k S_{L_i}\left(\sF_{i,(m)}\right).
\end{eqnarray*}
By \eqref{proposition:technical1}, there exists $C\in\R_{>0}$ such that: 
\begin{itemize}
\item
there exists $e_\infty\in\R$ such that, for any $\xi\in N_\R(\T)$, we have 
\[
\JJ^{\cp}\left(\left\{\sF_{i,\xi}\right\}_{i=1}^k\right)\geq 
C|\xi|+e_\infty-\sum_{i=1}^k S_{L_i}(\sF_i), 
\]
and 
\item
for any $m\in r\Z_{>0}$, there exists $e_m\in\R$
such that, for any $\xi\in N_\R(\T)$, we have 
\[
\JJ^{\cp}\left(\left\{\sF_{i,(m),\xi}\right\}_{i=1}^k\right)\geq 
C|\xi|+e_m-\sum_{i=1}^k S_{L_i}(\sF_{i,(m)}). 
\]
\end{itemize}
Moreover, since $\{\sF_{(m)}\}_{m\in r\Z_{>0}}$ is an approximating sequence of 
$\sF$, we can take the value $e_m$ 
independent of $m\in r \Z_{>0}$. We may assume that 
$e_m=e_\infty$ for any $m\in r \Z_{>0}$. Moreover, since 
$\lim_{m\to \infty}\sum_{i=1}^k S_{L_i}(\sF_{i,(m)})=\sum_{i=1}^k S_{L_i}(\sF_{i})$, 
there exists $e\in\R$ such that 
\begin{eqnarray*}
e_\infty-\sum_{i=1}^k S_{L_i}(\sF_i)&\geq& e, \\
e_m-\sum_{i=1}^k S_{L_i}(\sF_{i,(m)})&\geq &e.
\end{eqnarray*}
Thus we finish to prove Claim \ref{claim:technical}.
\end{proof}

By Claim \ref{claim:technical}, there exists a compact subset 
$\Xi\subset N_\R(\T)$ such that, the minimum of the functions 
\begin{itemize}
\item
$\xi\mapsto\JJ^{\cp}\left(\left\{\sF_{i,\xi}\right\}_{i=1}^k\right)$ and 
\item
$\xi\mapsto\JJ^{\cp}\left(\left\{\sF_{i,(m),\xi}\right\}_{i=1}^k\right)$
for all $m\in r\Z_{>0}$
\end{itemize}
over $N_{\R}(\T)$
can be attained over $\Xi$. 
Moreover, if we set 
\[
C_1:=\max\left\{\dist\left(0,\partial\left(\PP^L\right)\right)\right\}\in\R_{>0}, 
\]
then, for any $\xi_1,\xi_2\in N_\R(\T)$ and $m\in r\Z_{>0}$, we have 
\begin{eqnarray*}
&&\left|\JJ^{\cp}\left(\left\{\sF_{i,(m),\xi_1}\right\}_{i=1}^k\right)
-\JJ^{\cp}\left(\left\{\sF_{i,(m),\xi_2}\right\}_{i=1}^k\right)\right|\\
&=&\left|\lambda_{\max}\left(\left\{\sF_{i,(m),\xi_1}\right\}_{i=1}^k\right)
-\lambda_{\max}\left(\left\{\sF_{i,(m),\xi_2}\right\}_{i=1}^k\right)\right|\\
&\leq& 2 C_1|\xi_1-\xi_2|
\end{eqnarray*}
as in the proof of \cite[Proposition 6.29]{Xu}. Therefore, the sequence 
\[
\left\{\xi\mapsto\JJ^{\cp}\left(\left\{\sF_{i,(m),\xi}\right\}_{i=1}^k\right)
\right\}_{m\in r\Z_{>0}}
\]
of functions over $\Xi$ are Lipschitz continuous with the same Lipschitz 
constant $2 C_1$, and pointwise converges to 
\[
\xi\mapsto\JJ^{\cp}\left(\left\{\sF_{i,\xi}\right\}_{i=1}^k\right).
\]
Hence the sequence of functions uniformly converges to the above function 
over $\Xi$. 
This implies that 
\begin{eqnarray*}
\lim_{m\to\infty}\JJ_{\T}^{\cp}\left(\left\{\sF_{i,(m)}\right\}_{i=1}^k\right)
&=&\lim_{m\to\infty}\inf_{\xi\in\Xi}
\JJ_{\T}^{\cp}\left(\left\{\sF_{i,(m),\xi}\right\}_{i=1}^k\right)\\
&=&\inf_{\xi\in\Xi}
\JJ_{\T}^{\cp}\left(\left\{\sF_{i,\xi}\right\}_{i=1}^k\right)
=\JJ_{\T}^{\cp}\left(\left\{\sF_{i}\right\}_{i=1}^k\right).
\end{eqnarray*}
Thus the assertion \eqref{proposition:technical2} follows. 
\end{proof}

\begin{thm}[{cf.\ \cite[\S 3.4.2 and Proposition 6.3.1]{Xu}}]\label{theorem:DingD}
Assume that $\varepsilon\in [0,1]$ satisfies the following: 
for any $\T$-equivariant test configuration $(\sX_i,\sL_i)/\A^1$ of $(X,L_i)$ 
$(1\leq i\leq k)$, we have 
\[
\Ding\left(\left\{\sX_i,\sL_i\right\}_{i=1}^k\right)
\geq \varepsilon\cdot\JJ_{\T}^{\cp}\left(\left\{\sX_i,\sL_i\right\}_{i=1}^k\right).
\]
Then, for any $\T$-equivariant filtration $\sF_i$ on $R^i$ $(1\leq i\leq k)$, 
we have 
\[
\DD\left(\{\sF_i\}_{i=1}^k\right)\geq \varepsilon\cdot
\JJ_{\T}^{\cp}\left(\{\sF_i\}_{i=1}^k\right). 
\]
\end{thm}

\begin{proof}
By Corollary \ref{corollary:twist-ding}, we have $\alpha_{\bc}^{\cp}=0$. 
Moreover, by Lemma \ref{lemma:reduction-Z}, we may assume that 
each $\sF_i$ is $\Z$-valued. Let $\sF$ be the sum filtration of $\{\sF_i\}_{i=1}^k$. 
For any $1\leq i\leq k$, take any $\Z$-valued approximating sequence 
$\left\{\sF_{i,(m)}\right\}_{m\in r\Z_{>0}}$ of $\sF_i$, and let $\sF_{(m)}$ be the 
sum filtration of $\left\{\sF_{i,(m)}\right\}_{i=1}^k$ for any $m\in r\Z_{>0}$. 
Recall that, by Proposition \ref{proposition:approx-sum-filt}, 
$\left\{\sF_{(m)}\right\}_{m\in r\Z_{>0}}$ is an approximating sequence of $\sF$. 

Let us set 
\begin{eqnarray*}
I_{(m;x)}^{\langle i\rangle}&:=&
I_{(m;x)}(\sF_i)=I_{(m;x)}(\sF_{i,(m)}), \\
I_{(m;x)}&:=&I_{(m;x)}(\sF)=I_{(m;x)}(\sF_{(m)})
\end{eqnarray*}
as in Definitions \ref{definition:filtration} and \ref{definition:sum-tc}. 
Since $\sF_i$ is $\Z$-valued, for any $\lambda\in\Z$, we have 
\[
I_{(m;\lambda)}=\sum_{\substack{\lambda_1,\dots,\lambda_k\in\Z; \\
\lambda_1+\cdots+\lambda_k\geq \lambda}}
I_{(m;\lambda_1)}^{\langle 1\rangle}\cdots I_{(m;\lambda_k)}^{\langle k\rangle}. 
\]
Moreover, let us consider the fractional ideals 
\begin{eqnarray*}
\sI_m^{\langle i\rangle}&:=&\sI_m(\sF_i)=\sI_m(\sF_{i,(m)}), \\
\sI_m&:=&\sI_m(\sF)=\sI_m(\sF_{(m)}) 
\end{eqnarray*}
on $X_{\A^1}$. As above, we have 
\[
\sI_m=\prod_{i=1}^k\sI_m^{\langle i\rangle}. 
\]

For any $1\leq i\leq k$ and $m\in r\Z_{>0}$, let 
$\left(\sX_{i,(m)},\sL_{i,(m)}\right)/\A^1$ be the normalized blowup test configuration 
of $(X, L_i)$ along $\sI_m^{\langle i\rangle}$ in the sense of 
Definition \ref{definition:blowup-tc}. 
We recall the definition. Take a $(\G_m\times\T)$-equivariant birational morphism 
\[
p_m\colon\sW_m\to X_{\A^1}
\]
with $\sW_m$ normal and $p_m$ an isomorphism over $\A^1\setminus\{0\}$ and 
there exists a Cartier divisor $m D_m^{\langle i\rangle}$ on $\sW_m$ such that 
$p_m^{-1}\sI_m^{\langle i\rangle}=\sO_{\sW_m}\left(-m D_m^{\langle i\rangle}\right)$ 
holds for any $1\leq i\leq k$. The $\left(\sX_{i,(m)},\sL_{i,(m)}\right)/\A^1$ is 
defined to be the ample model of the semiample $\Q$-divisor 
\[
p_m^*(L_i)_{\A^1}-D_m^{\langle i\rangle}
\]
over $\A^1$. Since 
\[
p_m^{-1}\sI_m=\sO_{\sW_m}\left(-m \sum_{i=1}^k D_m^{\langle i\rangle}\right),
\]
the normalized blowup test configuration $\left(\sX_{(m)},\sL_{(m)}\right)/\A^1$ 
of $(X, L)$ along $\sI_m$ is equal to the normalized sum configuration of 
$\left\{\left(\sX_{i,(m)},\sL_{i,(m)}\right)\right\}_{i=1}^k$. 

Let us take any $\xi\in N_\Q(\T)$. Fix $e\in\Z_{>0}$ with $e\xi\in N(\T)$. 
For any $1\leq i\leq k$ and $m\in r\Z_{>0}$, set 
\[
\left(\sY_{i,(m)},\sM_{i,(m)}\right):=
\left(\left(\left(\sX_{i,(m)}\right)^{\overline{(e)}}\right)_{e\xi},
\left(\left(\sL_{i,(m)}\right)^{\overline{(e)}}\right)_{e\xi}\right)
\]
(see Definition \ref{definition:base-change}). Moreover, we set 
\[
\sG_i:=\left(\left(\sF_i\right)^{(e)}\right)_{e\xi}, \quad
\sG:=\left(\sF^{(e)}\right)_{e\xi}, \quad
\sG_{i,(m)}:=\left(\left(\sF_{i,(m)}\right)^{(e)}\right)_{e\xi}, \\
\sG_{(m)}:=\left(\left(\sF_{(m)}\right)^{(e)}\right)_{e\xi}
\]
(in the sense of Proposition \ref{proposition:tc-twist} 
\eqref{proposition:tc-twist2}). 
By Propositions \ref{proposition:sum-filt-twist}, \ref{proposition:approx-sum-filt}, 
\ref{proposition:approx-filt-e} and Lemma \ref{lemma:approx} \eqref{lemma:approx2}, 
we have: 
\begin{itemize}
\item
$\sG$ is the sum filtration of $\{\sG_i\}_{i=1}^k$. 
\item
$\left\{\sG_{(m)}\right\}_{m\in r\Z_{>0}}$ is an approximating sequence of $\sG$. 
\item
$\sG_{(m)}$ is the sum filtration of $\left\{\sG_{i,(m)}\right\}_{i=1}^k$. 
\item
$\left\{\sG_{i,(m)}\right\}_{m\in r\Z_{>0}}$ is an approximating sequence of 
$\sG_i$. 
\end{itemize}
Moreover, by Proposition \ref{proposition:tc-twist}, the test configuration 
$\left(\sY_{i,(m)},\sM_{i,(m)}\right)/\A^1$ is equal to the normalized blowup 
test configuration of $(X, L_i)$ along 
\[
\sJ_m^{\langle i\rangle}:=\sI_m(\sG_i)=\sI_m\left(\sG_{i,(m)}\right).
\]
We again recall the definition. 
Take a $(\G_m\times\T)$-equivariant birational morphism 
\[
q_m\colon\sZ_m\to X_{\A^1}
\]
with $\sZ_m$ normal and $q_m$ an isomorphism over $\A^1\setminus\{0\}$ and 
there exists a Cartier divisor $m E_m^{\langle i\rangle}$ on $\sZ_m$ such that 
$q_m^{-1}\sJ_m^{\langle i\rangle}=\sO_{\sZ_m}\left(-m E_m^{\langle i\rangle}\right)$ 
holds for any $1\leq i\leq k$. The $\left(\sY_{i,(m)},\sM_{i,(m)}\right)/\A^1$ is 
defined to be the ample model of the semiample $\Q$-divisor 
\[
q_m^*(L_i)_{\A^1}-E_m^{\langle i\rangle}
\]
on $\sZ_m$ over $\A^1$. We set 
\[
\sJ_m:=\sI_m(\sG)=\sI_m\left(\sG_{(m)}\right)=\prod_{i=1}^k\sJ_m^{\langle i\rangle},
\]
where the last equality follows from the property that each $\sG_i$ is $\Z$-valued.
Since 
\[
q_m^{-1}\sJ_m=\sO_{\sZ_{m}}\left(-m\sum_{i=1}^k E_m^{\langle i\rangle}\right), 
\]
the normalized blowup test configuration 
$\left(\sY_{(m)},\sM_{(m)}\right)/\A^1$ of $(X, L)$ along $\sJ_m$ is equal to 
the normalized sum configuration of 
$\left\{\left(\sY_{i,(m)},\sM_{i,(m)}\right)\right\}_{i=1}^k$. 
Moreover, again by Proposition 
\ref{proposition:tc-twist}, we have 
\[
\left(\sY_{(m)},\sM_{(m)}\right)=\left(
\left(\left(\sX_{(m)}\right)^{\overline{(e)}}\right)_{e\xi}, 
\left(\left(\sL_{(m)}\right)^{\overline{(e)}}\right)_{e\xi}\right). 
\]
Let $\theta_m\colon\sZ_m\to\sY_{(m)}$ be the natural morphism. 
Note that 
\[
q_m^*L_{\A^1}-\sum_{i=1}^k E_m^{\langle i\rangle}
=-\theta_m^*\left(K_{\sY_{(m)}}+\Delta_{\sY_{(m)}}+\sD_{\sY_{(m)},\sM_{(m)}}\right)
\]
holds, where $\sD_{\sY_{(m)},\sM_{(m)}}$ be as in \cite[Definition 2.24]{Xu}. 
Therefore, since $\left\{\sG_{(m)}\right\}_{m\in r_{\Z_{>0}}}$ is an approximating 
sequence, we have 
\begin{eqnarray*}
\bL\left(\sG_{(m)}\right)&=&\lct\left(X_{\A^1},\Delta_{\A^1}+\sJ_m^{\frac{1}{m}}; 
X\times\{0\}\right)-1\\
&=&\lct\left(\sY_{(m)},\Delta_{\sY_{(m)}}+\sD_{\sY_{(m)},\sM_{(m)}}; 
\sY_{(m),0}\right)-1,
\end{eqnarray*}
where $\sY_{(m),0}$ is the fiber of $\sY_{(m)}\to\A^1$ at $0\in\A^1$. 

Observe that 
\begin{eqnarray*}
&&\Ding\left(\left\{\sX_{i,(m)}, \sL_{i,(m)}\right\}_{i=1}^k\right)
-\varepsilon\cdot
\JJ^{\cp}\left(\left\{\sX_{i,(m),\xi}, \sL_{i,(m),\xi}\right\}_{i=1}^k\right)\\
&=&\Ding\left(\sX_{(m)}, \sL_{(m)}\right)
+\frac{\left(\bar{\sL}_{(m)}^{\cdot n+1}\right)}{(n+1)(L^{\cdot n})}
-\sum_{i=1}^k\frac{\left(\bar{\sL}_{i,(m)}^{\cdot n+1}\right)}{(n+1)(L_i^{\cdot n})}
-\varepsilon\cdot\frac{1}{e}
\JJ^{\cp}\left(\left\{\sY_{i,(m)}, \sM_{i,(m)}\right\}_{i=1}^k\right)\\
&=&\frac{1}{e}\left(
\Ding\left(\sX_{(m)}^{\overline{(e)}}, \sL_{(m)}^{\overline{(e)}}\right)
+\frac{\left(\overline{\sL_{(m)}^{\overline{(e)}}}^{\cdot n+1}\right)}{(n+1)(L^{\cdot n})}
-\sum_{i=1}^k\frac{\left(\overline{\sL_{i,(m)}^{\overline{(e)}}}^{\cdot n+1}\right)}{(n+1)
(L_i^{\cdot n})}
-\varepsilon\cdot
\JJ^{\cp}\left(\left\{\sY_{i,(m)}, \sM_{i,(m)}\right\}_{i=1}^k\right)\right)\\
&=&\frac{1}{e}\left(
\Ding\left(\left\{\sY_{i,(m)}, \sM_{i,(m)}\right\}_{i=1}^k\right)
-\varepsilon\cdot
\JJ^{\cp}\left(\left\{\sY_{i,(m)}, \sM_{i,(m)}\right\}_{i=1}^k\right)\right),
\end{eqnarray*}
where the second equality follows from \cite[Lemma 2.25]{Xu} and 
the third equality follows from Proposition \ref{proposition:D-Ding}
\eqref{proposition:D-Ding2} and Lemma \ref{lemma:D-twist}. 
On the other hand, we have 
\begin{eqnarray*}
&&\DD\left(\left\{\sF_{i,(m)}\right\}_{i=1}^k\right)
-\varepsilon\cdot\JJ^{\cp}\left(\left\{\sF_{i,(m),\xi}\right\}_{i=1}^k\right)\\
&=&\DD\left(\left\{\sF_{i,(m),\xi}\right\}_{i=1}^k\right)
-\varepsilon\cdot\JJ^{\cp}\left(\left\{\sF_{i,(m),\xi}\right\}_{i=1}^k\right)\\
&=&\DD\left(\sF_{(m),\xi}\right)+S_L\left(\sF_{(m),\xi}\right)
-\sum_{i=1}^k S_{L_i}\left(\sF_{i,(m),\xi}\right)
-\varepsilon\left(\lambda_{\max}\left(\sF_{(m),\xi}\right)-\sum_{i=1}^k
S_{L_i}\left(\sF_{i,(m),\xi}\right)\right)\\
&=&\frac{1}{e}\left(
\DD\left(\sG_{(m)}\right)+S_L\left(\sG_{(m)}\right)
-\sum_{i=1}^k S_{L_i}\left(\sG_{i,(m)}\right)
-\varepsilon\left(\lambda_{\max}\left(\sG_{(m)}\right)-\sum_{i=1}^k
S_{L_i}\left(\sG_{i,(m)}\right)\right)\right)\\
&=&\frac{1}{e}\left(\DD\left(\left\{\sG_{i,(m)}\right\}_{i=1}^k\right)
-\varepsilon\cdot\JJ^{\cp}\left(\left\{\sG_{i,(m)}\right\}_{i=1}^k\right)\right), 
\end{eqnarray*}
where the first equality follows from Lemma \ref{lemma:D-twist} 
and the third equality follows from Lemma \ref{lemma:prepare-e}. 

As in the proof of \cite[Theorem 3.63]{Xu}, more precisely, 
by \cite[equation (3.42)]{Xu}, we have 
\[
S_{L_i}\left(\sG_{i,(m)}\right)\leq S_{L_i}\left(\sF_{\sY_{i,(m)},\sM_{i,(m)}}\right), \quad
\lambda_{\max}\left(\sG_{i,(m)}\right)\leq 
\lambda_{\max}\left(\sF_{\sY_{i,(m)},\sM_{i,(m)}}\right).
\]
Thus we get
\begin{eqnarray*}
&&\DD\left(\left\{\sG_{i,(m)}\right\}_{i=1}^k\right)
-\varepsilon\cdot\JJ^{\cp}\left(\left\{\sG_{i,(m)}\right\}_{i=1}^k\right)\\
&=&\bL\left(\sG_{(m)}\right)-(1-\varepsilon)\cdot\sum_{i=1}^k
S_{L_i}\left(\sG_{i,(m)}\right)-\varepsilon\cdot\lambda_{\max}\left(\sG_{(m)}\right)\\
&\geq&\lct\left(\sY_{(m)},\Delta_{\sY_{(m)}}+\sD_{\sY_{(m)},\sM_{(m)}}; 
\sY_{(m),0}\right)-1\\
&&-(1-\varepsilon)\cdot\sum_{i=1}^k
S_{L_i}\left(\sF_{\sY_{i,(m)},\sM_{i,(m)}}\right)
-\varepsilon\cdot\lambda_{\max}\left(\sF_{\sY_{i,(m)},\sM_{i,(m)}}\right)\\
&=&\Ding\left(\sY_{(m)},\sM_{(m)}\right)
+\frac{\left(\bar{\sM}_{(m)}^{\cdot n+1}\right)}{(n+1)(L^{\cdot n})}
-\sum_{i=1}^k\frac{\left(\bar{\sM}_{i,(m)}^{\cdot n+1}\right)}{(n+1)(L_i^{\cdot n})}
-\varepsilon\cdot\sum_{i=1}^k\JJ\left(\sY_{i,(m)},\sM_{i,(m)}\right)\\
&=&\Ding\left(\left\{\sY_{i,(m)},\sM_{i,(m)}\right\}_{i=1}^k\right)
-\varepsilon\cdot\JJ^{\cp}\left(\left\{\sY_{i,(m)},\sM_{i,(m)}\right\}_{i=1}^k\right). 
\end{eqnarray*}

As a consequence, for any $\xi\in N_\Q(\T)$, we get 
\begin{eqnarray*}
&&\Ding\left(\left\{\sX_{i,(m)}, \sL_{i,(m)}\right\}_{i=1}^k\right)
-\varepsilon\cdot
\JJ^{\cp}\left(\left\{\sX_{i,(m),\xi}, \sL_{i,(m),\xi}\right\}_{i=1}^k\right)\\
&=&\frac{1}{e}\left(
\Ding\left(\left\{\sY_{i,(m)}, \sM_{i,(m)}\right\}_{i=1}^k\right)
-\varepsilon\cdot
\JJ^{\cp}\left(\left\{\sY_{i,(m)}, \sM_{i,(m)}\right\}_{i=1}^k\right)\right)\\
&\leq&\frac{1}{e}\left(\DD\left(\left\{\sG_{i,(m)}\right\}_{i=1}^k\right)
-\varepsilon\cdot\JJ^{\cp}\left(\left\{\sG_{i,(m)}\right\}_{i=1}^k\right)\right)\\
&=&\DD\left(\left\{\sF_{i,(m)}\right\}_{i=1}^k\right)
-\varepsilon\cdot\JJ^{\cp}\left(\left\{\sF_{i,(m),\xi}\right\}_{i=1}^k\right).
\end{eqnarray*}
From the assumption, we get 
\[
\DD\left(\left\{\sF_{i,(m)}\right\}_{i=1}^k\right)\geq
\varepsilon\cdot\JJ^{\cp}_{\T}\left(\left\{\sF_{i,(m)}\right\}_{i=1}^k\right)
\]
for any $m\in r\Z_{>0}$. 

By Proposition \ref{proposition:technical} \eqref{proposition:technical2} and 
\cite[Theorems 3.58 and 3.60]{Xu}, we have 
\begin{eqnarray*}
\lim_{m\to\infty}\JJ_{\T}^{\cp}\left(\left\{\sF_{i,(m)}\right\}_{i=1}^k\right)
=\JJ_{\T}^{\cp}\left(\{\sF_i\}_{i=1}^k\right), &\quad&
\lim_{m\to\infty}S_{L_i}\left(\sF_{i,(m)}\right)=S_{L_i}\left(\sF_i\right), \\
\lim_{m\to\infty}S_L\left(\sF_{(m)}\right)=S_L\left(\sF\right), &\quad&
\lim_{m\to\infty}\DD\left(\sF_{(m)}\right)=\DD(\sF). 
\end{eqnarray*}
In particular, we have 
\begin{eqnarray*}
\lim_{m\to \infty}\DD\left(\left\{\sF_{i,(m)}\right\}_{i=1}^k\right)
&=&\lim_{m\to \infty}\left(\DD\left(\sF_{(m)}\right)+S_L\left(\sF_{(m)}\right)
-\sum_{i=1}^k S_{L_i}\left(\sF_{i,(m)}\right)\right)\\
&=&\DD\left(\sF\right)+S_L\left(\sF\right)
-\sum_{i=1}^k S_{L_i}\left(\sF_i\right)=\DD\left(\{\sF_i\}_{i=1}^k\right).
\end{eqnarray*}
Thus we complete the proof of Theorem \ref{theorem:DingD}. 
\end{proof}

\begin{corollary}\label{corollary:DingD}
\begin{enumerate}
\renewcommand{\theenumi}{\arabic{enumi}}
\renewcommand{\labelenumi}{(\theenumi)}
\item\label{corollary:DingD1}
The following are equivalent: 
\begin{enumerate}
\renewcommand{\theenumii}{\roman{enumii}}
\renewcommand{\labelenumii}{(\theenumii)}
\item\label{corollary:DingD11}
$\left(X,\Delta;\{L_i\}_{i=1}^k\right)$ is $\T$-equivariantly coupled Ding semistable.
\item\label{corollary:DingD12}
For any $\T$-equivariant filtration $\sF_i$ on $R^i$ $(1\leq i\leq k)$, we have 
$\DD\left(\{\sF_i\}_{i=1}^k\right)\geq 0$. 
\end{enumerate}
\item\label{corollary:DingD2}
The following are equivalent: 
\begin{enumerate}
\renewcommand{\theenumii}{\roman{enumii}}
\renewcommand{\labelenumii}{(\theenumii)}
\item\label{corollary:DingD21}
$\left(X,\Delta;\{L_i\}_{i=1}^k\right)$ is $\T$-reduced uniformly 
coupled Ding stable.
\item\label{corollary:DingD22}
There exists $\varepsilon\in\R_{>0}$ such that, 
for any $\T$-equivariant filtration $\sF_i$ on $R^i$ $(1\leq i\leq k)$, we have 
$\DD\left(\{\sF_i\}_{i=1}^k\right)\geq \varepsilon\cdot\JJ_{\T}^{\cp}
\left(\{\sF_i\}_{i=1}^k\right)$. 
\end{enumerate}
\end{enumerate}
\end{corollary}

\begin{proof}
Trivial from Theorem \ref{theorem:DingD}, Lemma \ref{lemma:ding-normal} and 
Proposition \ref{proposition:D-Ding} \eqref{proposition:D-Ding2}. 
\end{proof}

\begin{thm}[{cf.\ \cite[Theorem 4.13]{Xu}}]\label{theorem:uD}
The following are equivalent: 
\begin{enumerate}
\renewcommand{\theenumi}{\arabic{enumi}}
\renewcommand{\labelenumi}{(\theenumi)}
\item\label{theorem:uD1}
There exists $\varepsilon\in\R_{>0}$ such that, 
for any $\T$-equivariant filtration $\sF_i$ on $R^i$ $(1\leq i\leq k)$, we have 
$\DD\left(\{\sF_i\}_{i=1}^k\right)\geq \varepsilon\cdot\JJ_{\T}^{\cp}
\left(\{\sF_i\}_{i=1}^k\right)$. 
\item\label{theorem:uD2}
There exists $\delta\in\R_{>1}$ such that, 
for any $\T$-equivariant filtration $\sF_i$ on $R^i$ $(1\leq i\leq k)$, there exists 
$\xi\in N_{\R}(\T)$ such that 
$\DD\left(\left\{\sF_{i,\xi}\right\}_{i=1}^k;\delta\right)\geq 0$ holds. 
\end{enumerate}
\end{thm}

\begin{proof}
We firstly remark that, each assumption implies that $\alpha_{\bc}^{\cp}=0$ 
by Corollaries \ref{corollary:DingD} and \ref{corollary:twist-ding}.

\eqref{theorem:uD1}$\implies$\eqref{theorem:uD2}: 
(The proof is essentially same as the proof of \cite[Theorem 6.34]{Xu}.)
Take $\delta\in\R_{>1}$ with 
$\delta=\delta\left(\varepsilon, n,\alpha(X,\Delta)\right)>1$ as in 
\cite[Theorem 3.50]{Xu}. (The definition of $\alpha(X,\Delta)\in\R_{>0}$ 
can be found in \cite[Definition 3.38]{Xu} with $L=-(K_X+\Delta)$.)
Take any $\T$-equivariant filtration $\sF_i$ on $R^i$ $(1\leq i\leq k)$. 
We can assume that, there exists $\xi\in N_{\R}(\T)$ such that 
\[
\DD\left(\left\{\sF_{i,\xi}\right\}_{i=1}^k\right)\geq
\varepsilon\cdot\JJ^{\cp}\left(\left\{\sF_{i,\xi}\right\}_{i=1}^k\right). 
\]
Set 
\[
\sG_i:=\sF_{i,\xi,\left[-S_{L_i}\left(\sF_{i,\xi}\right)\right]},
\]
and let $\sG$ be the sum filtration of $\left\{\sG_i\right\}_{i=1}^k$. 
Then we have 
\begin{eqnarray*}
S_{L_i}\left(\sG_i\right)&=&0, \quad
\lambda_{\max}\left(\sG_i\right)
=\lambda_{\max}\left(\sF_{i,\xi}\right)-S_{L_i}\left(\sF_{i,\xi}\right), \quad
\JJ\left(\sG_i\right)=\JJ\left(\sF_{i,\xi}\right),\\
\DD\left(\left\{\sG_i\right\}_{i=1}^k;\delta'\right)&=&
\DD\left(\left\{\sF_{i,\xi}\right\}_{i=1}^k;\delta'\right)
\end{eqnarray*}
for any $\delta'\in\R_{>0}$. In particular, we get the inequality 
\[
\mu\left(\sG\right)\geq \varepsilon\cdot\lambda_{\max}\left(\sG\right). 
\]
The last three lines of the proof of \cite[Theorem 3.50]{Xu} exactly shows that 
$\mu(\sG; \delta)\geq 0$ under the above assumption. 
Since 
\[
\DD\left(\left\{\sF_{i,\xi}\right\}_{i=1}^k;\delta\right)
=\DD\left(\left\{\sG_i\right\}_{i=1}^k;\delta\right)=\mu\left(\sG;\delta\right)\geq 0,
\]
we get the assertion \eqref{theorem:uD2}. 

\eqref{theorem:uD2}$\implies$\eqref{theorem:uD1}: 
Let us fix a small $\varepsilon\in (0,1)$ satisfying 
\[
(1-\varepsilon)\left(1-\frac{1}{\delta}\right)+\varepsilon
\left(1-\frac{1}{\alpha(X,\Delta)}\right)\geq 0.
\]
For any $1\leq i\leq k$, let $\sF_i$ be a $\T$-equivariant filtration on $R^i$, and 
let $\sF$ be the sum filtration of $\{\sF_i\}_{i=1}^k$. There exists 
$\xi\in N_{\R}(\T)$ such that $\DD\left(\left\{\sF_{i,\xi}\right\}_{i=1}^k;\delta\right)
\geq 0$. Recall that $\sF_\xi$ is the sum filtration of $\{\sF_{i,\xi}\}_{i=1}^k$. 

Set $\mu:=\mu\left(\sF_\xi\right)$. If $\mu=\lambda_{\max}\left(\sF_\xi\right)$, 
then we have 
\[
\DD\left(\left\{\sF_{i,\xi}\right\}_{i=1}^k\right)
-\JJ^{\cp}\left(\left\{\sF_{i,\xi}\right\}_{i=1}^k\right)
=\mu\left(\sF_\xi\right)-\lambda_{\max}\left(\sF_\xi\right)=0. 
\]
Thus we may assume that $\mu<\lambda_{\max}\left(\sF_\xi\right)$. 
By Lemma \ref{lemma:lct-compute}, we have 
\[
\lct\left(X,\Delta;I_{\bullet}^{(\mu)}\left(\sF_\xi\right)\right)=1. 
\]
Moreover, there exists $v\in\Val_X^{<\infty,\T}$ with 
$A_{X,\Delta}(v)=v\left(I_{\bullet}^{(\mu)}\left(\sF_\xi\right)\right)$ and 
$\sF_\xi\subset\sF_{v,\left[\mu-A_{X,\Delta}(v)\right]}$. 
In particular, we get 
\[
\lambda_{\max}\left(\sF_\xi\right)\leq\lambda_{\max}
\left(\sF_{v,\left[\mu-A_{X,\Delta}(v)\right]}\right)
=T_L(v)+\mu-A_{X,\Delta}(v), 
\]
where $T_L=T$ is in the sense of \cite[\S 3.1.2]{Xu}. 
Moreover, we have
\[
\mu\left(\sF_\xi;\delta\right)
\leq\mu\left(\sF_{v,\left[\mu-A_{X,\Delta}(v)\right]};\delta\right)
\leq\frac{A_{X,\Delta}(v)}{\delta}+\mu-A_{X,\Delta}(v), 
\]
where the last inequality follows from the argument in \cite[equation (4.10)]{Xu}.
This implies that 
\[
0\leq\DD\left(\left\{\sF_{i,\xi}\right\}_{i=1}^k;\delta\right)\leq
\frac{A_{X,\Delta}(v)}{\delta}+\mu-A_{X,\Delta}(v)-\sum_{i=1}^k 
S_{L_i}\left(\sF_{i,\xi}\right). 
\]
Therefore, we get
\begin{eqnarray*}
&&\DD\left(\left\{\sF_{i}\right\}_{i=1}^k\right)
-\varepsilon\cdot\JJ^{\cp}\left(\left\{\sF_{i,\xi}\right\}_{i=1}^k\right)\\
&=&\DD\left(\left\{\sF_{i,\xi}\right\}_{i=1}^k\right)
-\varepsilon\cdot\JJ^{\cp}\left(\left\{\sF_{i,\xi}\right\}_{i=1}^k\right)\\
&=&\mu-\sum_{i=1}^k S_{L_i}\left(\sF_{i,\xi}\right)-\varepsilon
\left(\lambda_{\max}\left(\sF_\xi\right)-\sum_{i=1}^k
S_{L_i}\left(\sF_{i,\xi}\right)\right)\\
&\geq&(1-\varepsilon)\left(\mu-\sum_{i=1}^k
S_{L_i}\left(\sF_{i,\xi}\right)\right)-\varepsilon\left(T_L(v)-A_{X,\Delta}(v)\right)\\
&\geq&(1-\varepsilon)A_{X,\Delta}(v)\left(1-\frac{1}{\delta}\right)
-\varepsilon\left(T_L(v)-A_{X,\Delta}(v)\right)\\
&\geq&A_{X,\Delta}(v)\left((1-\varepsilon)\left(1-\frac{1}{\delta}\right)+\varepsilon
\left(1-\frac{1}{\alpha(X,\Delta)}\right)\right)\geq 0. 
\end{eqnarray*}
Thus we get the assertion \eqref{theorem:uD1}. 
\end{proof}

\section{The coupled stability threshold}\label{section:c-delta}

In this section, we introduce the notion of the 
$\T$-reduced coupled stability threshold 
and see its basic properties and relationship with coupled Ding stability. 
In this section, we follow the notation in \S \ref{section:c-ding}. 

We begin with the following easy lemma: 

\begin{lemma}\label{lemma:A-S-twist}
For any $v\in\Val_X^{<\infty,\T}$ and $\xi\in N_{\R}(\T)$, we have 
\[
A_{X,\Delta}\left(v_\xi\right)-\sum_{i=1}^k S_{L_i}\left(v_\xi\right)
=A_{X,\Delta}\left(v\right)-\sum_{i=1}^k S_{L_i}\left(v\right)
-\left\langle\alpha_{\bc}^{\cp},\xi\right\rangle. 
\]
In particular, we have 
\[
A_{X,\Delta}\left(\operatorname{wt}_\xi\right)
-\sum_{i=1}^k S_{L_i}\left(\operatorname{wt}_\xi\right)
=-\left\langle\alpha_{\bc}^{\cp},\xi\right\rangle. 
\]
\end{lemma}

\begin{proof}
Follows immediately from Corollary \ref{corollary:S-twist}, Proposition 
\ref{proposition:theta} and Lemma \ref{lemma:sum-basics} 
\eqref{lemma:sum-basics3}. 
\end{proof}

\begin{definition}\label{definition:c-delta}
We set 
\[
\delta_{\T}^{\Red}\left(X,\Delta; \left\{L_i\right\}_{i=1}^k\right)
:=\inf_{v\in\Val_X^{*,\T}}\sup_{\xi\in N_{\R}(\T)}
\frac{A_{X,\Delta}\left(v_\xi\right)}{\sum_{i=1}^k S_{L_i}\left(v_\xi\right)}, 
\]
and call it the \emph{$\T$-reduced coupled stability threshold of 
$\left(X,\Delta; \left\{L_i\right\}_{i=1}^k\right)$}. 
If $\T$ is trivial, then we simply write it 
$\delta\left(X,\Delta; \left\{L_i\right\}_{i=1}^k\right)$ and call it the 
\emph{coupled stability threshold of $\left(X,\Delta; \left\{L_i\right\}_{i=1}^k\right)$}. 
If $\T\subset\Aut(X,\Delta)$ is a maximal torus, then we write it 
$\delta^{\Red}\left(X,\Delta; \left\{L_i\right\}_{i=1}^k\right)$
and call it the \emph{reduced coupled stability threshold of 
$\left(X,\Delta; \left\{L_i\right\}_{i=1}^k\right)$}. The notion of 
reduced coupled stability threshold does not depend on 
the choice of maximal tori as we observed in Definition 
\ref{definition:c-ding}. 
\end{definition}

\begin{remark}\label{remark:c-delta}
\begin{enumerate}
\renewcommand{\theenumi}{\arabic{enumi}}
\renewcommand{\labelenumi}{(\theenumi)}
\item\label{remark:c-delta1}
Assume that $\T$ is trivial. Then, as in the proof of \cite{BlJ} 
(see \cite{coupled-delta} in detail), we have 
\begin{eqnarray*}
&&\delta\left(X,\Delta; \left\{L_i\right\}_{i=1}^k\right)=\inf_{v\in\Val_X^{<\infty}}
\frac{A_{X,\Delta}(v)}{\sum_{i=1}^k S_{L_i}(v)}
=\inf_{v\in\QM_X}
\frac{A_{X,\Delta}(v)}{\sum_{i=1}^k S_{L_i}(v)}\\
&=&\inf_{\substack{E\text{: prime divisor}\\ \text{over }X}}
\frac{A_{X,\Delta}(E)}{\sum_{i=1}^k S_{L_i}(E)}
=\lim_{m\to\infty}\delta_m\left(X,\Delta; \left\{L_i\right\}_{i=1}^k\right),
\end{eqnarray*}
where, for $m\in r\Z_{>0}$ very big, 
\[
\delta_m\left(X,\Delta; \left\{L_i\right\}_{i=1}^k\right)
:=\inf\left\{\lct\left(X,\Delta;\sum_{i=1}^k D_i\right)\,\,\bigg|\,\,
\begin{split}D_i:\text{ $m$-basis type $\Q$-divisor}\\ \text{of $L_i$ for 
any }1\leq i\leq k\end{split}
\right\}.
\]
For the definition of basis type $\Q$-divisors, see \cite{BlJ} 
or \cite{coupled-delta}. 
\item\label{remark:c-delta2}
Assume that $(X,\Delta)$ is a toric pair with $\T=\G_m^n$ a maximal torus. 
Then $\Val_X^{*,\T}=\emptyset$. Therefore, we must have 
$\delta^{\Red}\left(X,\Delta; \left\{L_i\right\}_{i=1}^k\right)=\infty$. 
\end{enumerate}
\end{remark}

We consider a coupled analogue of Zhuang's theorem \cite[Theorem 4.4]{zhuang}. 

\begin{proposition}[{cf.\ \cite[Theorem 4.61]{Xu}}]\label{proposition:zhuang}
Assume that $\delta\left(X,\Delta; \left\{L_i\right\}_{i=1}^k\right)<1$. 
Then there exists a $\T$-invariant prime divisor $E$ over $X$ such that
$E$ is an lc place of a $\Q$-complement of $(X,\Delta)$ 
(see \cite[Definition 1.81]{Xu}) and 
$A_{X,\Delta}(E)<\sum_{i=1}^kS_{L_i}(E)$
holds. 
\end{proposition}

\begin{proof}
Set $\delta:=\delta\left(X,\Delta; \left\{L_i\right\}_{i=1}^k\right)$. Fix a small 
$\varepsilon\in\R_{>0}$ with $\delta\cdot(1+\varepsilon)^2<1$. 
By \cite[Corollary 3.6]{BlJ} (or \cite[Theorem 3.33]{Xu}), there exists 
$m_0\in r\Z_{>0}$ such that, for any $v\in\Val_X^{<\infty}$, for any $1\leq i\leq k$ 
and for any $m\in r\Z_{\geq m_0}$, we have 
\[
S_{L_i,m}(v)\leq (1+\varepsilon) S_{L_i}(v). 
\]
(For the definition of $S_{L_i,m}=S_m$, see \cite[Definition 3.26]{Xu}.)
After replacing $m_0$ if necessary, we may further assume that 
\[
\delta_m:=\delta_m\left(X,\Delta; \left\{L_i\right\}_{i=1}^k\right)<\delta\cdot
(1+\varepsilon)
\]
for any $m\in r\Z_{\geq m_0}$. Take any $m\in r\Z_{\geq m_0}$. 
By \cite[Theorem 4.61]{Xu}, there exists a $\T$-invariant prime divisor $E$ over 
$X$ such that 
\[
\delta_m=\frac{A_{X,\Delta}(E)}{\sum_{i=1}^k S_{L_i,m}(E)}
\]
holds. For any $1\leq i\leq k$, let us take an $m$-basis type $\Q$-divisor $D_i$ 
of $L_i$ compatible with $E$ (see \cite[Definition 3.9]{Xu}). 
Then the pair 
\[
\left(X,\Delta+\delta_m\cdot\sum_{i=1}^k D_i\right)
\]
is a log canonical Fano pair such that $E$ is an lc place. 
Therefore, $E$ is an lc place of a $\Q$-complement of $(X,\Delta)$. 
Moreover, 
since 
\[
\frac{A_{X,\Delta}(E)}{\sum_{i=1}^k S_{L_i}(E)}\leq
(1+\varepsilon)\cdot\frac{A_{X,\Delta}(E)}{\sum_{i=1}^k S_{L_i,m}(E)}
=(1+\varepsilon)\delta_m<(1+\varepsilon)^2\delta<1, 
\]
we get the assertion. 
\end{proof}

\begin{corollary}\label{corollary:zhuang}
Assume that $\delta\left(X,\Delta; \left\{L_i\right\}_{i=1}^k\right)<1$. 
Then, for all $1\leq i\leq k$, there exist \emph{$\T$-equivariant} normal 
test configurations $\left(\sX_i,\sL_i\right)/\A^1$ such that 
$\Ding\left(\left\{\sX_i,\sL_i\right\}_{i=1}^k\right)<0$ holds. 
\end{corollary}

\begin{proof}
Take a $\T$-invariant prime divisor $E$ over $X$ as in Proposition 
\ref{proposition:zhuang}. Since $E$ is an lc place of a $\Q$-complement of 
$(X,\Delta)$, by \cite{BCHM}, 
there exists the extraction $q\colon Y\to X$ of $E$, i.e., 
$E\subset Y$ is a prime $\Q$-Cartier divisor such that $-E$ is ample over $X$. 
Moreover, $Y$ is of Fano type. In particular, for any $1\leq i\leq k$, the divisor 
$E$ is dreamy with respects to $L_i$ in the sense of \cite[Definition 4.17]{Xu}. 
As in \cite[Lemmas 4.16 and 4.18]{Xu}, for any $1\leq i\leq k$, 
there exists a $\T$-equivariant normal 
test configuration $\left(\sX_i,\sL_i\right)/\A^1$ of $(X, L_i)$ 
such that the fiber of $\sX_i\to\A^1$ over $0\in\A^1$ is integral and 
$\sF_{\sX_i,\sL_i}=\sF_E$ holds. Set $\sF_i:=\sF_{\sX_i,\sL_i}$ on $R^i$, and let 
$\sF$ be the sum filtration of $\{\sF_i\}_{i=1}^k$. 
Observe that 
\[
\Ding\left(\left\{\sX_i,\sL_i\right\}_{i=1}^k\right)=\DD\left(\{\sF_i\}_{i=1}^k\right)
=\mu\left(\{\sF_i\}_{i=1}^k\right)-\sum_{i=1}^k S_{L_i}\left(\sF_i\right)
=\mu\left(\sF\right)-\sum_{i=1}^k S_{L_i}\left(E\right).
\]
For any $m\in r\Z_{>0}$ and for any $x\in\R$, we have 
\begin{eqnarray*}
\ord_E\left(I_{(m;x m)}(\sF)\right)
&=&\ord_E\left(\sum_{t_1+\cdots+t_k=x m}I_{(m; t_1)}(\sF_1)\cdots
I_{(m; t_k)}(\sF_k)\right)\\
&=&\min_{t_1+\cdots+t_k=x m}\sum_{i=1}^k\ord_E 
\left(I_{(m; t_i)}(\sF_i)\right)
\geq\min_{t_1+\cdots+t_k=x m}\sum_{i=1}^k t_i=x m. 
\end{eqnarray*}
This implies that $\mu(\sF)\leq A_{X,\Delta}(E)$. 
Since $A_{X,\Delta}(E)<\sum_{i=1}^k S_{L_i}(E)$, we get the inequality 
$\Ding\left(\left\{\sX_i,\sL_i\right\}_{i=1}^k\right)<0$. 
\end{proof}

\begin{proposition}[{cf.\ \cite[Theorem 4.12]{Xu}}]\label{proposition:DDD}
Set 
\begin{eqnarray*}
\delta&:=&\delta_{\T}^{\Red}\left(X,\Delta; \left\{L_i\right\}_{i=1}^k\right), \\
\delta_0&:=&\sup\left\{\delta'\in \left[0,\infty\right)\,\,\bigg|\,\,
\begin{aligned}\text{For any  $\T$-equivariant filtration }\sF_i\,\,(1\leq i\leq k), \\
\text{there exists }\xi\in N_{\R}(\T); \,\,
\DD\left(\left\{\sF_{i,\xi}\right\}_{i=1}^k; \delta'\right)\geq 0
\end{aligned}\right\}
\end{eqnarray*}
Then $\delta\geq \delta_0$ holds. If $\T$ is trivial, then $\delta=\delta_0$ holds. 
\end{proposition}

\begin{proof}
Let us prove $\delta\geq \delta_0$. We may assume that $\delta_0>0$. 
Take any $v\in\Val_X^{*,\T}$ and $\delta'\in \left(0,\delta_0\right)$. 
Note that, for any $\xi\in N_{\R}(\T)$, the valuation $v_\xi$ is not trivial. 
In particular, we have $A_{X,\Delta}(v_\xi), S_{L_i}(v_\xi)\in\R_{>0}$. 
For any $1\leq i\leq k$, set $\sF_i:=\sF_v$ on $R^i$. 
Then there exists $\xi\in N_{\R}(\T)$ such that 
$\DD\left(\left\{\sF_{i,\xi}\right\}_{i=1}^k; \delta'\right)\geq 0$ holds. 
Set $\sG_i:=\sF_{v_\xi}$ on $R^i$. Then, by Corollary \ref{corollary:twist-valuation}, 
we have $\sG_i=\sF_{i,\xi,\left[\theta_\xi^{L_i}(v)\right]}$. Observe that 
\[
\DD\left(\left\{\sF_{i,\xi}\right\}_{i=1}^k;\delta'\right)
=\DD\left(\left\{\sG_{i,\left[-\theta_\xi^{L_i}(v)\right]}\right\}_{i=1}^k;\delta'\right)
=\DD\left(\left\{\sG_i\right\}_{i=1}^k;\delta'\right).
\]
As in the proof of Corollary \ref{corollary:zhuang}, we have 
\[
\mu\left(\left\{\sG_i\right\}_{i=1}^k;\delta'\right)\leq\frac{A_{X,\Delta}(v_\xi)}{\delta'}.
\]
Therefore, we have 
\[
0\leq\DD\left(\left\{\sF_{i,\xi}\right\}_{i=1}^k; \delta'\right)\leq
\frac{A_{X,\Delta}(v_\xi)}{\delta'}-\sum_{i=1}^k S_{L_i}(v_\xi). 
\]
This immediately implies the inequality $\delta\geq \delta'$. Thus we get 
$\delta\geq \delta_0$. 

From now on, let us assume that $\T$ is trivial, and let us show the inequality 
$\delta\leq \delta_0$. 
Take any filtration $\sF_i$ on $R^i$ and let $\sF$ be the sum filtration of 
$\{\sF_i\}_{i=1}^k$. 
We set $I_{(m;x)}^{\langle i\rangle}:=I_{(m;x)}(\sF_i)$ and 
$I_{(m;x)}:=I_{(m;x)}(\sF)$. 
Set $\mu:=\mu(\sF; \delta)$. We want to show the inequality 
$\mu-\sum_{i=1}^k S_{L_i}(\sF_i)\geq 0$. 
Since 
\[
\lambda_{\max}(\sF)=\sum_{i=1}^k\lambda_{\max}(\sF_i), \quad
S_{L_i}(\sF_i)\leq \lambda_{\max}(\sF_i), 
\]
we may assume that $\mu<\lambda_{\max}(\sF)$. 
By Lemma \ref{lemma:lct-compute}, we have 
\[
\lct\left(X,\Delta; I_\bullet^{(\mu)}\right)=\delta. 
\]
Moreover, there exists $v\in\Val_X^{<\infty}$ such that 
\[
\delta=\frac{A_{X,\Delta}(v)}{v\left(I_\bullet^{(\mu)}\right)}, \quad
\sF\subset\sF_{v,\left[\mu-\frac{A_{X,\Delta}(v)}{\delta}\right]}.
\]

\begin{claim}\label{claim:t}
For any $t_1,\dots,t_k\in\R$, we have 
\[
\sum_{i=1}^k t_i+\frac{A_{X,\Delta}(v)}{\delta}-\mu\leq\sum_{i=1}^k
v\left(I_\bullet^{\langle i\rangle, (t_i)}\right).
\]
\end{claim}

\begin{proof}[Proof of Claim \ref{claim:t}]
If there exists $1\leq i\leq k$ such that $t_i>\lambda_{\max}(\sF_i)$, then 
the right hand side of the inequality is $\infty$. Thus we may assume that 
$t_i\leq\lambda_{\max}(\sF_i)$ for all $1\leq i\leq k$. Moreover, the function 
$t_i\mapsto v\left(I_\bullet^{\langle i\rangle, (t_i)}\right)$ is non-decreasing 
and convex over $t_i\in\left(-\infty, \lambda_{\max}(\sF_i)\right]$. 
Thus we may assume that 
$t_i<\lambda_{\max}(\sF_i)$ for all $1\leq i\leq k$. Set $t:=\sum_{i=1}^k t_i$. 
Fix $\lambda_i\in\left(t_i,\lambda_{\max}(\sF_i)\right)$. 
By Lemma \ref{lemma:uniform-convergence}, for any $\varepsilon\in\R_{>0}$, 
there exists $m_0\in r\Z_{>0}$ such that, for any $1\leq i\leq k$, 
$x_i\in\left(-\infty,\lambda_i\right]$ and $m\in r\Z_{\geq m_0}$, we have 
\[
0\leq\frac{v\left(I_{(m; x_i m)}^{\langle i\rangle}\right)}{m}
-v\left(I_\bullet^{\langle i\rangle, (x_i)}\right)<\frac{\varepsilon}{k}. 
\]
On the other hand, since 
$\sF\subset\sF_{v,\left[\mu-\frac{A_{X,\Delta}(v)}{\delta}\right]}$, we have 
\[
v\left(I_{(m; t m)}\right)\geq t m-\left(\mu-\frac{A_{X,\Delta}(v)}{\delta}\right)m
\]
for any $m\in r\Z_{\geq m_0}$. Note that 
\[
I_{(m; t m)}\supset I_{(m; t_1 m)}^{\langle 1\rangle}\cdots
I_{(m; t_k m)}^{\langle k\rangle}. 
\]
This implies that 
\[
t-\mu+\frac{A_{X,\Delta}(v)}{\delta}\leq\frac{v\left(I_{(m; t m)}\right)}{m}\leq
\sum_{i=1}^k\frac{v\left(I_{(m; t_i m)}^{\langle i\rangle}\right)}{m}
<\sum_{i=1}^k v\left(I_\bullet^{\langle i\rangle, (t_i)}\right)+\varepsilon
\]
for any $\varepsilon\in\R_{>0}$. Thus we get the assertion in 
Claim \ref{claim:t}. 
\end{proof}

\begin{claim}\label{claim:ineq}
Take any $1\leq j\leq k$ and $t_1,\dots,t_{j-1}\in\R$. 
\begin{enumerate}
\renewcommand{\theenumi}{\arabic{enumi}}
\renewcommand{\labelenumi}{(\theenumi)}
\item\label{claim:ineq1}
We have 
\[
\sum_{h=j}^k\left(S_{L_h}(v)-S_{L_h}(\sF_h)\right)
\geq \sum_{i=1}^{j-1}\left(t_i-v\left(I_\bullet^{\langle i\rangle, (t_i)}\right)\right)
+\frac{A_{X,\Delta}(v)}{\delta}-\mu.
\]
\item\label{claim:ineq2}
We have 
\[
\sum_{h=j}^k\left(T_{L_h}(v)-\lambda_{\max}(\sF_h)\right)
\geq \sum_{i=1}^{j-1}\left(t_i-v\left(I_\bullet^{\langle i\rangle, (t_i)}\right)\right)
+\frac{A_{X,\Delta}(v)}{\delta}-\mu.
\]
\end{enumerate}
\end{claim}

\begin{proof}[Proof of Claim \ref{claim:ineq}]
We only see \eqref{claim:ineq1}, since
the strategy of the proof of \eqref{claim:ineq2} is completely 
the same as the proof of \eqref{claim:ineq1}. 
We show \eqref{claim:ineq1} by induction on $-j$. Assume that $j=k$. 
By Claim \ref{claim:t}, we have 
\[
v\left(I_\bullet^{\langle k\rangle, (t_k)}\right)\geq t_k+\sum_{i=1}^{k-1}
\left(t_i-v\left(I_\bullet^{\langle i\rangle, (t_i)}\right)\right)
+\frac{A_{X,\Delta}(v)}{\delta}-\mu
\]
for any $t_k\in\R$. This implies that 
\[
\sF_k\subset\sF_{v,\left[-\sum_{i=1}^{k-1}
\left(t_i-v\left(I_\bullet^{\langle i\rangle, (t_i)}\right)\right)
-\frac{A_{X,\Delta}(v)}{\delta}+\mu\right]}.
\]
Thus we get 
\[
S_{L_k}(\sF_k)\leq S_{L_k}(v)-\sum_{i=1}^{k-1}
\left(t_i-v\left(I_\bullet^{\langle i\rangle, (t_i)}\right)\right)
-\frac{A_{X,\Delta}(v)}{\delta}+\mu,
\]
which is nothing but the assertion \eqref{claim:ineq1} for $j=k$. 

Assume that the assertion \eqref{claim:ineq1} holds for $j+1$, i.e., 
\[
v\left(I_\bullet^{\langle j\rangle, (t_j)}\right)\geq t_j
-\sum_{h=j+1}^k\left(S_{L_h}(v)-S_{L_h}(\sF_h)\right)
+\sum_{i=1}^{j-1}\left(t_i-v\left(I_\bullet^{\langle i\rangle, (t_i)}\right)\right)
+\frac{A_{X,\Delta}(v)}{\delta}-\mu
\]
holds for any $t_j\in\R$. This implies that 
\[
\sF_j\subset\sF_{v,\left[\sum_{h=j+1}^k\left(S_{L_h}(v)-S_{L_h}(\sF_h)\right)
-\sum_{i=1}^{j-1}\left(t_i-v\left(I_\bullet^{\langle i\rangle, (t_i)}\right)\right)
-\frac{A_{X,\Delta}(v)}{\delta}+\mu
\right]}. 
\]
Therefore, we get 
\[
S_{L_j}(\sF_j)\leq S_{L_j}(v)+\sum_{h=j+1}^k\left(S_{L_h}(v)-S_{L_h}(\sF_h)\right)
-\sum_{i=1}^{j-1}\left(t_i-v\left(I_\bullet^{\langle i\rangle, (t_i)}\right)\right)
-\frac{A_{X,\Delta}(v)}{\delta}+\mu, 
\]
which is nothing but the assertion \eqref{claim:ineq1}. Thus we have proved 
Claim \ref{claim:ineq}. 
\end{proof}

By applying Claim \ref{claim:ineq} \eqref{claim:ineq1} with $j=1$, we get 
\[
\sum_{i=1}^k S_{L_i}(\sF_i)\leq\sum_{i=1}^k S_{L_i}(v)
-\frac{A_{X,\Delta}(v)}{\delta}+\mu. 
\]
As a consequence, we have
\[ 
\DD\left(\left\{\sF_i\right\}_{i=1}^k;\delta\right)=\mu-\sum_{i=1}^k S_{L_i}(\sF_i)
\geq\frac{A_{X,\Delta}(v)}{\delta}-\sum_{i=1}^k S_{L_i}(v)\geq 0, 
\]
where the last inequality follows from the definition of $\delta$. 
Thus we get the inequality $\delta\leq \delta_0$ when $\T$ is trivial. 
\end{proof}

\begin{thm}\label{theorem:A}
The following are equivalent: 
\begin{enumerate}
\renewcommand{\theenumi}{\arabic{enumi}}
\renewcommand{\labelenumi}{(\theenumi)}
\item\label{theorem:A1}
There exists $\varepsilon\in\R_{>0}$ such that, 
for any $\T$-equivariant filtration $\sF_i$ on $R^i$ $(1\leq i\leq k)$, we have 
$\DD\left(\{\sF_i\}_{i=1}^k\right)\geq \varepsilon\cdot\JJ_{\T}^{\cp}
\left(\{\sF_i\}_{i=1}^k\right)$. 
\item\label{theorem:A2}
$\left(X,\Delta;\{L_i\}_{i=1}^k\right)$ has vanishing coupled 
$\T$-Futaki characters and 
$\delta_{\T}^{\Red}\left(X,\Delta; \left\{L_i\right\}_{i=1}^k\right)>1$ holds. 
\end{enumerate}
\end{thm}

\begin{proof}
\eqref{theorem:A1}$\implies$\eqref{theorem:A2}: 
By Corollaries \ref{corollary:twist-ding} and \ref{corollary:DingD}, we know that 
$\alpha_{\bc}^{\cp}=0$. The inequality 
$\delta_{\T}^{\Red}\left(X,\Delta; \left\{L_i\right\}_{i=1}^k\right)>1$ 
follows from Theorem \ref{theorem:uD} and Proposition \ref{proposition:DDD}. 

\eqref{theorem:A2}$\implies$\eqref{theorem:A1}: 
Fix any $\delta\in\left(1, 
\delta_{\T}^{\Red}\left(X,\Delta; \left\{L_i\right\}_{i=1}^k\right)\right)$, and 
take a small $\varepsilon\in (0,1)$ with $1+n\varepsilon\leq \delta$. 
We show that the $\varepsilon$ satisfies the assertion \eqref{theorem:A1}. 
For any $1\leq i\leq k$, let $\sF_i$ be a $\T$-equivariant filtration on $R^i$, and 
let $\sF$ be the sum filtration of $\{\sF_i\}_{i=1}^k$. 
We set 
$I_{(m;x)}:=I_{(m;x)}(\sF)$. Set $\mu:=\mu(\sF)$. 
If $\mu=\lambda_{\max}(\sF)$, then 
\[
\DD\left(\left\{\sF_i\right\}_{i=1}^k\right)
-\JJ^{\cp}\left(\left\{\sF_i\right\}_{i=1}^k\right)=0.
\]
Thus we may assume that $\mu<\lambda_{\max}(\sF)$. 
By Lemma \ref{lemma:lct-compute}, we have 
\[
\lct\left(X,\Delta; I_\bullet^{(\mu)}\right)=1, 
\]
Moreover, there exists $v\in\Val_X^{<\infty,\T}$ such that 
$A_{X,\Delta}(v)=v\left(I_\bullet^{(\mu)}\right)$ and 
$\sF\subset\sF_{v,\left[\mu-A_{X,\Delta}(v)\right]}$. 

We observe that, there exists $\xi\in N_{\R}(\T)$ such that 
$A_{X,\Delta}(v_\xi)\geq\delta\cdot\sum_{i=1}^k S_{L_i}(v_\xi)$. Indeed, if 
$v\in\Val_X^{*,\T}$, then the above inequality follows from the definition of $\delta$. 
Otherwise, we can write $v=\operatorname{wt}_{\xi'}$ for some 
$\xi'\in N_{\R}(\T)$, and then the above inequality holds if we take $\xi:=-\xi'$. 
Note that, for such $\xi$, we have 
\[
\sF_\xi\subset\left(\sF_{v,\left[\mu-A_{X,\Delta}(v)\right]}\right)_\xi
=\sF_{v_\xi,\left[\mu-A_{X,\Delta}(v_\xi)\right]}
\]
by Corollary \ref{corollary:twist-valuation}. 
Recall that $\sF_\xi$ is the sum filtration of $\{\sF_{i,\xi}\}_{i=1}^k$. 
Moreover, by \cite[Lemma 6.24]{Xu}, we have $\mu=\mu(\sF_\xi)$. 
Set $I_{(m;x)}^{\langle\xi, i\rangle}:=I_{(m; x)}(\sF_{i,\xi})$. 
By completely the same argument in the proof of Claims \ref{claim:t} and 
\ref{claim:ineq}, we get the following: 
\begin{itemize}
\item
For any $t_1,\dots,t_k\in\R$, we have 
\[
\sum_{i=1}^k t_i+A_{X,\Delta}(v_\xi)-\mu\leq\sum_{i=1}^k 
v_\xi\left(I_\bullet^{\langle\xi,i\rangle,(t_i)}\right). 
\]
\item
For any $1\leq j\leq k$ and $t_1,\dots,t_{j-1}\in\R$, we have 
\begin{eqnarray*}
\sum_{h=j}^k\left(S_{L_h}(v_\xi)-S_{L_h}(\sF_{h,\xi})\right)
&\geq&\sum_{i=1}^{j-1}\left(t_i-
v_\xi\left(I_\bullet^{\langle\xi,i\rangle, (t_i)}\right)\right)
+A_{X,\Delta}(v_\xi)-\mu,\\
\sum_{h=j}^k\left(T_{L_h}(v_\xi)-\lambda_{\max}(\sF_{h,\xi})\right)
&\geq&\sum_{i=1}^{j-1}\left(t_i-
v_\xi\left(I_\bullet^{\langle\xi,i\rangle, (t_i)}\right)\right)
+A_{X,\Delta}(v_\xi)-\mu.
\end{eqnarray*}
\item
In particular, we have 
\begin{eqnarray*}
\sum_{i=1}^k S_{L_i}(\sF_{i,\xi})&\leq&\sum_{i=1}^k S_{L_i}(v_\xi)
-A_{X,\Delta}(v_\xi)+\mu, \\ 
\sum_{i=1}^k \lambda_{\max}(\sF_{i,\xi})&\leq&\sum_{i=1}^k T_{L_i}(v_\xi)
-A_{X,\Delta}(v_\xi)+\mu.
\end{eqnarray*}
\end{itemize}

Since $\alpha_{\bc}^{\cp}=0$, we have 
$\DD\left(\{\sF_{i,\xi}\}_{i=1}^k\right)=\DD\left(\{\sF_i\}_{i=1}^k\right)$ 
by Lemma \ref{lemma:D-twist}. Moreover, by \cite[Lemma 2.6]{BlJ} (or 
\cite[Lemma 3.31]{Xu}), we have 
$T_{L_i}(v_\xi)\leq (n+1)S_{L_i}(v_\xi)$. 
Therefore, we get 
\begin{eqnarray*}
&&\DD\left(\{\sF_i\}_{i=1}^k\right)
-\varepsilon\cdot\JJ^{\cp}\left(\{\sF_{i,\xi}\}_{i=1}^k\right)\\
&=&\DD\left(\{\sF_{i,\xi}\}_{i=1}^k\right)
-\varepsilon\cdot\JJ^{\cp}\left(\{\sF_{i,\xi}\}_{i=1}^k\right)\\
&=&\mu-(1-\varepsilon)\sum_{i=1}^k S_{L_i}\left(\sF_{i,\xi}\right)
-\varepsilon\cdot\sum_{i=1}^k\lambda_{\max}\left(\sF_{i,\xi}\right) \\
&\geq&\mu-(1-\varepsilon)\left(\sum_{i=1}^k S_{L_i}\left(v_\xi\right)
-A_{X,\Delta}(v_\xi)+\mu\right)
-\varepsilon\left(\sum_{i=1}^k T_{L_i}\left(v_\xi\right)
-A_{X,\Delta}(v_\xi)+\mu\right)\\
&=&A_{X,\Delta}(v_\xi)-(1-\varepsilon)\sum_{i=1}^k S_{L_i}\left(v_\xi\right)
-\varepsilon\cdot\sum_{i=1}^k T_{L_i}\left(v_\xi\right)\\
&\geq&A_{X,\Delta}(v_\xi)-(1+n\varepsilon)\sum_{i=1}^k S_{L_i}\left(v_\xi\right)
\geq A_{X,\Delta}(v_\xi)-\delta\cdot\sum_{i=1}^k S_{L_i}\left(v_\xi\right)\geq 0.
\end{eqnarray*}
Thus we get the assertion \eqref{theorem:A1}. 
\end{proof}

As a consequence, we get the following desired result: 

\begin{corollary}\label{corollary:conclusion}
\begin{enumerate}
\renewcommand{\theenumi}{\arabic{enumi}}
\renewcommand{\labelenumi}{(\theenumi)}
\item\label{corollary:conclusion1}
The following are equivalent: 
\begin{enumerate}
\renewcommand{\theenumii}{\roman{enumii}}
\renewcommand{\labelenumii}{(\theenumii)}
\item\label{corollary:conclusion11}
$\left(X,\Delta;\{L_i\}_{i=1}^k\right)$ is coupled Ding semistable.
\item\label{corollary:conclusion12}
$\left(X,\Delta;\{L_i\}_{i=1}^k\right)$ is $\T$-equivariantly coupled Ding semistable.
\item\label{corollary:conclusion13}
For any $\T$-equivariant filtration $\sF_i$ on $R^i$ $(1\leq i\leq k)$, we have 
$\DD\left(\{\sF_i\}_{i=1}^k\right)\geq 0$. 
\item\label{corollary:conclusion14}
$\delta\left(X,\Delta;\{L_i\}_{i=1}^k\right)\geq 1$ holds. 
\end{enumerate}
\item\label{corollary:conclusion2}
If $\left(X,\Delta;\{L_i\}_{i=1}^k\right)$ is $\T$-reduced uniformly coupled Ding 
stable, then $\T\subset\Aut(X,\Delta)$ must be a maximal torus. 
\item\label{corollary:conclusion3}
Let $\T\subset\Aut(X,\Delta)$ be a maximal torus. 
The following are equivalent: 
\begin{enumerate}
\renewcommand{\theenumii}{\roman{enumii}}
\renewcommand{\labelenumii}{(\theenumii)}
\item\label{corollary:conclusion31}
$\left(X,\Delta;\{L_i\}_{i=1}^k\right)$ is reduced uniformly coupled Ding stable.
\item\label{corollary:conclusion32}
There exists $\varepsilon\in\R_{>0}$ such that, 
for any $\T$-equivariant filtration $\sF_i$ on $R^i$ $(1\leq i\leq k)$, we have 
$\DD\left(\{\sF_i\}_{i=1}^k\right)\geq \varepsilon\cdot\JJ_{\T}^{\cp}
\left(\{\sF_i\}_{i=1}^k\right)$. 
\item\label{corollary:conclusion33}
There exists $\delta\in\R_{>1}$ such that, 
for any $\T$-equivariant filtration $\sF_i$ on $R^i$ $(1\leq i\leq k)$, there exists 
$\xi\in N_{\R}(\T)$ such that 
$\DD\left(\left\{\sF_{i,\xi}\right\}_{i=1}^k;\delta\right)\geq 0$ holds. 
\item\label{corollary:conclusion34}
$\left(X,\Delta;\{L_i\}_{i=1}^k\right)$ has vanishing coupled Futaki characters and 
$\delta^{\Red}\left(X,\Delta;\left\{L_i\right\}_{i=1}^k\right)> 1$ holds. 
\end{enumerate}
\end{enumerate}
\end{corollary}

\begin{proof}
\eqref{corollary:conclusion1}
\eqref{corollary:conclusion11}$\implies$\eqref{corollary:conclusion12} 
is trivial. 
\eqref{corollary:conclusion12}$\iff$\eqref{corollary:conclusion13} 
follows from Corollary \ref{corollary:DingD}. 
\eqref{corollary:conclusion12}$\implies$\eqref{corollary:conclusion14}
follows from Corollary \ref{corollary:zhuang}. 
\eqref{corollary:conclusion14}$\implies$\eqref{corollary:conclusion11}
follows from Proposition \ref{proposition:DDD}. 

\eqref{corollary:conclusion2}
Assume that $\T\subsetneq\tilde{\T}\subset\Aut(X,\Delta)$. 
Since $\left(X,\Delta;\{L_i\}_{i=1}^k\right)$ is $\T$-equivariantly coupled Ding 
semistable, by \eqref{corollary:conclusion1}, it is 
$\tilde{\T}$-equivariantly coupled Ding semistable. In particular, we have 
\[
N_{\R}\left(\tilde{\T}\right)\supset\tilde{\PP}^L
\ni\sum_{i=1}^k{\tilde{\alpha}}_{\bc}^{\cp}=0. 
\]
Take any $\tilde{\xi}\in N\left(\tilde{\T}\right)\setminus N(\T)$. Consider 
the product test configurations 
$\left(X_{\tilde{\xi}}, (L_i)_{\tilde{\xi}}\right)/\A^1$ of $(X, L_i)$. Then, by 
Corollary \ref{corollary:twist-ding}, 
we have 
\[
\Ding\left(\left\{X_{\tilde{\xi}}, (L_i)_{\tilde{\xi}}\right\}_{i=1}^k\right)=0. 
\]
On the other hand, by Lemma \ref{lemma:sum-basics} \eqref{lemma:sum-basics4}, 
we have 
\[
\JJ_{\T}^{\cp}\left(\left\{X_{\tilde{\xi}}, (L_i)_{\tilde{\xi}}\right\}_{i=1}^k\right)>0. 
\]
This leads to a contradiction. Therefore we get the assertion 
\eqref{corollary:conclusion2}. 

\eqref{corollary:conclusion3}
\eqref{corollary:conclusion31}$\iff$\eqref{corollary:conclusion32} 
follows from Corollary \ref{corollary:DingD}. 
\eqref{corollary:conclusion32}$\iff$\eqref{corollary:conclusion33} 
follows from Theorem \ref{theorem:uD}. 
\eqref{corollary:conclusion32}$\iff$\eqref{corollary:conclusion34} 
follows from Theorem \ref{theorem:A}. 
\end{proof}

\begin{remark}\label{remark:cDss}
One may expect that the conditions in Corollary \ref{corollary:conclusion} 
\eqref{corollary:conclusion1} might be equivalent to the following condition: 
\begin{center}
$\left(X,\Delta;\{L_i\}_{i=1}^k\right)$ has vanishing coupled 
$\T$-Futaki characters and 
$\delta_{\T}^{\Red}\left(X,\Delta; \left\{L_i\right\}_{i=1}^k\right)\geq 1$ holds. 
\end{center}
However, this condition is not equivalent to the conditions in 
Corollary \ref{corollary:conclusion} \eqref{corollary:conclusion1} in general. 
In fact, assume that 
$\left(X,\Delta;\{L_i\}_{i=1}^k\right)$ is not coupled Ding semistable, 
a maximal subtorus $\T$ of $\Aut(X,\Delta)$ is nontrivial and 
$\alpha_{\bc}^{\cp}=0$. Take any 
$\xi\in N_{\R}(\T)\setminus\{0\}$ and $v\in\Val_X^{*, \T}$. 
By Lemma \ref{lemma:A-S-twist}, we have 
\[
A_{X,\Delta}(v_{e\xi})-\sum_{i=1}^k S_{L_i}(v_{e\xi})
=A_{X,\Delta}(v)-\sum_{i=1}^k S_{L_i}(v)
\]
for any $e\in\R_{>0}$. On the other hand, 
$A_{X,\Delta}(v_{e\xi})=A_{X,\Delta}(v)+\theta_{e\xi}^L(v)$ as in Proposition 
\ref{proposition:theta}. By the definition of $\theta_{e\xi}^L(v)$, we have 
\[
\lim_{e\to\infty}\theta_{e\xi}^L(v)=\infty. 
\]
This implies that 
\[
\lim_{e\to\infty}\frac{A_{X,\Delta}(v_{e\xi})}{\sum_{i=1}^k S_{L_i}(v_{e\xi})}=1. 
\]
Together with Corollary \ref{corollary:conclusion} \eqref{corollary:conclusion3}, 
we must have 
$\delta_{\T}^{\Red}\left(X,\Delta; \left\{L_i\right\}_{i=1}^k\right)=1$. 
\end{remark}

\section{On the conjecture of Hultgren and Witt Nystr\"om}\label{section:analytic}

In this section, we work over the complex number field $\C$. 
Let $X$ be an $n$-dimensional Fano manifold, with $\Delta = 0$. We fix a maximal (algebraic) torus $\T \subset \mathrm{Aut}_0(X)$, and its maximal compact subgroup $\T_r \subset \T$ which is a real torus. We then fix $\T$-linearized ample $\Q$-line bundles $L_1,\dots,L_k$ on $X$ with $-K_X=\sum_{i=1}^k L_i$.

We first review the bare minimum of analytic details concerning the coupled K\"ahler--Einstein metrics and the coupled Ding functional. All the details can be found in the original paper \cite{HWN} where these objects were introduced. We pick reference $\T_r$-invariant hermitian metrics $h_1 , \dots , h_k$ on $L_1 , \dots , L_k$ respectively, and associated K\"ahler metrics $\theta_1 , \dots , \theta_k$ in $c_1(L_1) , \dots , c_1(L_k)$. We also define the volume form $d \mu'_0$ on $X$ by $h_1 \otimes \cdots \otimes h_k$, recalling that a hermitian metric on $-K_X$ naturally defines a volume form; see \cite[\S 2]{HWN} or \cite[\S 3.4]{hashimoto} for more details.

For each $i=1 , \dots , k$, we define
\begin{equation*}
	\mathcal{H}_i := \{ \phi_i \in C^{\infty} (X , \mathbb{R}) \mid \theta_i + \sqrt{-1} \partial \bar{\partial} \phi_i / 2 \pi >0 \}
\end{equation*}
to be the space of K\"ahler metrics in $c_1 (L_i)$. We then define the space of coupled K\"ahler metrics as
\begin{equation*}
\bm{\mathcal{H}} := \mathcal{H}_1 \times \cdots \times \mathcal{H}_k.
\end{equation*}
Since $L_1 , \dots , L_k$ are $\T$-linearized, $\T$ naturally acts on $\mathcal{H}_1 , \dots , \mathcal{H}_k$ by pullback. With this action understood, the space of $\T_r$-invariant K\"ahler potentials is written as $\mathcal{H}^{\T_r}_1 , \dots , \mathcal{H}^{\T_r}_k$, with
\begin{equation*}
\bm{\mathcal{H}}^{\T_r} := \mathcal{H}^{\T_r}_1 \times \cdots \times \mathcal{H}^{\T_r}_k.
\end{equation*}

\begin{definition} \label{apdfcpdfc}
The \emph{coupled Ding functional} is a map $\mathrm{D}^{\mathrm{cp}} : \bm{\mathcal{H}} \to \mathbb{R}$ defined by
\begin{equation*}
\mathrm{D}^{\mathrm{cp}} (\phi_1 , \dots , \phi_k) := \mathrm{L}^{\mathrm{cp}} (\phi_1 , \dots , \phi_k) - \sum_{i=1}^k \mathrm{E}_i (\phi_i),
\end{equation*}
where
\begin{equation*}
\mathrm{L}^{\mathrm{cp}} (\phi_1 , \dots , \phi_k ) := -\log \int_X \exp \left( - \sum_{i=1}^k \phi_i \right) d \mu'_0
\end{equation*}
and
\begin{equation*}
	\mathrm{E}_i (\phi_i) := \frac{1}{(n+1) \int_X c_1 (L_i)^n} \sum_{j=0}^n \int_X \phi_i \theta_i^{n-j} \wedge (\theta_i + \sqrt{-1} \partial \bar{\partial} \phi_i / 2 \pi)^j .
\end{equation*}

The coupled Ding functional is said to be \emph{$\T$-coercive} if there exists $\varepsilon >0$ such that
\begin{equation} \label{apexcoercpdg}
	\mathrm{D}^{\mathrm{cp}} (\phi_1 , \dots , \phi_k) \ge \varepsilon \mathrm{J}_{\mathrm{cp}, \T} (\phi_1 , \dots , \phi_k ) - \frac{1}{\varepsilon}
\end{equation}
holds for all $(\phi_1 , \dots , \phi_k) \in \bm{\mathcal{H}}^{\T_r}$; in the above, $\mathrm{J}_{\mathrm{cp}, \T}$ is a functional defined on $\bm{\mathcal{H}}^{\T_r}_i$ as
\begin{equation*}
	\mathrm{J}_{\mathrm{cp}, \T} (\phi_1 , \dots , \phi_k ) := \inf_{\sigma \in \T} \sum_{i=1}^k \mathrm{J}_i (\sigma^* \phi_i),
\end{equation*}
where
\begin{equation*}
	\mathrm{J}_i (\phi_i) := \frac{1}{\int_X c_1 (L_i)^n} \int_X \phi_i \theta^n_i - \mathrm{E}_i (\phi_i).
\end{equation*}
\end{definition}

Using the functional $\mathrm{L} : \mathcal{H}(-K_X) \to \mathbb{R}$ that appears in the usual Ding functional (see e.g.~\cite[Definition 2.7]{BBJ}), we may also write
\begin{equation} \label{eqlcplsm}
	\mathrm{L}^{\mathrm{cp}} (\phi_1 , \dots , \phi_k ) = \mathrm{L} \left( \sum_{i=1}^k \phi_i \right).
\end{equation}

We write $\mathcal{E}^1_1 , \dots , \mathcal{E}^1_k$ respectively for the completion of $\mathcal{H}_1 , \dots , \mathcal{H}_k$ with respect to the $d_1$-topology, as explained in \cite{darlectures,gzbook}. As above, $\T_r$-invariant spaces are written as $(\mathcal{E}^1_1)^{\T_r} , \dots , (\mathcal{E}^1_k)^{\T_r}$, and we denote
\begin{equation*}
	\bm{\mathcal{E}}^1 := \mathcal{E}^1_1 \times \cdots \times \mathcal{E}^1_k , \quad (\bm{\mathcal{E}}^1)^{\T_r} := (\mathcal{E}^1_1)^{\T_r} \times \cdots \times (\mathcal{E}^1_k)^{\T_r}.
\end{equation*}
We give a product metric on the spaces above, induced from the $d_1$ metrics on each factor. In case it is necessary to make the polarization explicit, we may write $d_{1,i}$ for the $d_1$ metric on $\mathcal{E}^1_i$.

While these spaces of metrics are defined for $L_1 , \dots , L_k$, we also need to consider the space of K\"ahler metrics in $-K_X$, which we denote as $\mathcal{H} (-K_X)$ and $\mathcal{E}^1 (-K_X)$, where the reference metric is $\sum_{i=1}^k \theta_i$. The decomposition $-K_X = L_1 + \cdots + L_k$ gives rise to the following natural map.

\begin{proposition} \label{ppsmctna}
	The map $\mathsf{S} :  \mathcal{H}_1 \times \cdots \times \mathcal{H}_k \to \mathcal{H} (-K_X)$ defined by 
	\begin{equation*}
		\mathsf{S}(\phi_1 , \dots , \phi_k ) := \sum_{i=1}^k \phi_i
	\end{equation*}
	extends to a continuous map $\mathsf{S} :  \mathcal{E}^1_1 \times \cdots \times \mathcal{E}^1_k \to \mathcal{E}^1 (-K_X)$ with respect to the $d_1$ metric.
\end{proposition}

In the rest of the paper, we do not need the continuity of $\mathsf{S}$, but the statement is included here for completeness. This result can be regarded as an Archimedean version of Proposition \ref{ppsmctnna}.

\begin{proof}
	The map $\mathsf{S} :  \mathcal{H}_1 \times \cdots \times \mathcal{H}_k \to \mathcal{H} (-K_X)$ is well-defined since the reference K\"ahler metric on $\mathcal{H}^1 (-K_X)$ is given by $\sum_{j=1}^k \theta_j$. We prove that $\mathsf{S} :  \mathcal{E}^1_1 \times \cdots \times \mathcal{E}^1_k \to \mathcal{E}^1 (-K_X)$ is well-defined. We start by proving that if $\phi_i \in \mathcal{E}^1_i$, then $\phi_i \in \mathcal{E}^1 (-K_X)$. We set $\omega := \sum_{i=1}^k \theta_i$ for the reference metric, and assume first that $\sup_X \phi_i =0$ for all $1 \le i \le k$.

	It is obvious that $\phi_i$ is $\omega$-psh. By the definition of $\mathcal{E}^1_i$ (see e.g.~\cite[Definition 10.15]{gzbook}), we have
	\begin{equation*}
		\sum_{j=1}^n \int_X \phi_i  (\theta_i + \sqrt{-1} \partial \bar{\partial} \phi_i)^j \wedge \theta_i^{n-j} > - \infty.
	\end{equation*}
	Since $\mathcal{E}^1_i$ is convex \cite[Corollary 2.20]{darlectures}, we also have $s \phi_i \in \mathcal{E}^1_i$ for all $0 \le s \le 1$, meaning
	\begin{equation*}
		\sum_{j=1}^n \int_X \phi_i  (\theta_i + s \sqrt{-1} \partial \bar{\partial} \phi_i)^j \wedge \theta_i^{n-j} > - \infty.
	\end{equation*}
	for any $0 \le s \le 1$. We  take $s>0$ to be small enough so that $\theta_i \ge s \omega$, and hence
	\begin{equation*}
		\sum_{j=1}^n \int_X \phi_i  \left( s \omega + s \sqrt{-1} \partial \bar{\partial} \phi_i  \right)^j \wedge (s\omega)^{n-j} > \sum_{j=1}^n \int_X \phi_i  (\theta_j + s \sqrt{-1} \partial \bar{\partial} \phi_i)^j \wedge \theta_i^{n-j} > - \infty
	\end{equation*}
	which establishes $\phi_i \in \mathcal{E}^1 (-K_X)$ for all $i=1 , \dots , k$. Again by the convexity of $\mathcal{E}^1 (-K_X)$, we thus find $\sum_{i=1}^k \phi_i /k \in \mathcal{E}^1 (-K_X)$. Note on the other hand that $\sum_{i=1}^k \phi_i$ is $\omega$-psh by definition, with $\sup_X \left( \sum_{i=1}^k \phi_i \right) \le 0$, and hence we get $\sum_{i=1}^k \phi_i \in \mathcal{E}^1 (-K_X)$ by
	\begin{align*}
		&\frac{1}{k^{n+1}} \sum_{j=1}^n \int_X \sum_{i=1}^k \phi_i \left( \sum_{i=1}^k \left( \theta_i + \sqrt{-1} \partial \bar{\partial} \phi_i \right) \right)^j \wedge \omega^{n-j} \\
		&> \sum_{j=1}^n  \int_X  \sum_{j=1}^k \frac{\phi_j}{k}  \left( \sum_{j=1}^k \left( \theta_j + \frac{1}{k} \sqrt{-1} \partial \bar{\partial} \phi_j \right) \right)^j \wedge \omega^{n-j} > - \infty .
	\end{align*}
	The general case is obvious by considering $\phi_i - \sup_X \phi_i$.
	
	We prove the continuity. Pick $(\phi_1 , \dots , \phi_k ) \in \bm{\mathcal{E}}^1$ and a sequence $\{ (\psi^{(j)}_1 , \dots , \psi^{(j)}_k ) \}_{j=1}^{\infty} \subset \bm{\mathcal{E}}^1$ such that $d_{1,i}(\phi_i, \psi^{(j)}_i) \to 0$ as $j \to \infty$ for all $i = 1, \dots , k$. By the triangle inequality, we have
	\begin{align*}
		&d_1 \left( \sum_{i=1}^k \phi_i , \sum_{i=1}^k \psi^{(j)}_i \right) \\
		&\le d_1 \left( \sum_{i=1}^k \phi_i , \psi^{(j)}_1+ \sum_{i=2}^k \phi_i \right) + d_1 \left( \psi^{(j)}_1+ \sum_{i=2}^k \phi_i , \sum_{i=1}^k \psi^{(j)}_i \right) \\
		&\le d_1 \left( \sum_{i=1}^k \phi_i , \psi^{(j)}_1+ \sum_{i=2}^k \phi_i \right) + \sum_{m=1}^{k-1} d_1 \left( \sum_{i=1}^m \psi^{(j)}_i + \sum_{i=m+1}^k \phi_i ,  \sum_{i=1}^{m+1} \psi^{(j)}_i + \sum_{i=m+2}^k \phi_i \right) ,
	\end{align*}
	where $d_1$ is the metric on $\mathcal{E}^1(-K_X)$ with the reference metric $\omega$ and we decree $\sum_{i=m+2}^k \phi_i=0$ for $m = k-1$. Recalling \cite[Theorem 3.32]{darlectures}, it thus suffices to show
	\begin{equation*}
		\int_X |\phi_i - \psi^{(j)}_i| \omega^n_{v} \to 0
	\end{equation*}
	as $j \to \infty$, for any $i=1, \dots, k$ and for any potential $v$ which appears in the inequality above. We take a constant $C >0$ such that all the metrics appearing in the above inequality are within distance $C$ from the origin with respect to $d_1$. Then, \cite[Corollary 3.50]{darlectures} shows that there exists a continuous function $f_C : \mathbb{R}^+ \to \mathbb{R}^+$ with $f_C (0)=0$ such that
	\begin{equation*}
		\int_X |\phi_j - \psi^{(j)}_i| \omega^n_{v} \le f_C (d_1 (\phi_i , \psi^{(j)}_i)).
	\end{equation*}
	Again by \cite[Theorem 3.32]{darlectures}, $d_1 (\phi_i , \psi^{(j)}_i)$ is Lipschitz equivalent to
	\begin{equation*}
		\int_X |\phi_i - \psi^{(j)}_i| \omega^n_{\phi_i} + \int_X |\phi_i - \psi^{(j)}_i| \omega^n_{\psi^{(j)}_i}.
	\end{equation*}
	We now take $s >1$ large enough so that $\omega < s \theta_i$. We then find
	\begin{equation*}
		\int_X |\phi_j - \psi^{(j)}_i| \omega^n_{\phi_i} \le s^n \int_X |\phi_i - \psi^{(j)}_i| (\theta_i + \sqrt{-1} \partial \bar{\partial} \phi_i /s)^n \to 0
	\end{equation*}
	as $j \to \infty$, by \cite[Theorem 3.46 (ii)]{darlectures} and $d_{1,i} (\phi_i , \psi^{(j)}_i) \to 0$ which follows from the hypothesis, where we note $\phi_i /s \in \mathcal{E}^1_i$ by the convexity \cite[Corollary 2.20]{darlectures}. Similarly, from $d_{1,i} (\phi_i , \psi^{(j)}_i) \to 0$ as $j \to \infty$, we also have
	\begin{equation*}
		\int_X |\phi_i - \psi^{(j)}_i| \left( \theta_i + \sqrt{-1} \partial \bar{\partial} \psi^{(j)}_i /s \right)^n \to 0
	\end{equation*}
	by a diagonal argument for $j$ with \cite[Theorem 3.46 (ii)]{darlectures}. Thus we find $d_1 (\phi_i , \psi^{(j)}_i) \to 0$ as $j \to \infty$ as claimed.
\end{proof}

\begin{proposition} \label{ppkickedcptc}
	$\left(X ;\{L_i\}_{i=1}^k\right)$ admits a $\T_r$-invariant coupled K\"ahler--Einstein metric if and only if $\mathrm{D}^{\mathrm{cp}}$ is $\T$-coercive.
\end{proposition}

\begin{remark}
Note that the existence of a $\T_r$-invariant coupled K\"ahler--Einstein metric implies that $\Aut_0(X)$ is necessarily reductive by \cite[Corollary 1.6]{HWN}.	
\end{remark}

The proof of Proposition \ref{ppkickedcptc} is an almost word-by-word repetition of \cite[Proof of Theorem 2.19]{Li} and relies on \cite[Theorem 3.4]{DR} by Darvas--Rubinstein, but we provide some details for the reader's convenience. It suffices to check (A1)--(A4) and (P1)--(P7) in \cite[\S 3]{DR}. As in \cite[Proof of Theorem 2.19]{Li}, we set
\begin{equation*}
	\bm{\mathcal{R}}:= (\mathcal{E}_1^1 \cap L^{\infty} (X))^{\T_r} \times \cdots \times  (\mathcal{E}_k^1 \cap L^{\infty} (X))^{\T_r}, \quad \overline{\bm{\mathcal{R}}}:= (\mathcal{E}_1^1 )^{\T_r} \times \cdots \times  (\mathcal{E}_k^1)^{\T_r},
\end{equation*}
where we set $\bm{\mathcal{M}}$ to be the set of (smooth) $\T_r$-invariant coupled K\"ahler--Einstein metrics. Noting that (A2) follows from \cite[Lemma 3.1]{HWN}, (A1)--(A4) are immediate. (P1) again follows from \cite[Lemma 3.1]{HWN}, noting that $\mathrm{D}^{\mathrm{cp}}$ is continuous along geodesics since it is lower semicontinuous and convex along geodesics. (P2) is straightforward by adapting \cite[Proof of Proposition 5.27]{DR} to the coupled case. (P3) follows from the smoothness of the weak solution \cite[Proposition 2.8 and Theorem 2.9]{HWN}. (P4) and (P7) are obvious. (P5) can be proved exactly as in \cite[Proof of Theorem 2.19]{Li}, by recalling the uniqueness of coupled K\"ahler--Einstein metrics up to $\mathrm{Aut}_0(X)$ \cite[Theorem 1.5]{HWN} and noting that $\T$ is reductive whose center is $\T$ itself. (P6) follows from \cite[Proposition 6.8]{DR}.


We first prove the easier direction of Theorem \ref{theorem:mainmc}, which generalizes \cite[Theorem 1.15]{HWN}:

\begin{thm}\label{theorem:apmnthm}
If $\left(X ;\{L_i\}_{i=1}^k\right)$ admits a $\T_r$-invariant coupled K\"ahler--Einstein metric, then 
it is $\T$-reduced uniformly coupled Ding stable.
\end{thm}

\begin{proof}
By Proposition \ref{ppkickedcptc}, it suffices to derive the $\T$-reduced uniform coupled Ding stability from the condition (\ref{apexcoercpdg}).

It is well-known that to each test configuration we can associate a ray of smooth K\"ahler potentials called a subgeodesic ray \cite[\S 3.1]{BHJ2} (also called a psh ray in \cite[Definition 1.3]{BBJ}), such that the non-Archimedean limit \cite[Definition 3.1]{BHJ2} of the ray agrees with the non-Archimedean metric represented by the given test configuration; see \cite[\S 3.1]{BHJ2} for more details. In a more down-to-earth terminology, for each $k$-tuple of ample test configurations $\left\{(\sX_i,\sL_i)\right\}_{i=1}^k$ we have a ray $\{ (\phi_1 (t) , \dots , \phi_k (t) ) \}_{t \ge 0} \subset \bm{\mathcal{H}}$ of smooth K\"ahler potentials, all of which are in fact Fubini--Study metrics explicitly defined in terms of the generators of the $\G_m$-action (see e.g.~\cite[Lemma 9 and Theorem 19]{hashimoto}), such that the following slope formulae hold:
\begin{equation} \label{apeqjisf}
	\lim_{t \to + \infty} \frac{\mathrm{J}_i (\phi_i (t))}{t} = \JJ\left(\sX_i,\sL_i\right)
\end{equation}
for each $1 \le i \le k$ by \cite[Theorem 3.6]{BHJ2}, and
\begin{equation} \label{apeqcpdsf}
	\lim_{t \to + \infty} \frac{\mathrm{D}^{\mathrm{cp}} (\phi_1 (t) , \dots , \phi_k (t))}{t} = \Ding\left(\{\sX_i,\sL_i\}_{i=1}^k\right)
\end{equation}
by \cite[Theorem 19]{hashimoto}. There are two important technical points in the formula (\ref{apeqcpdsf}). Firstly, the coupled Ding invariant in (\ref{apeqcpdsf}) is defined slightly differently to Definition \ref{definition:ding}, since the test configuration used in (\ref{apeqcpdsf}) (see \cite[Definition 19]{hashimoto}) is $(\mathcal{Y} , \mathcal{L}_{\mathcal{Y}})$ generated by the $\G_m$-actions of $\left\{(\sX_i,\sL_i)\right\}_{i=1}^k$, but it agrees with the one in Definition \ref{definition:ding} by Proposition \ref{proposition:hashimoto}. The second and more minor point is that \cite[Definition 19]{hashimoto} is defined for a $k$-tuple of very ample test configurations with exponent $m$, but the definition of the coupled Ding invariant therein involves re-scaling by the exponent and hence agrees with Definition \ref{definition:ding}. Thus, the right hand side of (\ref{apeqcpdsf}) is indeed the coupled Ding invariant defined in this paper.

Furthermore, if we have a $k$-tuple of $\T$-equivariant ample test configurations $\left\{(\sX_i,\sL_i)\right\}_{i=1}^k$, we may choose the ray $\{ (\phi_1 (t) , \dots , \phi_k (t) ) \}_{t \ge 0}$ to be contained in $\bm{\mathcal{H}}^{\T_r}$; in terms of the explicit formulation using the generators of the $\G_m$-action, this fact can be proved by noting that the generator is $\T_r$-invariant if it commutes with the $\T$-action (see e.g.~\cite[Lemma 5 and the discussion that follows]{hashimoto}). Moreover, by considering $\sum_{i=1}^k \mathrm{J}_i$ in the proof of \cite[Theorem 3.14]{Li} and together with the slope formula (\ref{apeqjisf}), we find that \cite[Theorem 3.14]{Li} generalizes to the coupled case so that
\begin{equation} \label{apeqcpjsf}
	\lim_{t \to + \infty} \frac{\mathrm{J}_{\mathrm{cp}, \T} (\phi_1 , \dots , \phi_k )}{t} = \inf_{\xi \in N_{\Q} (\T)} \sum_{i=1}^k \JJ \left(\sX_{i, \xi},\sL_{i,\xi} \right).
\end{equation}

With the preparation above, suppose that we are given a $k$-tuple of $\T$-equivariant ample test configurations $\left\{(\sX_i,\sL_i)\right\}_{i=1}^k$. The condition (\ref{apexcoercpdg}) implies
\begin{equation*}
	\mathrm{D}^{\mathrm{cp}} (\phi_1  (t) , \dots , \phi_k (t) ) \ge \varepsilon \mathrm{J}_{\mathrm{cp}, \T} (\phi_1 (t), \dots , \phi_k (t) ) - \frac{1}{\varepsilon}
\end{equation*}
for any $t \ge 0$ in the associated ray $\{ (\phi_1 (t) , \dots , \phi_k (t) ) \}_{t \ge 0} \subset \bm{\mathcal{H}}^{\T_r}$. Dividing both sides of this inequality by $t$ and taking the limit $t \to + \infty$, we get
\begin{equation} \label{apeqcudgst}
	\Ding\left(\{\sX_i,\sL_i\}_{i=1}^k\right) \geq \varepsilon \inf_{\xi \in N_{\Q} (\T)} \sum_{i=1}^k \JJ \left(\sX_{i, \xi},\sL_{i,\xi} \right) = \varepsilon\cdot
\JJ^{\cp}_{\T} \left(\{\sX_i,\sL_i\}_{i=1}^k\right)
\end{equation}
by the slope formulae (\ref{apeqcpdsf}) and (\ref{apeqcpjsf}), for any $k$-tuple of $\T$-equivariant ample test configurations $\left\{(\sX_i,\sL_i)\right\}_{i=1}^k$.

We note that the above inequality (\ref{apeqcudgst}) establishes the $\T$-reduced coupled uniform Ding stability. Indeed, all the terms appearing in the coupled Ding invariant depend only on the non-Archimedean metric on $L_i$ represented by $(\sX_i,\sL_i)$, for each $1 \le i \le k$, by \cite[\S 6.1 and \S 7]{BHJ}, \cite[Theorems 3.6 and 3.7]{BHJ2}, and Lemma \ref{lemma:aplmstcna}. Since each non-Archimedean metric contains a unique ample normal model \cite[Lemma 6.3]{BHJ}, we thus conclude that the inequality (\ref{apeqcudgst}) holds for any $k$-tuple of $\T$-equivariant semiample test configurations $\left\{(\sX_i,\sL_i)\right\}_{i=1}^k$, establishing the claimed result.
\end{proof}

We now prove the converse, which is generally considered to be harder. The key result is the following geodesic stability.
	
\begin{thm} \label{thgdsttc}
	Suppose that $\left(X;\{L_i\}_{i=1}^k\right)$ has vanishing coupled $\T$-Futaki characters. If $\mathrm{D}^{\mathrm{cp}}$ is not $\T$-coercive, there exists a $k$-tuple of maximal geodesic rays $\{ (\phi_1 (t) , \dots , \phi_k (t) ) \}_{t \ge 0} \subset (\bm{\mathcal{E}}^1)^{\T_r}$, with $\sup_X \phi_1 (t) = \cdots = \sup_X \phi_k (t)=0$ for all $t \ge 0$, and a constant $C \ge 0$ independent of $t$ such that
	\begin{equation*}
		\frac{\mathrm{D}^{\mathrm{cp}} (\phi_1 (t) , \dots , \phi_k (t) ) -C }{t} \le 0
	\end{equation*}
	for all $t > 0$, and
	\begin{equation*}
		\lim_{t \to + \infty} \frac{\mathrm{J}_{\mathrm{cp}, \T} (\phi_1 (t) , \dots , \phi_k (t) )}{t} > 0.
	\end{equation*}
\end{thm}

\begin{remark}
Results of this type are often proved by comparing the Ding functional with the $K$-energy, since the commonly known proof such as \cite[Theorem 2.16]{BBJ} rests on the compactness result involving the entropy (see \cite[Definition 2.12]{BBJ}). Surprisingly, Darvas--Zhang \cite[Theorem 5.3]{DZ} provided a proof which uses another compactness result for psh functions, and hence circumvents the use of the $K$-energy. It is still natural to find an approach using the coupled version of $K$-energy, but this problem will be treated in a separate paper.
\end{remark}

Before we start the proof of the theorem above, we first explain some preliminary results involving the holomorphic vector fields. Suppose that we write $\sigma_{\xi} : \G_m \to \T$ for the 1-parameter subgroup generated by $\xi \in N_{\Z} (\T )$, and that we extend the action to $\xi \in N_{\R} (\T )$ by the matrix exponential function. We define a K\"ahler potential $ \psi^{\xi}_{i,t}$ by
\begin{equation*}
	\sigma_{\xi } (t)^* \theta_i = \theta_i + \sqrt{-1} \partial \bar{\partial} \psi^{\xi}_{i,t}
\end{equation*}
for $1 \le i \le k$, and it is well-known that $\{ \psi^{\xi}_{i,t} \}_{t \in \mathbb{R}}$ defines a geodesic line in $\mathcal{H}^{\T_r}_i$. Writing $V_{\xi} \in H^0 (X, T^{1,0}X)$ for the holomorphic vector field generated by $\xi$, it is well-known that $\dot{\psi}^{\xi}_{i,t}$ is the holomorphy potential of $V_{\xi}$ with respect to $\theta_i$. With this notation, the coupled $\T$-Futaki character $\mathrm{Fut}^{\mathrm{cp}} : H^0 (X, T^{1,0}X) \to \mathbb{C}$ was defined in \cite[Definition 1.1]{FZ} as
\begin{equation*}
	\mathrm{Fut}^{\mathrm{cp}} (V_{\xi}) := \sum_{i=1}^n \frac{\int_X \dot{\psi}^{\xi}_{i,t} (\sigma_{\xi}(t)^* \theta_i)^n}{\int c_1 (L_i)^n}.
\end{equation*}

We prove two auxiliary lemmas. The first lemma can be regarded as an analytic version of Corollary \ref{corollary:twist-ding}.

\begin{lemma} \label{lmcpfich1}
	The coupled $\T$-Futaki character vanishes for all $\xi \in N_{\R} (\T )$ if and only if $\alpha_{\bc}^{\cp} = 0$. Moreover, with respect to the notation above, we have
	\begin{equation*}
		\frac{d}{dt} \sum_{i=1}^k \mathrm{E}_i (\psi^{\xi}_{i,t}) = \frac{d}{dt} \mathrm{D}^{\mathrm{cp}} (\psi^{\xi}_{1,t} , \dots , \psi^{\xi}_{k,t}) = \mathrm{Fut}^{\mathrm{cp}} (V_{\xi}) .
	\end{equation*}
\end{lemma}

\begin{proof}
	Recalling the equivariant Riemann--Roch theorem (see e.g.~\cite[Proposition 7.12]{Sze}), we find that the coupled $\T$-Futaki character equals the sum of Chow weights over $i=1, \dots , k$ for $\xi \in N_{\R} (\T )$. Thus, using \cite[(60)]{LiCsck} and the necessary adaptation as in Definition \ref{definition:cFutaki}, we find $\mathrm{Fut}^{\mathrm{cp}} (V_{\xi}) =0$ for all $\xi \in N_{\R} (\T )$ if and only if $\alpha_{\bc}^{\cp} = 0$, by arguing as in \cite[Lemma 2.40]{Xu}.
	
	The equality
	\begin{equation*}
		\frac{d}{dt} \sum_{i=1}^k \mathrm{E}_i (\psi^{\xi}_{i,t})  = \mathrm{Fut}^{\mathrm{cp}} (V_{\xi}) .
	\end{equation*}
	is obvious from the formula for the derivative of $\mathrm{E}_i$; note that an essentially equivalent result is also given in \cite[Lemma 2.23]{LiCsck}.
	
	We then recall that $\theta_i$ is the K\"ahler form of the hermitian metric $h_i$ on $L_i$, and that the volume form in $\mathrm{L}^{\mathrm{cp}}$ is the one naturally defined by $h_1 \otimes \cdots \otimes h_k$. Thus, we find
	\begin{equation*}
		\sigma_{\xi } (t)^*(d \mu') = \sigma_{\xi } (t)^*(h_1 \otimes \cdots \otimes h_k) = \exp \left( - \sum_{i=1}^k \psi^{\xi}_{i,t} \right) d \mu' ,
	\end{equation*}
	which proves $\frac{d}{dt} \mathrm{L}^{\mathrm{cp}} (\psi^{\xi}_{1,t} , \dots , \psi^{\xi}_{k,t}) = 0$ (see also \cite[Lemma 2.6]{kewei2} for a similar argument), yielding the required result.
\end{proof}

\begin{lemma} \label{lmcpfich2}
	For any geodesic segment $\{ (u_1 (t) , \dots , u_k (t) ) \}_{t} \subset \bm{\mathcal{E}}^1$, define its twist by $\xi \in N_{\R} (\T )$ as
	\begin{equation*}
		u_i^{\xi} (t) := \sigma_{\xi}(t)^* u_i (t) + \psi^{\xi}_{i,t}
	\end{equation*}
	following \cite[(12)]{kewei2}. If $\left(X;\{L_i\}_{i=1}^k\right)$ has vanishing coupled $\T$-Futaki characters, we have
	\begin{equation*}
		 \mathrm{E}_i (u_i^{\xi} (t)) =  \mathrm{E}_i (u_i (t))
	\end{equation*}
	for all $1 \le i \le k$ and
	\begin{equation*}
		\mathrm{L}^{\mathrm{cp}} (u_1^{\xi} (t), \dots , u_k^{\xi} (t)) = \mathrm{L}^{\mathrm{cp}} (u_1 (t), \dots , u_k (t)) .
	\end{equation*}
\end{lemma}

\begin{proof}
Arguing as in \cite[Lemma 2.6]{kewei2}, the claimed statements follow from Lemma \ref{lmcpfich1}.
\end{proof}

We now prove the main result.

\begin{proof}[Proof of Theorem \ref{thgdsttc}]
	We largely follow the arguments in the proof of \cite[Theorem 5.3]{DZ} and \cite[Theorem 3.5]{kewei2}. By the hypothesis, for any $j \in \mathbb{Z}_{>0}$ there exists $(\phi^{(j)}_1 , \dots , \phi^{(j)}_k) \in (\bm{\mathcal{E}}^1)^{\T_r}$ such that
	\begin{equation*}
		\mathrm{D}^{\mathrm{cp}} (\phi^{(j)}_1 , \dots , \phi^{(j)}_k) \le \frac{1}{j} \mathrm{J}_{\mathrm{cp}, \T} (\phi^{(j)}_1 , \dots , \phi^{(j)}_k ) -j.
	\end{equation*}
	We now observe that the proof of \cite[Lemma 2.15]{Li} easily generalizes to the coupled case, so that we may assume that
	\begin{equation*}
		\mathrm{J}_{\mathrm{cp}, \T} (\phi^{(j)}_1 , \dots , \phi^{(j)}_k ) = \sum_{i=1}^k \mathrm{J}_{i} (\phi^{(j)}_i)
	\end{equation*}
	and $\sup_X \phi^{(j)}_i = 0$, which implies $\mathrm{J}_{i} (\phi^{(j)}_i) \le - \mathrm{E}_{i} (\phi^{(j)}_i)$, for all $1 \le i \le k$ and $j \in \mathbb{Z}_{>0}$. We thus get
	\begin{equation*}
		\mathrm{L}^{\mathrm{cp}} (\phi^{(j)}_1 , \dots , \phi^{(j)}_k) \le \left( 1 - \frac{1}{j} \right) \sum_{i=1}^k \mathrm{E}_{i} (\phi^{(j)}_i) -j
	\end{equation*}
	for all $j \in \mathbb{Z}_{>0}$.

	We claim that $ \sum_{i=1}^k d_{1,i} (\phi^{(j)}_i, 0 ) \to + \infty$, and hence there is at least one index $i$ such that $d_{1,i} (\phi^{(j)}_i, 0 ) \to + \infty$ as $j \to \infty$. If there exists $C >0$ such that $d_{1,i} (\phi^{(j)}_i, 0 ) \le C$ for all $i$ and $j$, or equivalently $\mathrm{E}_{i} (\phi^{(j)}_i) \ge -C$ since $\sup_X \phi^{(j)}_i = 0$ (by \cite[Proposition 3.43]{darlectures}), there exists $\phi^{\infty}_i \in \mathcal{E}^1_i$ such that $\phi^{(j)}_i \to \phi^{\infty}_i$ in $L^1$-topology as $j \to \infty$, up to passing to a subsequence, for all $1 \le i \le k$ \cite[Proposition 10.23]{gzbook}. We may moreover assume $\phi^{\infty}_i \in (\mathcal{E}^1_i)^{\T_r}$, since it is a limit of $\T_r$-invariant functions. Since $\mathrm{E}_{i}$ is upper semicontinuous \cite[Corollary 4.14]{darlectures} and $\mathrm{L}$ is continuous in $L^1$-topology (as pointed out in \cite[proof of Proposition 5.5]{DZ}), together with $\mathrm{L}^{\mathrm{cp}}(\phi^{(j)}_1 , \dots , \phi^{(j)}_k) = \mathrm{L} \left( \sum_{i=1}^k \phi^{(j)}_i \right)$ as in (\ref{eqlcplsm}), we find
	\begin{equation*}
		- \infty < \mathrm{D}^{\mathrm{cp}} (\phi^{\infty}_1 , \dots , \phi^{\infty}_k) \le \liminf_{j \to \infty} \mathrm{D}^{\mathrm{cp}} (\phi^{(j)}_1 , \dots , \phi^{(j)}_k) \le \liminf_{j \to \infty} \left( \frac{kC}{j} -j \right) = - \infty
	\end{equation*}
	which is a contradiction.
	
	We thus find, together with \cite[Proposition 3.43]{darlectures}, that
	\begin{equation*}
		\limsup_{j \to \infty} \frac{\mathrm{L} \left( \sum_{i=1}^k \phi^{(j)}_i \right)}{\sum_{i=1}^k d_{1,i} (\phi^{(j)}_i,0)} = \limsup_{j \to \infty} \frac{\mathrm{D}^{\mathrm{cp}}(\phi^{(j)}_1 , \dots , \phi^{(j)}_k) - \sum_{i=1}^k d_{1,i} (\phi^{(j)}_i,0)}{\sum_{i=1}^k d_{1,i} (\phi^{(j)}_i,0)} \le -1
	\end{equation*}
	with
	\begin{equation*}
		\sum_{i=1}^k d_{1,i} (\phi^{(j)}_i,0) \to + \infty
	\end{equation*}
	as $j \to \infty$.
	

	By re-ordering the indices $i=1, \dots , k$ if necessary, there exists $1 \le l \le k$ such that $d_{1,i} (\phi^{(j)}_i, 0 ) \to + \infty$ as $j \to \infty$ if and only if $1 \le i \le l$, and $d_{1,i} (\phi^{(j)}_i, 0 )$ remains bounded, by $C>0$ say, if and only if $l+1 \le i \le k$. For $1 \le i \le l$, we set $[0, d_{1,i} (\phi^{(j)}_i ,0) ] \ni t \mapsto \phi^{(j)}_i (t) \in (\mathcal{E}^1_i)^{\T_r}$ to be the unit speed finite energy geodesic ray connecting $0$ and $\phi^{(j)}_i$, noting that we may assume $\phi^{(j)}_i (t)$ is $\T_r$-invariant since it can be approximated by $C^{1,\bar{1}}$-geodesics \cite[Proposition 3.15]{darlectures} which we may assume are $\T_r$-invariant by integrating the homogeneous Monge--Amp\`ere equation \cite[\S 3.1]{darlectures} over $\T_r$. With \cite[Remark 3.3]{DZ}, we thus get
	\begin{equation} \label{eqpijtsp}
		\sup_X \phi^{(j)}_i (t) =0, \quad \mathrm{E}_i (\phi^{(j)}_i (t)) = -t
	\end{equation}
	for all $t \in [0, d_{1,i} (\phi^{(j)}_i ,0) ]$ and all $1 \le i \le l$. For $l+1 \le i \le k$, we define a constant subgeodesic ray $\phi^{(j)}_i (t) \equiv \phi^{(j)}_i $ for all $t \in [0, d_{1,i} (\phi^{(j)}_i ,0) ]$. We note that $\sum_{i=1}^k \phi_i (t)$ is a sublinear subgeodesic ray in $(\mathcal{E}^1 (-K_X))^{\T_r}$, since
	\begin{equation*}
		\sum_{i=1}^k \theta_i + \sqrt{ -1} \partial_{X,t} \bar{\partial}_{X,t} \left( \sum_{i=1}^k \phi_i (t) \right) = \sum_{i=1}^k \left( \theta_i + \sqrt{ -1} \partial_{X,t} \bar{\partial}_{X,t} \phi_i (t) \right) \ge 0,
	\end{equation*}
	and the sublinearity is obvious under addition. Then the convexity of $\mathrm{L}$ for subgeodesics in $\mathcal{E}^1 (-K_X)$ \cite[Theorem 1.1]{Ber} implies
	\begin{equation*}
		\limsup_{j \to \infty} \frac{\mathrm{L} \left( \sum_{i=1}^k \phi^{(j)}_i (t) \right)}{\sum_{i=1}^l d_{1,i} (\phi^{(j)}_i(t),0) + \sum_{i=l+1}^k d_{1,i} (\phi^{(j)}_i(t),0)} \le -1 ,
	\end{equation*}
	for all $0 \le t \le \min_{1 \le i \le l} d_{1,i} (\phi^{(j)}_i ,0)$. By (\ref{eqpijtsp}) and \cite[Proposition 3.43]{darlectures}, we also have
	\begin{equation*}
		d_{1,i}(\phi^{(j)}_i (t), 0) = - \mathrm{E}_{i} (\phi^{(j)}_i (t)) = t,
	\end{equation*}
	for all $1 \le i \le l$ and all $0 \le t \le \min_{1 \le i \le l} d_{1,i} (\phi^{(j)}_i ,0)$. We thus get
	\begin{equation*}
		\limsup_{j \to \infty} \frac{\mathrm{L} \left( \sum_{i=1}^k \phi^{(j)}_i (t) \right)}{lt} \le  -1
	\end{equation*}
	for all $1 \le i \le l$ and all $0 \le t \le \min_{1 \le i \le l} d_{1,i} (\phi^{(j)}_i ,0)$.
	
	By arguing exactly as in the proof of \cite[Theorem 5.3]{DZ}, we find that there exists a finite energy sublinear subgeodesic ray $\{ \tilde{\phi}_i (t) \}_{t \ge 0}$ such that $\phi^{(j)}_i (t)$ converges to $\tilde{\phi}_i (t)$ with respect to the $L^1$-convergence for all $t \in (0, + \infty) \setminus E$ for some set $E$ of Lebesgue measure zero, for all $1 \le i \le k$ (for $l+1 \le i \le k$, this is a constant subgeodesic ray). We note that $\tilde{\phi}_i (t)$ is $\T_r$-invariant, since it is the limit of $\T_r$-invariant functions.
	
	Thus, for any $t \in (0, + \infty) \setminus E$ we have
	\begin{equation} \label{eqprpitspn}
		\frac{\mathrm{L} \left( \sum_{i=1}^k \tilde{\phi}_i (t) \right)}{l t} \le -1
	\end{equation}
	and
	\begin{equation} \label{eqpitspn}
		\sup_X \tilde{\phi}_i (t) =0, \quad 0 \ge \mathrm{E}_{i} (\tilde{\phi}_i (t) ) \ge -t
	\end{equation}
	for all $1 \le i \le l$, since $\mathrm{L}$ is $L^1$-continuous, where we note that the supremum being zero is preserved under the $L^1$-convergence of plurisubharmonic functions by the mean value inequality, and $\mathrm{E}_{i}$ is $L^1$-usc (upper semicontinuous) by \cite[Corollary 4.14]{darlectures}. Again as in the proof of \cite[Theorem 5.3]{DZ} the above holds for $t \in (0, + \infty)$ and the subgeodesics are nontrivial for $1 \le i \le l$. We then find, from (\ref{eqprpitspn}) and (\ref{eqpitspn}) that
	\begin{align*}
		\frac{\mathrm{D}^{\mathrm{cp}} (\tilde{\phi}_1 (t) , \dots , \tilde{\phi}_k (t) ) - C(k-l)}{t} &= \frac{\mathrm{L} \left( \sum_{i=1}^k \tilde{\phi}_i (t) \right) - \sum_{i=1}^k \mathrm{E}_{i} (\tilde{\phi}_i (t) ) - C(k-l)}{t} \\
		&\le \frac{-lt+(lt+C(k-l)) -C(k-l)}{t} \le 0.
	\end{align*}

	The non-Archimedean metric $\phi^{\mathrm{NA}}_i$ associated to the subgeodesic ray $\tilde{\phi}_i (t)$ is defined as in \cite[Definition 4.2]{BBJ}, for $1 \le i \le k$. We note moreover that $\phi^{\mathrm{NA}}_i \in (\mathcal{E}^{1 , \mathrm{NA}}_i)^{\T_r}$ (see \cite[\S 2.1.3]{LiCsck} for the definition), since it is a decreasing limit of the approximations that are $\T_r$-invariant (see the proof of \cite[Theorem 6.6]{BBJ}), where we recall that each approximation \cite[Lemmas 5.6 and 5.7]{BBJ} by test configurations is given by a $\T_r$-invariant multiplier ideal sheaf. Note that $\phi^{\mathrm{NA}}_i$ is trivial for $l+1 \le i \le k$, and $\phi^{\mathrm{NA}}_i \le 0$ for $1 \le i \le k$ by (\ref{eqpitspn}) and \cite[Lemma 4.3]{BBJ}. We recall that there exists a maximal geodesic ray $\phi_i (t)$ whose associated non-Archimedean metric is $\phi^{\mathrm{NA}}_i$, by \cite[Theorem 6.6]{BBJ} (\cite{DZ} uses the maximization, but it was not obvious to us that taking the maximization of each $\tilde{\phi}_i (t)$, rather than the maximization of $\sum_i \tilde{\phi}_i (t)$, yields the required result). We may even assume that $\phi_1 (t) , \dots , \phi_k (t)$ are all $\T_r$-invariant, since it is a decreasing limit of the approximations that are $\T_r$-invariant (see the proof of \cite[Theorem 6.6]{BBJ}). Moreover, since $\phi_i^{\mathrm{NA}} \le 0$ for all $1 \le i \le k$, we find that
	\begin{equation*}
		\sup_X \phi_i (t) =0
	\end{equation*}
	for all $t$; we have $\sup_X \phi_i (t) = ct$ for some $c \le 0$ by \cite[Remark 3.3]{DZ} and by comparing it with the trivial geodesic ray, but $c <0$ contradicts maximality. Thus, by \cite[Corollary 6.7]{BBJ} and \cite[Lemma 2.41]{LiCsck}, we have
	\begin{equation} \label{eqeinamgjina}
		- \mathrm{E}^{\mathrm{NA}}_i (\phi_i^{\mathrm{NA}}) = - \lim_{t \to + \infty} \frac{\mathrm{E}_{i} (\phi_i (t) )}{t} = \lim_{t \to + \infty} \frac{\mathrm{J}_{i} (\phi_i (t) )}{t} = \mathrm{J}^{\mathrm{NA}}_i (\phi_i^{\mathrm{NA}}) \ge 0
	\end{equation}
	for all $1 \le i \le k$.

	We replace $\tilde{\phi}_i (t)$ by the maximal geodesic ray $\phi_i (t)$ above. For these maximal geodesic rays, we have 
	\begin{equation} \label{eqntlcppht}
		\frac{\mathrm{L} \left( \sum_{i=1}^k \phi_i (t) \right)}{l t} \le -1 ,
	\end{equation}
	since its limit $t \to + \infty$ agrees with that of $\mathrm{L} \left( \sum_{i=1}^k \tilde{\phi}_i (t) \right)/ (l t)$ by \cite[Theorem 5.4 and the following paragraph]{BBJ}, and again by the convexity \cite[Theorem 1.1]{Ber}. Moreover, the slope of $- \mathrm{E}_{i} (\phi_i (t) )$ is smaller than that of $- \mathrm{E}_{i} (\tilde{\phi}_i (t) )$ by the maximality of the geodesic and the monotonicity of $\mathrm{E}_{i}$ \cite[Proposition 3.43]{darlectures}. Thus, with these replacements, we get the required inequality
	\begin{equation*}
		\frac{\mathrm{D}^{\mathrm{cp}} (\phi_1 (t) , \dots , \phi_k (t) ) -C(k-l)}{t}  \le 0.
	\end{equation*}
	Now, by definition \cite[Definition 4.2 and Appendix B]{BBJ}, the non-Archimedean metric associated to the subgeodesic ray $\sum_{i=1}^k \phi_i (t)$ in $\mathcal{E}^1(-K_X)$ is given by $\sum_{i=1}^k \phi_i^{\mathrm{NA}} = \mathsf{S} (\phi_1^{\mathrm{NA}}, \dots , \phi_k^{\mathrm{NA}})$ (in the sense of Proposition \ref{ppsmctnna}). The inequality (\ref{eqntlcppht}), together with \cite[Theorem 5.4 and the following paragraph]{BBJ}, implies that $\sum_{i=1}^k \phi_i^{\mathrm{NA}}$ cannot be trivial. Thus, recalling $\phi_i^{\mathrm{NA}} \le 0$ for all $1 \le i \le k$, we find that there exists at least one $i$ with nontrivial $\phi_i^{\mathrm{NA}}$, which implies in particular $ - \sum_{i=1}^k \mathrm{E}^{\mathrm{NA}}_i (\phi_i^{\mathrm{NA}})>0 $ by \cite[Corollary 10.5]{BJ22}, noting that $X$ is smooth in our case. This in turn implies 
	\begin{equation*}
		- \lim_{t \to + \infty} \sum_{i=1}^k \frac{\mathrm{E}_{i} (\phi_i (t) )}{t} =: a > 0
	\end{equation*}
	again by \cite[Corollary 6.7]{BBJ}.

	Furthermore, recalling that we have vanishing coupled $\T$-Futaki characters and also Lemma \ref{lmcpfich2}, we can argue as in \cite[Proof of Proposition 6.2]{LiCsck} or \cite[page 15]{kewei2} to prove the following: if we have
	\begin{equation*}
		- \lim_{t \to + \infty} \sum_{i=1}^k \frac{\mathrm{E}_{i} (\phi_i (t) )}{t} \ge a >0 ,
	\end{equation*}
	then
	\begin{equation} \label{eqntjcpt}
		\lim_{t \to + \infty} \frac{\mathrm{J}_{\mathrm{cp}, \T} (\phi_1 , \dots , \phi_k )}{t} \ge a .
	\end{equation}
	Indeed, if $\phi_i (t)$ is the maximal geodesic and $\tilde{\phi}_i (t)$ is the subgeodesic as above, \cite[(12)]{kewei2} implies that $\phi^{\xi}_t (t) \ge \tilde{\phi}^{\xi}_i (t)$ for the twist (in the sense of Lemma \ref{lmcpfich2}) by any $\xi \in N_{\mathbb{R}} (\T )$. Thus, following \cite[page 16]{kewei2}, we have
	\begin{align*}
		\sum_{i=1}^k \mathrm{J}_{i} (\phi^{\xi}_i (t) )  &= \sum_{i=1}^k  \frac{1}{\int_X c_1 (L_i)^n} \int_X \phi^{\xi}_i (t) \theta^n_i - \sum_{i=1}^k  \mathrm{E}_i (\phi_i (t)) \\
		&\ge \lim_{j \to \infty} \sum_{i=1}^k \frac{1}{\int_X c_1 (L_i)^n} \int_X \phi^{(j), \xi}_i (t) \theta_i^n - \sum_{i=1}^k \mathrm{E}_i (\phi_i (t)) \\
		&\ge \lim_{j \to \infty} \sum_{i=1}^k \left( \mathrm{J}_{i} ( \phi^{(j), \xi}_i (t) ) -t \right) +at
	\end{align*}
	from Lemma \ref{lmcpfich1} and (\ref{eqpijtsp}), where we used the maximality of $\phi_i (t)$ in the second line and the linearity of $\mathrm{E}_i$ along geodesics in the third line \cite[Proposition 3.42]{darlectures}. The claimed inequality (\ref{eqntjcpt}) follows from the one above by arguing exactly as in \cite[page 16]{kewei2}, thereby establishing all the claimed results.
\end{proof}

\begin{thm} \label{thsttcsv}
	If $X$ is $\T$-reduced uniformly coupled Ding stable, then $\mathrm{D}^{\mathrm{cp}}$ is $\T$-coercive.
\end{thm}

\begin{proof}
	Suppose for contradiction that $\mathrm{D}^{\mathrm{cp}}$ is not $\T$-coercive. Then, Theorem \ref{thgdsttc} implies that there exists a $k$-tuple of maximal geodesic rays $\{ ( \phi_1 (t) , \dots , \phi_k (t) \}_{t \ge 0} \subset (\bm{\mathcal{E}}^1)^{\T_r}$ such that $\sup_X \phi_1 (t) = \cdots = \sup_X \phi_k (t)=0$ for all $t \ge 0$ and
	\begin{equation} \label{eqcsqdjt}
		\frac{\mathrm{D}^{\mathrm{cp}} (\phi_1 (t) , \dots , \phi_k (t) ) }{t} \le \frac{C}{t} , \quad \lim_{t \to + \infty} \frac{\mathrm{J}_{\mathrm{cp}, \T} (\phi_1 (t) , \dots , \phi_k (t))}{t}  \ge a
	\end{equation}
	for some $C, a>0$ and all $t >0$, by noting that the coupled $\T$-Futaki characters must vanish due to $\T$-reduced uniform coupled Ding stability, by considering geodesic lines generated by holomorphic vector fields (see Corollary \ref{corollary:twist-ding} and Lemma \ref{lmcpfich1}). Thus, taking the limit $t \to +\infty$ in the above and writing $\phi_i^{\mathrm{NA}} \in (\mathcal{E}^{1, \mathrm{NA}}_i)^{\T_r}$ for the non-Archimedean metric associated to $\phi_i (t)$, we have
	\begin{equation*}
		\lim_{t \to + \infty} \frac{\mathrm{D}^{\mathrm{cp}} (\phi_1 (t) , \dots , \phi_k (t) ) }{t} = \mathrm{L}^{\mathrm{NA}} \left( \sum_{i=1}^k \phi_i^{\mathrm{NA}} \right) - \sum_{i=1}^k \mathrm{E}_i^{\mathrm{NA}} (\phi_i^{\mathrm{NA}}) \le 0
	\end{equation*}
	and
	\begin{equation} \label{eqjnacpgta}
		\mathrm{J}^{\mathrm{NA}}_{\mathrm{cp}, \T} (\phi^{\mathrm{NA}}_1 , \dots , \phi^{\mathrm{NA}}_k ) \ge a, 
	\end{equation}
	where we define
	\begin{equation*}
		\mathrm{J}^{\mathrm{NA}}_{\mathrm{cp}, \T} (\phi^{\mathrm{NA}}_1 , \dots , \phi^{\mathrm{NA}}_k ) := \inf_{\xi \in N_{\mathbb{R}} (\T )} \sum_{i=1}^k \mathrm{J}^{\mathrm{NA}}_{i} (\phi^{\mathrm{NA}}_{i,\xi}),
	\end{equation*}
	in which $\phi^{\mathrm{NA}}_{i,\xi}$ is defined as in \cite[Definition 2.20]{LiCsck}. Note that (\ref{eqjnacpgta}) follows from (\ref{eqeinamgjina}) and (\ref{eqcsqdjt}), together with the definition of the twist $\theta_i + \sqrt{-1} \partial \bar{\partial} \phi^{\xi}_i (t) = \sigma_{\xi} (t)^* (\theta_i + \sqrt{-1} \partial \bar{\partial} \phi_i (t))$, since they yield $\sum_{i=1}^k \mathrm{J}^{\mathrm{NA}}_{i} (\phi^{\mathrm{NA}}_{i,\xi}) \ge a$ for all $\xi \in N_{\mathbb{R}} (\T )$; strictly speaking we need to check that the non-Archimedean metric associated to $\phi^{\xi}_i (t)$ is $\phi^{\mathrm{NA}}_{i, - \xi}$, with the sign convention as in Remark \ref{remark:theta}, but this follows from the computation for the Gauss extension in \cite[page 1542]{LiCsck} (or alternatively from approximating $\phi^{\xi}_i (t)$ by a decreasing sequence of geodesic rays corresponding to test configurations \cite[Proof of Theorem 6.6]{BBJ} and taking the limit in \cite[Lemma 2.19]{LiCsck}).

	We apply the construction in \cite[\S 5.3]{BBJ} to the maximal geodesic rays $\phi_1 (t) , \dots , \phi_k (t)$, to find that the non-Archimedean metric $\phi_i^{\mathrm{NA}} \in (\mathcal{E}^{1, \mathrm{NA}}_i)^{\T_r}$ associated to $\phi_i (t)$ can be approximated by $\varphi^{\mathrm{NA}}_{i,m} \in (\mathcal{H}^{\mathrm{NA}}_i)^{\T_r}$, by using \cite[Lemma 5.7]{BBJ}. As previously noted, the non-Archimedean metric associated to the subgeodesic ray $\sum_{i=1}^k \phi_i (t)$ in $\mathcal{E}^1 (-K_X)$ is given by $\sum_{i=1}^k \phi_i^{\mathrm{NA}}$; note that this observation, together with Lemma \ref{lemma:aplmstcna}, leads to an alternative proof of the slope formula (\ref{apeqcpdsf}). Arguing exactly as in \cite[Proof of Lemma 5.7]{BBJ}, we find
	\begin{equation*}
		\mathrm{L}^{\mathrm{NA}} \left( \sum_{i=1}^k \varphi^{\mathrm{NA}}_{i,m} \right) \to \mathrm{L}^{\mathrm{NA}} \left( \sum_{i=1}^k \phi_i^{\mathrm{NA}} \right)
	\end{equation*}
	as $m \to \infty$.

	
	Note moreover that the approximation given in \cite[Lemma 5.7]{BBJ} converges with respect to the strong topology in $\mathcal{E}^1_i$, since after re-labelling as in \cite[page 1548]{LiCsck} we find that $\{ \varphi^{\mathrm{NA}}_{i,m} \}_m$ is a decreasing sequence converging to $\phi^{\mathrm{NA}}_i$ (see also \cite[Example 12.2]{BJ22}). We thus get
	\begin{equation*}
		\mathrm{E}_i^{\mathrm{NA}} (\phi_i^{\mathrm{NA}}) = \lim_{m \to \infty} \mathrm{E}_i^{\mathrm{NA}} (\varphi^{\mathrm{NA}}_{i,m})
	\end{equation*}
	by \cite[Definition 12.1]{BJ22}, and
	\begin{equation*}
		\mathrm{J}^{\mathrm{NA}}_{\mathrm{cp}, \T} (\phi^{\mathrm{NA}}_1 , \dots , \phi^{\mathrm{NA}}_k ) = \lim_{m \to \infty} \mathrm{J}^{\mathrm{NA}}_{\mathrm{cp}, \T} (\varphi^{\mathrm{NA}}_{1,m} , \dots , \varphi^{\mathrm{NA}}_{k,m} ) 
	\end{equation*}
	by noting that the proof of \cite[Lemma 6.4]{LiCsck} easily generalizes to the coupled case as above. In particular, we have
	\begin{equation*}
		\mathrm{J}^{\mathrm{NA}}_{\mathrm{cp}, \T} (\varphi^{\mathrm{NA}}_{1,m} , \dots , \varphi^{\mathrm{NA}}_{k,m} )  > \frac{a}{2} >0 
	\end{equation*}
	for all large enough $m$.
	
	Thus, for any $\varepsilon >0$ and for all large enough $m$, we have
	\begin{equation*}
		\mathrm{L}^{\mathrm{NA}} \left(\sum_{i=1}^k \varphi^{\mathrm{NA}}_{i,m} \right) - \sum_{i=1}^k \mathrm{E}_i^{\mathrm{NA}} (\varphi^{\mathrm{NA}}_{i,m}) - \varepsilon < \lim_{t \to + \infty} \frac{\mathrm{D}^{\mathrm{cp}} (\phi_1 (t) , \dots , \phi_k (t) ) }{t} \le 0,
	\end{equation*}
	but the first two terms on the left hand side add up to the coupled Ding invariant for the $k$-tuple of $\T$-equivariant test configurations $\left\{(\sX_{i,m},\sL_{i,m})\right\}_{i=1}^k$ representing $\varphi^{\mathrm{NA}}_{1,m} , \dots , \varphi^{\mathrm{NA}}_{k,m}$, since 
	\begin{align*}
		&\mathrm{L}^{\mathrm{NA}} \left(\sum_{i=1}^k \varphi^{\mathrm{NA}}_{i,m} \right) - \sum_{i=1}^k \mathrm{E}_i^{\mathrm{NA}} (\varphi^{\mathrm{NA}}_{i,m}) \\
		&= \mathrm{L}^{\mathrm{NA}} \left(\sum_{i=1}^k \varphi^{\mathrm{NA}}_{i,m} \right) - \mathrm{E}^{\mathrm{NA}} \left(\sum_{i=1}^k \varphi^{\mathrm{NA}}_{i,m} \right) +  \mathrm{E}^{\mathrm{NA}} \left(\sum_{i=1}^k \varphi^{\mathrm{NA}}_{i,m} \right) - \sum_{i=1}^k \mathrm{E}_i^{\mathrm{NA}} (\varphi^{\mathrm{NA}}_{i,m}) \\
		&=\Ding  (\sX_{m},\sL_{m}) +\frac{\left(\bar{\sL}_m^{\cdot n+1}\right)}{(n+1)(L^{\cdot n})} -\sum_{i=1}^k \frac{\left(\bar{\sL}_{i,m}^{\cdot n+1}\right)}{(n+1)(L_i^{\cdot n})} \\
		&=\Ding\left(\{\sX_{i, m},\sL_{i,m}\}_{i=1}^k\right),
	\end{align*}
	where we recall \cite[Definition 3.4]{BBJ} and note that the sum $\sum_{i=1}^k \varphi^{\mathrm{NA}}_{i,m}$ is represented by the sum test configuration $(\sX_{m},\sL_{m})$ of $\left\{(\sX_{i,m},\sL_{i,m})\right\}_{i=1}^k$ by Lemma \ref{lmsmnatc}. Thus, for any $\varepsilon >0$ there exists a $k$-tuple of test configurations $\left\{(\sX_{i,m},\sL_{i,m})\right\}_{i=1}^k$ such that
	\begin{equation*}
		\Ding\left(\{\sX_{i, m},\sL_{i, m}\}_{i=1}^k\right) < \varepsilon < \varepsilon \cdot \frac{2}{a} \mathrm{J}^{\mathrm{NA}}_{\mathrm{cp}, \T} (\varphi^{\mathrm{NA}}_{1,m} , \dots , \varphi^{\mathrm{NA}}_{k,m} ) = \frac{2 \varepsilon}{a} \JJ^{\cp}_{\T} \left(\{\sX_{i, m},\sL_{i, m}\}_{i=1}^k\right)
	\end{equation*}
	which contradicts the $\T$-reduced uniform coupled Ding stability.
\end{proof}

\begin{proof}[Proof of Theorem \ref{theorem:mainmc}]
It is an immediate consequence of Theorems \ref{theorem:apmnthm}, \ref{thsttcsv}, and Proposition \ref{ppkickedcptc}.
\end{proof}

	

\appendix

\section{On semiample divisors}\label{section:appendix}

In this section, we see basic properties of semiample divisors on normal projective 
varieties, and see an example that the sum of normal test configurations may not 
be normal. 

We begin with the following basic proposition: 

\begin{proposition}\label{proposition:sa-surj}
Let $X$ be a normal projective variety, $L_1,\dots,L_k$ be semiample $\Q$-divisors 
on $X$. Set $L:=\sum_{i=1}^k L_i$. Let $f_i\colon X\to Y_i$ 
(resp., $f\colon X\to Y$) be the ample model of $L_i$ (resp., $L$) 
in the sense of \cite[Definition 3.6.5]{BCHM}. It is well-known that there exists 
$g_i\colon Y\to Y_i$ such that $g_i\circ f=f_i$. 
Set 
\[
g:=\left(g_1,\dots,g_k\right)\colon Y\to Y_1\times\cdots\times Y_k. 
\]
Then the following are equivalent: 
\begin{enumerate}
\renewcommand{\theenumi}{\arabic{enumi}}
\renewcommand{\labelenumi}{(\theenumi)}
\item\label{proposition:sa-surj1}
The morphism $g$ is a closed embedding. 
\item\label{proposition:sa-surj2}
For any sufficiently divisible $m\in\Z_{>0}$, the multiplication homomorphism
\[
H^0\left(X, m L_1\right)\otimes_{\Bbbk}\cdots\otimes_{\Bbbk}
H^0\left(X, m L_k\right)\to H^0\left(X, m L\right)
\]
is surjective. 
\end{enumerate}
\end{proposition}

\begin{proof}
We may assume that $L$ is ample and $f=\id_X$. From the definition of 
ample models, there exist an ample $\Q$-divisors $A_i$ on $Y_i$ such that 
$g_i^*A_i\sim_\Q L_i$. We may assume that $A_i$ are line bundles with 
$g_i^*A_i\sim L_i$. 

\eqref{proposition:sa-surj1} $\Rightarrow$ \eqref{proposition:sa-surj2}: 
Let $I\subset\sO_{Y_1\times\cdots\times Y_k}$ be the ideal sheaf 
corresponds to the closed embedding $Y\subset Y_1\times\cdots\times Y_k$. 
Note that $A_1\boxtimes\cdots\boxtimes A_k$ on 
$Y_1\times\cdots\times Y_k$ is ample and 
\[
\left(A_1\boxtimes\cdots\boxtimes A_k\right)|_Y\simeq g_1^*A_1+\cdots+
g_k^*A_k=L 
\]
holds. By Serre's vanishing theorem, for any $m\gg 0$, we have 
\[
H^1\left(Y_1\times\cdots\times Y_k, I\otimes
\left(A_1\boxtimes\cdots\boxtimes A_k\right)^{\otimes m}\right)=0. 
\]
Thus 
\[
H^0\left(Y_1\times\cdots\times Y_k, (m A_1)
\boxtimes\cdots\boxtimes (m A_k)\right)\to H^0(Y, m L)
\]
is surjective. By the K\"unneth formula, we get the assertion 
\eqref{proposition:sa-surj2}. 

\eqref{proposition:sa-surj2} $\Rightarrow$ \eqref{proposition:sa-surj1}: 
For any sufficiently divisible $m\in \Z_{>0}$, from the assumption, there is a 
natural commutative diagram: 
\[\xymatrix{
&&&&& 
\pr^*H^0\left(Y, m L_1\right)\times\cdots\times\pr^*H^0\left(Y, m L_k\right)
\ar@{_{(}->}[d]^-{\text{Segre}}\\
Y \ar[rrrrru]^-{\left(\phi_{|m L_1|},\dots,\phi_{|m L_k|}\right)\quad}
\ar@{_{(}->}[rrrrrd]_-{\phi_{|m L|}}
&&&&& \pr^*\left(H^0\left(Y, m L_1\right)\otimes_{\Bbbk}\cdots\otimes_{\Bbbk}
H^0\left(Y, m L_k\right)\right)\\
&&&&& \pr^*H^0\left(Y, m L\right). \ar@{^{(}->}[u]_-{\text{linear}}
}\]
We note that the vertical linear embedding is given by the assumption 
\eqref{proposition:sa-surj2}. Since $m\in\Z_{>0}$ is sufficiently divisible, 
the morphism $\phi_{|m L_i|}$ is equal to $g_i$ and $\phi_{|m L|}$ is a 
closed embedding. 
This implies that $g$ is a closed embedding. 
\end{proof}

\begin{example}\label{example:counterHK}
Take any $k\geq 2$ and $n_1,\dots,n_k,d_1,\dots,d_k,e_1,\dots,e_k\in\Z_{>0}$. 
Let us consider any smooth 
\[
B\in\left|\sO_{\pr^{n_1}\times\cdots\times\pr^{n_k}}\left(
2d_1,\dots,2d_k\right)\right|,
\]
and let 
\[
\tau\colon X\to \pr^{n_1}\times\cdots\times\pr^{n_k}
\]
be the double cover branched along $B$. Set 
\[
A_i:=\sO_{\pr^{n_i}}(e_i), \quad M_i:=\tau^*p_i^* A_i, \quad
M:=\sum_{i=1}^k M_i, 
\]
where $p_i\colon\pr^{n_1}\times\cdots\times\pr^{n_k}\to\pr^{n_i}$ be the 
projection. Since $M_i$ is semiample and the morphism 
$p_i\circ\tau\colon X\to\pr^{n_i}$ satisfies that 
$\left(p_i\circ\tau\right)_*\sO_X=\sO_{\pr^{n_i}}$ and 
$\left(p_i\circ\tau\right)^*\sO_{\pr^{n_i}}(e_i)=M_i$, the morphism is the 
ample model of $M_i$. Moreover, $M$ is ample on $X$. However, since 
\[
\left(\left(p_1\circ\tau\right),\dots,\left(p_k\circ\tau\right)\right)=\tau
\colon X\to \pr^{n_1}\times\cdots\times\pr^{n_k} 
\]
is not a closed embedding, the homomorphism 
\[
H^0\left(X, m M_1\right)\otimes_{\Bbbk}\cdots\otimes_{\Bbbk}
H^0\left(X, m M_k\right)\to H^0\left(X, m M\right)
\]
is not surjective for any sufficiently divisible $m\in\Z_{>0}$ 
by Proposition \ref{proposition:sa-surj}. 
\end{example}

\begin{example}\label{example:counterHKbig}
Under the assumption in Example \ref{example:counterHK}, and assume moreover 
that $n_i+1\geq d_i$ for any $1\leq i\leq k$. Let us set 
\[
Y:=\pr_X\left(\sO_X\oplus M\right)\xrightarrow{p}X, 
\]
let $L_Y\in\Pic(Y)$ be the tautological line bundle with respects to the 
projectivization, and let $X_0:=\pr_X(\sO_X)$ be the section of $p$ corresponding to 
the natural projection 
\[
\sO_X\oplus M\to \sO_X\to 0.
\]
Of course, $X_0$ is canonically isomorphic to $X$. 
Set 
\[
H_i:=L_Y+p^*M_i\quad(1\leq i\leq k), \quad
H:=\sum_{i=1}^k H_i=k L_Y+p^*M.
\]
Since the evaluation homomorphism 
\[
H^0\left(X,\left(\sO_X\oplus M\right)\otimes M_i\right)\otimes_\Bbbk\sO_X\to
\left(\sO_X\oplus M\right)\otimes M_i
\]
is surjective, the divisor $H_i$ is base point free for any $1\leq i\leq k$. 
Let $\chi_i\colon Y\to Y_i$ be the ample model of $H_i$. 
An irreducible curve $C\subset Y$ is contracted by $\chi_i$ if and only if 
$C\subset X_0$ and $C$ is a fiber of the morphism 
$p_i\circ\tau\colon X\to\pr^{n_i}$ under the canonical identification $X_0\simeq X$. 
In particular, each $H_i$ is big and $H$ is ample. Since 
$n_i+1\geq d_i$ for any $1\leq i\leq k$, the divisor $-K_X$ is nef. 
This immediately implies that 
$-K_Y-X_0$ is $\chi_i$-nef and $\chi_i$-big for any $1\leq i\leq k$. 
By the Kawamata--Viehweg vanishing theorem, we have 
$R^1(\chi_i)_*\sO_Y(-X_0)=0$. Thus we get the surjection 
\[
(\chi_i)_*\sO_Y\twoheadrightarrow(\chi_i|_{X_0})_*\sO_{X_0}.
\]
Since the left hand side is nothing but $\sO_{Y_i}$, we have 
\[
(\chi_i|_{X_0})_*\sO_{X_0}=\sO_{\chi_i(X_0)}. 
\]
In particular, we have $\chi_i(X_0)\simeq \pr^{n_i}$ and the morphism 
$\chi_i|_{X_0}$ is isomorphic to $p_i\circ\tau$. Thus the restriction of 
\[
\chi:=\left(\chi_1,\dots,\chi_k\right)\colon Y\to Y_1\times\cdots\times Y_k
\]
to $X_0$ is isomorphic to $\tau$, which is not an embedding. 
In particular, the morphism $\chi$ is not an embedding. 
\end{example}

\begin{remark}\label{remark:counterHK}
Therefore, the surjectivity assertion in \cite[Lemma 2.8]{HK} is not true 
in general, even when all of $L_i$ are big and semiample. 
\end{remark}

\begin{example}\label{example:non-normal-sum}
Under the assumption in Example \ref{example:counterHK}, and assume moreover 
that $n_i>d_i$ for any $1\leq i\leq k$. From now on, just for simplicity, we set 
\[
\sO_X\left(m_1,\dots,m_k\right)
:=\tau^*\sO_{\pr^{n_1}\times\cdots\times\pr^{n_k}}
\left(m_1,\dots,m_k\right)
\]
for any $m_1,\dots,m_k\in\Z$. Note that 
\[
-K_X\sim\sO_X\left(n_1+1-d_1,\dots,n_k+1-d_k\right).
\]
Take any $c_1,\dots,c_k\in\Z_{>0}$ with $n_i-d_i\geq c_i$ for any $1\leq i\leq k$, 
and let us take a general smooth member 
\[
C\in\left|\sO_X\left(c_1,\dots,c_k\right)\right|.
\]
We set $\sO_C\left(m_1,\dots,m_k\right):=\sO_X\left(m_1,\dots,m_k\right)|_C$ 
for any $m_1,\dots,m_k\in\Z$. Let us take the blowup 
\[
\sigma\colon\sX\to X_{\A^1}
\]
along $C\times\{0\}$, let $E\subset\sX$ be the exceptional divisor, and 
let $\hat{X}_0\subset\sX$ be the strict transform of the prime divisor 
$X\times\{0\}\subset X_{\A^1}$. 
Since
\[
\sN_{C\times\{0\}/X_{\A^1}}\simeq\sO_C\oplus\sO_C\left(c_1,\dots,c_k\right), 
\]
we get 
\[
E\simeq\pr_C\left(\sO_C\oplus\sO_C\left(-c_1,\dots,-c_k\right)\right)
\xrightarrow{\pi_E}C,
\]
where $\pi_E$ is the natural projection. Let $L_E\in \Pic(E)$ be the 
tautological line bundle of the projective bundle under the above isomorphism. 
Then we have 
\[
-E|_E\sim L_E\sim\hat{X}_0|_E, \quad
-K_E\sim 2L_E+\pi_E^*\sO_C\left(n_1+1-d_1,\dots,n_k+1-d_k\right).
\]
Moreover, we have 
\[
\hat{X}_0|_{\hat{X}_0}\sim -E|_{\hat{X}_0}\sim
\sO_X\left(-c_1,\dots,-c_k\right)
\]
under the canonical isomorphism $\hat{X}_0\simeq X$. 

For any $1\leq i\leq k$, let us set 
\begin{eqnarray*}
L_i&:=&\sO_X\left(c_1,\dots,c_k\right)+M_i, \\
\sL_i^{\sX}&:=&\sigma^*(L_i)_{\A^1}-E, 
\end{eqnarray*}
where $M_i$ is as in Example \ref{example:counterHK}. 
Observe that 
\begin{itemize}
\item
$\sL_i^{\sX}|_{\hat{X}_0}\sim L_i-E|_{\hat{X}_0}\sim M_i$ is nef, and
\item
$\sL_i^{\sX}|_E\sim\pi_E^*\left(L_i|_C\right)+L_E\sim
\pi_E^*\left(L_i|_C\right)+\hat{X}_0|_E$ is also nef. 
Moreover, any curve $\gamma \subset E$ with 
$\left(\sL_i^{\sX}\cdot\gamma\right)=0$ satisfies that $\gamma\subset\hat{X}_0$. 
\end{itemize}
In particular, $\sL_i^{\sX}$ is nef over $\A^1$. Moreover, 
\begin{itemize}
\item
$\left(\sL_i^{\sX}-K_{\sX}\right)|_{\hat{X}_0}\sim M_i
+\sO_X\left(n_1+1-d_1-c_1,\dots,n_k+1-d_k-c_k\right)$
is ample, and 
\item
$\left(\sL_i^{\sX}-K_{\sX}\right)|_E\sim 2L_E+\pi_E^*\left(M_i|_C+
\sO_C\left(n_1+1-d_1+c_1,\dots,n_k+1-d_k+c_k\right)\right)$
is also ample. 
\end{itemize}
In particular, $\sL_i^{\sX}-K_{\sX}$ is ample over $\A^1$. By the base point free 
theorem, $\sL_i^{\sX}$ is semiample over $\A^1$. Let 
\[
\left(\sX,\sL_i^{\sX}\right)\xrightarrow{\sigma_i}
\left(\sX_i,\sL_i\right)
\xrightarrow{\pi_i}\A^1
\]
be the ample model over $\A^1$. The $\left(\sX_i,\sL_i\right)/\A^1$ can 
naturally be seen as a normal test configuration of $(X, L_i)$. 
As we already observed, the morphism $\sigma_i$ is an isomorphism on 
$\sX\setminus\hat{X}_0$, and the set of curves contracted by $\sigma_i$ is 
equal to the set of curves contracted by the morphism 
$p_i\circ\tau\colon X\to \pr^{n_i}$ under the canonical isomorphism 
$\hat{X}_0\simeq X$. 

Note that $\left(-K_{\sX}-\hat{X}_0\right)|_{\hat{X}_0}\sim -K_X$ is ample. 
Thus $-K_{\sX}-\hat{X}_0$ is nef and big over $\sX_i$. 
By the Kawamata--Viehweg vanishing theorem, we have 
$R^1(\sigma_i)_*\sO_{\sX}\left(-\hat{X}_0\right)=0$. Thus we get 
\[
\left(\sigma_i|_{\hat{X}_0}\right)_*\sO_{\hat{X}_0}=\sO_{\sigma_i\left(\hat{X}_0\right)}.
\]
In particular, we have $\sigma_i\left(\hat{X}_0\right)\simeq\pr^{n_i}$ and 
the morphism $\sigma_i|_{\hat{X}_0}$ is isomorphic to $p_i\circ\tau$. 

Let us set 
\begin{eqnarray*}
L&:=&\sum_{i=1}^k L_i=\sO_X\left(k c_1,\dots, k c_k\right)+\sum_{i=1}^k M_i, \\
\sL&:=&\sum_{i=1}^k \sL_i=\sigma^*L_{\A^1}-k E. 
\end{eqnarray*}
Since $\sL|_{\hat{X}_0}\sim\sum_{i=1}^k M_i$ and $\sL|_E\sim k\xi+\pi_E^*(L|_C)$, 
the divisor $\sL$ is ample over $\A^1$. In particular, by 
Lemma \ref{lemma:sum-tc} \eqref{lemma:sum-tc3}, the normal test configuration 
$\left(\sX,\sL\right)/\A^1$ is the normalized sum configuration of 
$\left\{\left(\sX_i,\sL_i\right)\right\}_{i=1}^k$. Let 
$\left(\sX',\sL'\right)/\A^1$ be the sum configuration of 
$\left\{\left(\sX_i,\sL_i\right)\right\}_{i=1}^k$. Then there is a morphism 
$\nu\colon \sX\to\sX'$ such that $\sL=\nu^*\sL'$ holds. 
As in Lemma \ref{lemma:sum-tc} \eqref{lemma:sum-tc4}, 
for any $1\leq i\leq k$, we have the following natural commutative diagram: 
\[\xymatrix{
\sX\ar[r]^\nu \ar[rrrd]_{\sigma_i} & \sX' \ar@{^{(}->}[r] & 
\sX_{\PROD}\ar@{}[r]|-{=} & 
\sX_1\times_{\A^1}\cdots\times_{\A^1}\sX_k\ar[d]^{q_i}\\
&&&\sX_i.
}\]
Let $\sX_{i,0}$ (resp., $\sX_0$) be the fiber of $\sX_i\to\A^1$ (resp., 
$\sX\to\A^1$) over $0\in\A^1$. The fiber of the above diagram 
over $0\in\A^1$ gives
\[\xymatrix{
\hat{X}_0 \ar@{}[r]|{\subset} \ar[rrrd]_{(\sigma_i)|_{\hat{X}_0}}
&\sX_0 \ar[r] & \sX_0' \ar@{^{(}->}[r] & 
\sX_{1,0}\times\cdots\times\sX_{k,0}\ar[d]^{q_i}\\
&&&\sX_{i,0}.
}\]
If $\nu$ is an isomorphism, then we get the closed embedding 
\[
\hat{X}_0\hookrightarrow\sX_{1,0}\times\cdots\times\sX_{k,0}.
\]
However, since the morphism $\sigma_i|_{\hat{X}_0}$ is isomorphic to $p_i\circ\tau$ 
for any $1\leq i\leq k$, the above embedding must be isomorphic to 
$\tau\colon X\to \pr^{n_1}\times\cdots\times\pr^{n_k}$. 
Clearly, the $\tau$ is not an embedding, a contradiction. Thus $\nu$ is 
not an isomorphism. In particular, 
the sum configuration of $\left\{\left(\sX_i,\sL_i\right)\right\}_{i=1}^k$ is 
not normal. 
\end{example}

\end{document}